\documentclass[11pt,english]{article}
\usepackage[T1]{fontenc}

\usepackage{graphicx, subfigure}
\expandafter\def\csname ver@subfig.sty\endcsname{}
\usepackage[latin9]{inputenc}
\usepackage{geometry}
\usepackage{amsmath}
\usepackage{appendix}
\usepackage{amssymb}
\usepackage{esint}
\usepackage{makecell,multirow,diagbox}
\usepackage{mathtools}
\usepackage{epstopdf}
\usepackage[latin9]{inputenc}
\usepackage{pict2e}
\usepackage{stmaryrd}
\usepackage{esint}
\usepackage[all,cmtip]{xy}
\usepackage{mathtools}
\usepackage{float}
\usepackage{setspace}
\usepackage{authblk}
\usepackage{amsthm}
\usepackage{mathrsfs}
\usepackage{stmaryrd}
\usepackage{bbold}
\usepackage{color}
\usepackage{tikz}
\usepackage[all,cmtip]{xy}
\usepackage{algorithm}  
\usepackage{algorithmic}

\theoremstyle{plain}\newtheorem{theorem}{Theorem}[section]
\theoremstyle{plain}\newtheorem{lemma}{Lemma}[section]
\newtheorem{pro}{Proposition}
\newtheorem{remark}{Remark}
\usepackage{color}
\definecolor{marin}{rgb} {0., 0.3, 0.7}
\definecolor{rouge}{rgb} {0.8, 0., 0.}
\definecolor{sepia}{rgb} {0.8, 0.5, 0.}
\usepackage[colorlinks,citecolor=sepia,linkcolor=marin,bookmarksopen,bookmarksnumbered]{hyperref}

\theoremstyle{definition}
\newtheorem{definition}{Definition}

\newcommand*{\transp}[2][-3mu]{\ensuremath{\mskip1mu\prescript{\smash{\mathrm t\mkern#1}}{}{\mathstrut#2}}}    

\DeclareSymbolFont{largesymbol}{OMX}{yhex}{m}{n}
\DeclareMathAccent{\Widehat}{\mathord}{largesymbol}{"62}

\makeatletter

\usepackage{babel}
\usepackage{subfig}

\makeatother
\DeclareMathSizes{14}{14}{9.8}{7}

\usepackage{lipsum}

\newcommand\THANK[1]{%
  \begingroup
  \renewcommand\thefootnote{}\footnote{#1}%
  \addtocounter{footnote}{-1}%
  \endgroup
}

\begin{document}

\title{\textbf{Exact splitting methods for kinetic and Schr\"odinger equations}}
\date{}

\author[1]{\small \textbf{Joackim Bernier}}
\author[2]{\small \textbf{Nicolas Crouseilles}}
\author[3]{\small \textbf{Yingzhe Li}}
\affil[1]{Institut  de  Math\'ematiques  de  Toulouse  ;  UMR5219,  Universit\'e  de  Toulouse  ;  CNRS,  Universit\'e  Paul Sabatier, F-31062 Toulouse Cedex 9, France }
\affil[2]{Univ Rennes, INRIA, CNRS, IRMAR - UMR 6625, F-35042 Rennes, France}
\affil[3]{ University of Chinese Academy of Sciences,  Beijing 100049; LSEC, Academy of Mathematics and Systems Science, Chinese Academy of Sciences, Beijing 100190,  CHINA; Inria (MINGuS team), France}

\maketitle

\vspace{-1cm}

\THANK{The author J.B. was supported by the French National Research Agency project NABUCO, grant ANR-17-CE40-0025. The author Y.L. is supported by a  scholarship from Academy of Mathematics and Systems Science, Chinese Academy of Sciences.}

\begin{abstract}
In \cite{essiqo}, some exact splittings are proposed for inhomogeneous quadratic differential equations including, for example, transport equations, kinetic equations, and Schr\"odinger type equations with a rotation term.
In this work, these exact splittings are combined with pseudo-spectral methods in space to illustrate 
their high accuracy and efficiency.
\end{abstract}

\setcounter{tocdepth}{1} 

\tableofcontents

\section{Introduction}
Operator splitting methods  have gained a lot of attention in recent years to solve numerically partial differential equations, 
as the subsystems obtained are usually easier to solve or even can be solved exactly, which allows a keen reduction of the computational cost and the derivation of high order time integrators.  For a general introduction to splitting methods, we refer to~\cite{MQ, HO} and references therein. To obtain high order splitting methods, 
usually several subsystems are needed to be solved, 
and proper regularity conditions about the original system must be assumed. However, there exist some systems for which splitting methods can give exact solutions indeed, such as in~\cite{JC2, LH, Ameres}. 

Generally, exact splittings is one kind of splitting methods that give exact solutions for the original systems. 
However, exact splitting are generally available for very simple cases, for which the operators involved commute. 
In~\cite{essiqo}, exact splittings are obtained for a large class of PDEs, namely inhomogeneous quadratic 
differential equations (see definition~\ref{def_comput} below). In this framework, each subsystem can be solved 
accurately and efficiently by pseudo-spectral method or pointwise multiplication. 

In this work, our goal is to illustrate numerically the efficiency of these exact splitting methods which have been proposed for inhomogeneous quadratic differential equations in~\cite{essiqo}. First, we will focus on high dimensional transport equations 
for which efficient exact splittings can be derived from the splitting of the underlying linear ordinary differential equation.  
Second, we will see that exact splitting can be obtained for Fokker-Planck type equations and last but not least, 
several applications are proposed in the case of Schr\"odinger type equations. In this case, 
the derivation of exact splittings is based on 
the Weyl quantization and H\"ormander theory~\cite{Lars}, which reduces the infinite dimensional system 
to finite dimensional system. 
Note that even if exact splittings are applied on inhomogeneous quadratic differential equations, they 
can be used to derive new efficient methods for non-quadratic equations by using composition techniques  
such as Strang splitting for instance. Indeed, the equation can be simply split 
into the quadratic part and the non-quadratic part.

The exact splittings are not only important but also useful for the time integration of PDEs, 
they can be also of great interest at the theoretical level since they can reduce the original complicated evolution equation 
into several simpler operators, which gives a way to analyze the properties for the original system (see \cite{AB}). 
On the numerical side, since the exact splittings we propose can be combined with highly accuracy space discretization 
methods, the resulting fully discretized methods are very accurate and turn out to be very useful to study 
the long time behavior of the original system. We also compare the efficiency of our methods to 
high order splitting methods from the literature and illustrate that in the examples we consider, the exact splitting methods 
are more efficient and accurate.

In this work, exact splittings (and its non-quadratic extensions) are used to simulate transport, kinetic, and Schr\"odinger type equations. After recalling some basic tools introduced and proved in \cite{essiqo}, we focus on the 
numerical performances of the exact splitting in different applications. For transport equations, 
we consider high dimensional systems (dimension $3$ and $4$) and compare with standard methods from the literature, 
namely operator splitting method and direct semi-Lagrangian method (combined with NUFFT interpolation). 
Then, we consider the Fokker-Planck type equations and show that the exact splittings are able to recover 
the property that its solution converges to equilibrium exponentially fast for Fokker-Planck equation, 
and the regularizing effects of Kramer-Fokker-Planck equation.
Lastly, Schr\"odinger type equations are studied numerically in dimension $2$ and $3$. More precisely, 
we consider the magnetic Schr\"odinger equation with quadratic potentials (see \cite{Jin, Ostermann}) and Gross-Pitaevskii equation with one rotation term (see \cite{Bao, wang, review_3}). When non-quadratic terms are considered in these models 
(non quadratic potential or nonlinear terms for instance), it is worth mentioning that the new splittings proposed here 
 give higher accuracy, in particular when the amplitude of non-quadratic terms are small.

\section{Exact splittings}
\label{section1}
In this section, we introduce exact splittings for three kinds of inhomogeneous quadratic differential equations: transport, quadratic Schr\"odinger, and Fokker-Planck equations, which is studied theoretically in \cite{essiqo}. 
We start by introducing what we mean by  inhomogeneous quadratic equations and  exact splitting.

Inhomogeneous quadratic partial differential equations can be written as 
\begin{equation}
\label{linear_pde}
\left\{ \begin{array}{cccl} \partial_t u(t,{\mathbf x}) & =& - p^w u(t,{\mathbf x}), & t\geq 0,\ {\mathbf x}\in \mathbb{R}^n \\ 
				u(0,{\mathbf x}) &=& u_0({\mathbf x}), & {\mathbf x}\in \mathbb{R}^n
\end{array} \right.
\end{equation}
where $n\geq 1$, $u_0\in L^2(\mathbb{R}^n)$ and $p^w$ is an inhomogeneous quadratic differential operator acting on $L^2(\mathbb{R}^n)$.  When the solution at time $t$ of this equation is well defined, it is denoted, as usual, by $e^{-t p^w}u_0$.  This operator $p^w$ is defined through an oscillatory  integral involving a polynomial function 
(called the symbol) $p$ on $\mathbb{C}^{2n}$ of degree $2$. In this context, one can write $p$ as 
\begin{equation}
\label{pX}
p(X)=\transp{X} Q X + \transp{Y}X +c, 
\end{equation}
where $X= (\transp{\mathbf x}, \transp{\boldsymbol  \xi})  = \transp{(x_1,\dots,x_n,\xi_1,\dots,\xi_n)}$, $Q$ is a symmetric matrix of size $2n$ with complex coefficients, 
$Y\in \mathbb{C}^{2n}$ is a vector and $c\in \mathbb{C}$ is a constant.  The associated 
differential operator $p^w$ then writes 
$$
p^w = \transp{\begin{pmatrix} {\mathbf x} \\ -i\nabla \end{pmatrix}}Q\begin{pmatrix} {\mathbf x} \\ -i\nabla \end{pmatrix} + \transp{Y} \begin{pmatrix} {\mathbf x} \\ -i\nabla \end{pmatrix} + c, 
$$

For $(-p^w)$ whose real part is bounded by below on $\mathbb{R}^{2n}$, it generates a strongly continuous semigroup on $L^2(\mathbb{R}^n)$~\cite{Lars}.  In~\cite{essiqo}, one of the authors proved that that $e^{-p^w}$ can be split {\it exactly} 
into simple semigroups. As we shall see below, there are several examples which enter in this framework 
and for which the solution can be split into operators which are easy to compute. In the following definition, 
we define what we mean by exact splitting in this work. 

\begin{definition}
\label{def_comput} 
An operator acting on $L^2(\mathbb{R}^n)$ can be \emph{computed by an exact splitting} if it can be factorized as a product of operators of the form 
\begin{equation}
\label{eq:factor}
e^{ \alpha \partial_{x_j}},\ e^{ i \alpha x_j },\ e^{ i a(\nabla)},\ e^{ i a({\mathbf x}) },\ e^{ \alpha x_k  \partial_{x_j} },\ e^{ -b({\mathbf x})},\ e^{ b(\nabla)},\ e^{\gamma}
\end{equation} 
with $\alpha\in \mathbb{R},\gamma\in \mathbb{C},a,b : \mathbb{R}^n \to \mathbb{R}$ are some real quadratic forms, $b$ is nonnegative and $j,k\in \llbracket 1,n \rrbracket$ and $k\neq j$. As usual, $a(\nabla)$ (resp. $b(\nabla)$) denotes the Fourier multiplier associated with $-a(\xi)$ (resp. $-b(\xi)$), i.e. $a(\nabla) = (-a(\xi))^w$. 
\end{definition}

%

Hence, from Definition~\ref{def_comput} exact splittings mean that every subsystem in (\ref{eq:factor})  
can be solved exactly in time at least in Fourier variables and as such 
can be solved efficiently and accurately by pseudo-spectral methods or pointwise multiplications.  
The resulting fully discretized method will benefit from the spectral accuracy in space 
so that the error will be negligible in practice. 

Below we detail the way we compute the solutions of \eqref{linear_pde} using pseudo-spectral methods. First, note that, being given a factorization of an operator as a product of elementary operators of the form \eqref{eq:factor}, there is a natural and minimal factorization of this operator as product of partial Fourier transforms, inverse partial Fourier transforms and multipliers (i.e. operators associated with a multiplication by a function). So, as usual, we just discretize the partial Fourier transforms, their inverts and the multipliers.

In order to get an approximation of the solution on a large box $[-R_1,R_1]\times \dots \times [-R_n,R_n]$, we discretize the box as a product of grids $\mathbb{G}_1 \times \dots \times \mathbb{G}_n$ where each grid $\mathbb{G}_j$ has $N_j$ points and is of the form
\begin{equation}
\label{grid}
 \mathbb{G}_j = h_j \left\llbracket - \left\lfloor \frac{N_j-1}2 \right\rfloor ,\left\lfloor \frac{N_j}2 \right\rfloor  \right\rrbracket
\end{equation}
where $h_j = 2R_j/N_j$ is its step-size. Associated with such a grid, there is its dual, denoted  $\widehat{\mathbb{G}_j}$ and defined by
$$
\widehat{\mathbb{G}_j} = \eta_j \left\llbracket - \left\lfloor \frac{N_j-1}2 \right\rfloor ,\left\lfloor \frac{N_j}2 \right\rfloor  \right\rrbracket
$$
where $\eta_j = \pi / R_j$ is its step-size. In this paper, the variable implicitly naturally associated with $\mathbb{G}_j$ (resp. $\widehat{\mathbb{G}_j}$) is denoted $g_j$ (resp. $\omega_j$). 

If $\mathcal{L}$ is a product of $j-1$ grids (and duals of grids) and $\mathcal{R}$ is a product of grids (and duals of grids) then the \emph{discrete $j^{\mathrm{st}}$ partial Fourier transform} on $\mathcal{L} \times  \mathbb{G}_j \times \mathcal{R}$ is defined by
$$
 \mathcal{F}_j : \left\{ \!\!\!\!\begin{array}{ccc} \mathcal{L} \times  \mathbb{G}_j \times \mathcal{R} &\to& \mathbb{C}^{ \mathcal{L} \times  \widehat{\mathbb{G}_j} \times \mathcal{R}} \\
\psi & \mapsto &  \big( \displaystyle h_j \sum_{g_j \in \mathbb{G}_j} \psi_{r,g_j, \ell} e^{-i g_j\omega_j } \big)_{ (r,\omega_j, \ell) }. \end{array} \right. 
$$
The discrete partial inverse Fourier transforms are defined similarly and are the inverses of the discrete inverse Fourier transforms
$$
 \mathcal{F}_j^{-1} : \left\{ \!\!\!\!\begin{array}{ccc} \mathcal{L} \times  \widehat{\mathbb{G}_j} \times \mathcal{R} &\to& \mathbb{C}^{ \mathcal{L} \times  \mathbb{G}_j \times \mathcal{R}} \\
\psi & \mapsto &  \big( \displaystyle \frac{\eta_j}{2\pi} \sum_{\omega_j \in \widehat{\mathbb{G}_j}} \psi_{r,\omega_j, \ell} e^{i g_j\omega_j } \big)_{ (r,g_j, \ell) }. \end{array} \right. 
$$
Note that these discrete transforms can be computed efficiently using Fast Fourier Transforms. Finally, the multipliers are naturally discretized through pointwise multiplications. An explicit example is provided in Algorithm \ref{algo1} for Schr\"odinger equations.

\section{Application to transport equations}
\label{transport}
In this section, we introduce the exact splittings for constant coefficients transport equations, which is one kind of quadratic equations.
The transport equation we consider here is 
\begin{align}
&\partial_t f({\bf x}, t) =  ({ M}  {\bf x})\cdot \nabla f({\bf x},t), \quad {\mathbf x} \in {\mathbb R}^n, \ n\geq 1, \ f({\mathbf x}, t=0) = f_0({\mathbf x}), \label{eq:rot}
\end{align}
where $M$ is a real square matrix of size $n\geq 1$ such that
\begin{equation}
\label{bobof_prop}
\left\{ \begin{array}{lll} \forall i, \ &M_{i,i} = 0, \\
				\exists i,\forall j  \neq i , & M_{j,i} \neq 0,
\end{array} \right.
\end{equation} 
and the corresponding symbol of (\ref{eq:rot}) is $p(X) = -i (M {\mathbf x}) \cdot {\boldsymbol \xi}$ according to the notations 
\eqref{pX}.

Even if the solution of (\ref{eq:rot}) can be computed from the initial condition as $f({\mathbf x}, t) = f_0(e^{tM}{\mathbf x})$, 
efficient numerical methods are required when the initial data is only know on a mesh or when (\ref{eq:rot})  
is a part of a more complex model.  Below, we start by giving some details of the time (exact) splitting 
before illustrating the efficiency of the strategy with numerical results. 
 
\subsection{Presentation of the exact splitting} 
In this part, we construct an exact splitting for \eqref{eq:rot}. 
Let us start with a simple example for $n=2$ and $M = \begin{pmatrix} 0 & 1 \\ -1 & 0 \end{pmatrix}$, $e^{tM}$ becomes a two dimensional rotation matrix, 
which can be expressed as the product of three shear matrices (see \cite{JC2}) 
\begin{equation}\label{eq:shear}
e^{tM}  = 
\begin{pmatrix} 1&\tan(\theta/2) \\
0 & 1
 \end{pmatrix}
 \begin{pmatrix} 1&0 \\
-\sin{\theta} & 1
 \end{pmatrix}
 \begin{pmatrix} 1&\tan(\theta/2) \\
0 & 1
 \end{pmatrix}.
\end{equation}
As a consequence,  the computation of $f$ can be done by solving three one dimensional linear equations (in $x_1$, $x_2$, and $x_1$ directions successively), i.e., 
\begin{equation}
\begin{aligned}
&f_0({\bf x}) \stackrel{\tan(\theta/2)}{\longrightarrow} f_0( x_1 + \tan(\theta/2)x_2, x_2) \stackrel{-\sin(\theta)}{\longrightarrow} f_0(x_1+ \tan(\theta/2)(x_2-\sin(\theta)x_1), x_2-\sin(\theta)x_1)\\
& \stackrel{\tan(\theta/2)}{\longrightarrow} 
 f({\bf x}, t).
\end{aligned}
\end{equation}
Formula~(\ref{eq:shear}) has been used in the computation of Vlasov--Maxwell equations to improve efficiency and accuracy by avoiding high dimensional reconstruction in~\cite{JC2, Ameres}. 

For the case $n=3$ and $M$ is skew symmetric, similar formula of expressing the rotation 
matrix as the product of 4 shear matrices is proposed in~\cite{shear, constants}.
To generalize this formula to arbitrary dimension, we have the following results proved in~\cite{essiqo}.
 \begin{pro}
 \label{prop_rot_gen} Let $M$ be a real square matrix of size $n\geq 1$ satisfying condition (\ref{bobof_prop}), then there exist $t_0>0$ and an analytic function $(y^{(\ell)},(y^{(k)})_{k=1, \dots, n ; k\neq i},y^{(r)}):(-t_0,t_0)\to \mathbb{R}^{n\times (n+1)}$ satisfying
\begin{equation}
\label{penible_prop}
\left\{ \begin{array}{lll} y^{(\ell)}_i = y^{(r)}_i =0 \\
				 \forall k\neq i, \ y^{(k)}_k = 0
\end{array} \right.
\end{equation} 
such that for all $t\in  (-t_0,t_0)$ we have
\begin{equation}\label{eq:esr}
e^{t M{\mathbf x}\cdot \nabla} = e^{t (y^{(\ell)}(t)\cdot {\mathbf x})\partial_{x_i}} \left(  \prod_{ k\neq i} e^{t (y^{(k)}(t)\cdot {\mathbf x})\partial_{x_k}}  \right)  e^{t (y^{(r)}(t)\cdot {\mathbf x})\partial_{x_i}}.
\end{equation}
\end{pro}
%

\begin{remark}
Proposition \ref{prop_rot_gen} not only enables to recover some results from the literature 
(in particular when $M$ is skew symmetric) but it also claims that $n$ dimensional linear equations of the form \eqref{eq:rot} 
can be split into $(n+1)$ one dimensional linear equations which can be solved very efficiently by means of pseudo-spectral 
methods or semi-Lagrangian methods.  In particular, this turns out to be much more efficient (only in terms of efficiency) 
than standard Strang splitting which would require $2(n-1)+1$ linear equations to solve. Let us also recall that Strang splitting 
produces second order error terms whereas the splitting proposed in Proposition \ref{prop_rot_gen} are exact in time.  
\end{remark}

\begin{remark}
Another alternative to solve \eqref{eq:rot} would be the direct $n$-dimensional semi-Lagrangian method. 
However, this approach requires a huge complexity at the interpolation stage since high-dimensional algorithms 
are known to be very costly. 
\end{remark}

\subsection{Numerical results}
For the transport equation, exact splitting are used to solve 3D and 4D transport equations, and compared with the usual Strang splitting and Semi-Lagrangian method combined with NUFFT in space.  We then are interested in the numerical approximation 
of 
\begin{equation}
\label{eq:nDtrans}
\partial_t{ f({\mathbf x}, t)}=  (M{\mathbf x})\cdot \nabla f({\mathbf x}, t), \quad f({\mathbf x}, t=0) = f_0({\mathbf x}), \quad {\mathbf x}\in \mathbb{R}^n, 
\end{equation}
for $n=3, 4$. For numerical reasons, the domain will be truncated to $x\in [-R, R]^n$ and we will consider $N$ points 
per direction so that the mesh size is $h=2R/N$. The grid, defined as usual through \eqref{grid}, is denoted $\mathbb{G}^n$.  We shall denote by $f^n_g$ an approximation of 
$f(n\Delta t, g)=f_0((e^{n\Delta t M}g))$ the exact solution of \eqref{eq:nDtrans} with  $g\in \mathbb{G}^n$ 
and $\Delta t>0$ the time step. We also define the $L^2$ 
error between the numerical solution and the exact one as 
\begin{equation}
 \sqrt{h^n \sum_{g\in \mathbb{G}^n} |f^n_g - f(n\Delta t,g)|^2}
\label{error_L2}
\end{equation}

\noindent{\bf 3D transport equation}\\
We consider \eqref{eq:nDtrans} in the case $n=3$ with 
$$
M = \begin{pmatrix} 0 & -0.36 & -0.679\\
 0.36 & 0 & -0.758\\
 0.679 & 0.758 & 0 
 \end{pmatrix}. 
 $$ 
The initial value is chosen as follows 
 $$
f_0({\bf x}) = \frac{1}{2(\pi \beta)^2}\Big(e^{-(x_1-0.3)^2/\beta}+ e^{-(x_1+0.3)^2/\beta}\Big)e^{-x_2^2/\beta} e^{-x_3^2/\beta}, 
$$
with  $\beta = 0.06$. 
The following three numerical methods are used to solve the three dimensional transport equation 
\begin{itemize}
\item \text{NUFFT}: direct 3D Semi-Lagrangian method combined with interpolation by NUFFT; this method is exact in time. 
\item \text{Strang}: Strang directional splitting method combined with Fourier pseudo-spectral method; this method is second order accurate in time. 
\item \text{ESR}: Exact splitting (\ref{eq:esr}) combined with Fourier pseudo-spectral method; this method is exact in time.
\end{itemize}
Let us detail the coefficients used for ESR. From Prop~\ref{prop_rot_gen}, we have 
$$
e^{\Delta t{M}{\mathbf x}\cdot \nabla} = e^{\Delta t(y^{(\ell)}\cdot {\mathbf x})\partial_{x_3}}  e^{\Delta t(y^{(2)}\cdot {\mathbf x})\partial_{x_1}} e^{\Delta t(y^{(3)}\cdot {\mathbf x})\partial_{x_2}} e^{\Delta t(y^{(r)}\cdot {\mathbf x})\partial_{x_3}},  
$$
where the coefficients are as follows when $\Delta t = 0.3$
$$
y^{(\ell)} \simeq \begin{pmatrix} 
0.345224363827786 \\
 0.379204563977292 \\
0
 \end{pmatrix}, \;\; \;\; 
y^{(2)} \simeq \begin{pmatrix} 
  0 \\
 -0.036460351430518 \\
 -0.664426864374562
 \end{pmatrix},
 $$
 and 
$$
y^{(3)} \simeq \begin{pmatrix} 
0.036504386840795 \\
 0 \\
 -0.742627150015417
 \end{pmatrix}, 
\;\; \;\; 
y^{(r)} \simeq \begin{pmatrix} 
 0.339075826535304 \\
0.384712290654848 \\
0
 \end{pmatrix}.
 $$

First, the time evolution of $L^2$ error (defined by  \eqref{error_L2}) in semi-$\log$ scale is plotted in Figure \ref{fig:error_toy} 
for Strang and ESR for $N=64$ and $\Delta t=0.3$. As expected, we observe that the error from 
ESR is close to the level of machine precision whereas the error from Strang is much larger. 
We can also see that the error from Strang has a almost periodic behavior (similar to what has been observed in 2D in \cite{JC2}) 
that deserves some further analysis in a future work. 
We can see that the error from Strang is increasing with time whereas the error for ESR remains close to $10^{-11}$. 
We also compare in Figure \ref{fig:error_toy} the CPU time of the two methods 
and the NUFFT method (which also gives error close to machine precision) 
by running them on $100$ steps in $\log-\log$ scale. 
We can observe that ESR is the most efficient. 
Indeed, for each time step, $5$ one dimensional transport equations are needed for Strang splitting, whereas 
ESR only has $4$ one dimensional transport equations to solve. Moreover, the NUFFT method is the most expensive 
method. 
Even if NUFFT and ESR have the same complexity ${\cal O}(N^3 \log(N))$, ESR is clearly cheaper 
which means that it involves a smaller constant. Moreover, let us mention that  parallelization 
can be developed to improve the efficiency of splitting methods like ESR (see \cite{JC2, coulaud}).


\begin{figure}[htbp]
\center{
\subfigure[]{\includegraphics[scale=0.35]{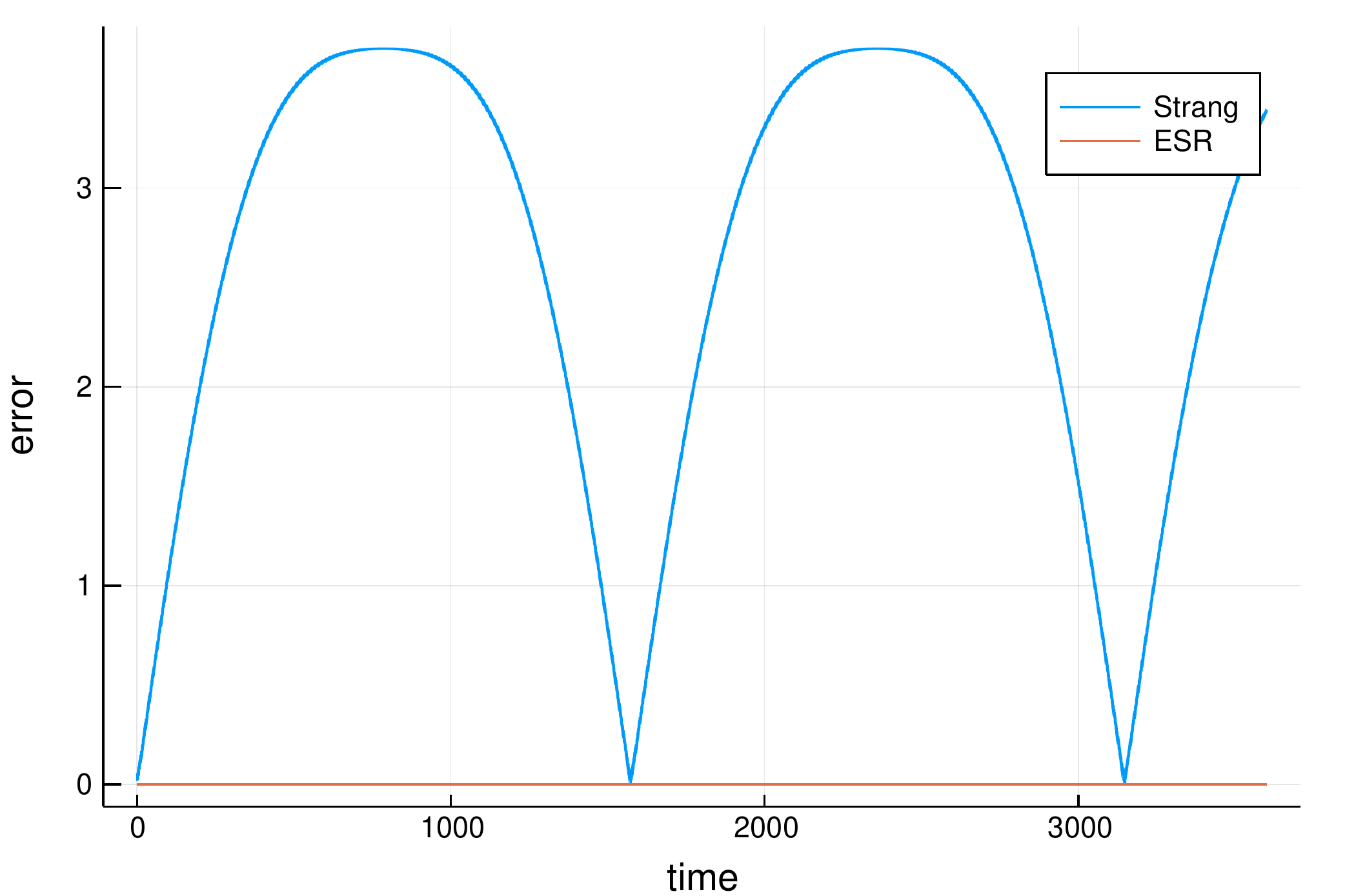}}
\subfigure[]{\includegraphics[scale=0.35]{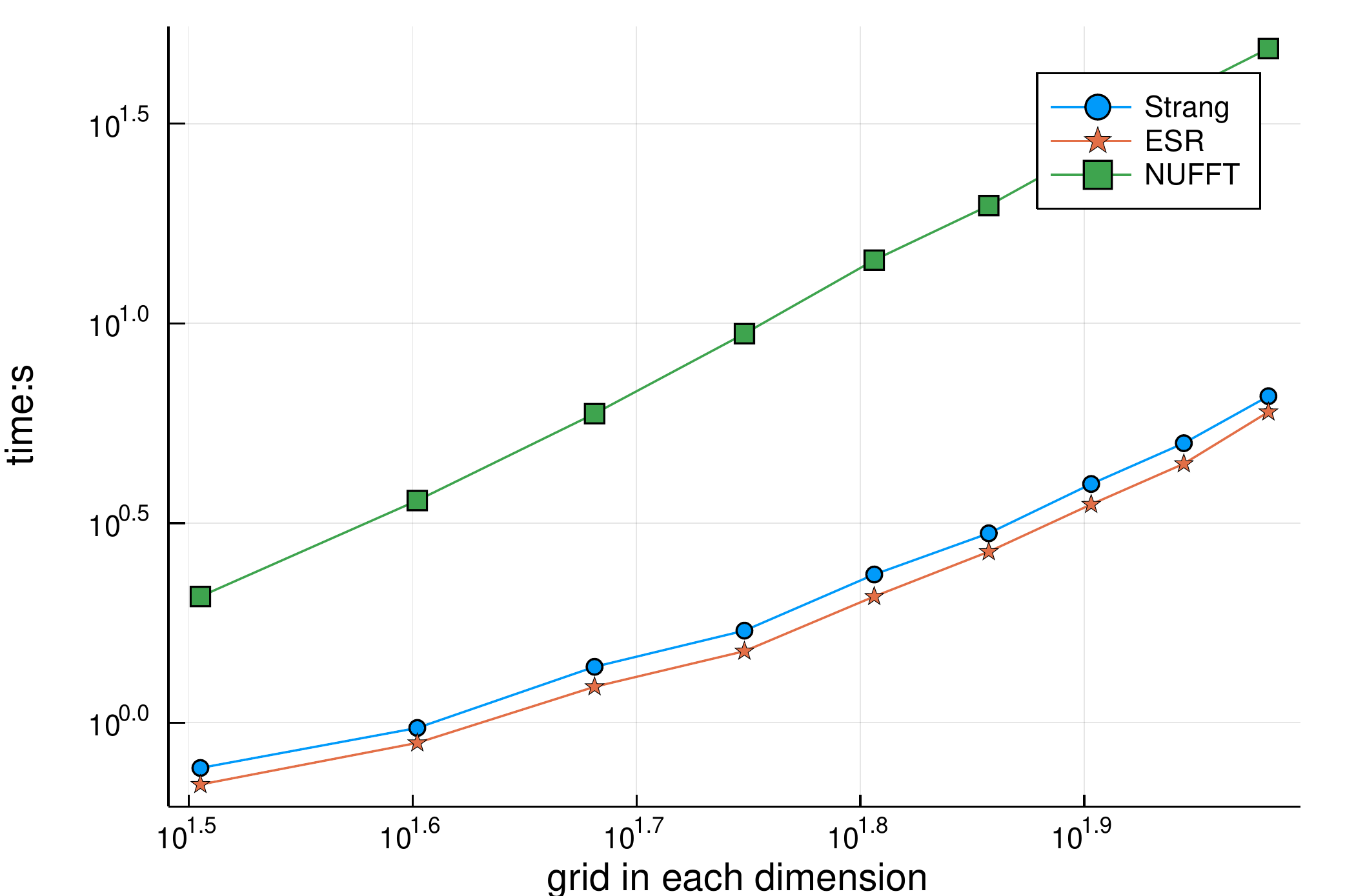}}
}
\caption{(a)Time evolution of $L^2$ error (semi-$\log_{10}$ scale) for NUFFT, ESR, and Strang for 3D transport problem 
with grids $64^3$ and step size $\Delta t = 0.3$; (b) CPU time for NUFFT, ESR and Strang of after running $100$ steps.}\label{fig:error_toy}
\end{figure}


\noindent{\bf 4D transport equation}\\
We consider now the case $n=4$ where the matrix $M$ in \eqref{eq:rot} is given by 
$$
M = \begin{pmatrix} 
0 & 1 &-1.5  & -3\\
 -1 & 0 & 2 & 1\\
 1.5 & -2 & 0 & 0 \\
  3 & -1 & 0 & 0
 \end{pmatrix}.
 $$
 The domain is defined by $R = 5$ and the initial value is 
  \begin{align}
 f_0({\mathbf x}) = \left(\frac{2}{\pi}\right)^4 e^{-{|\mathbf x|}^2}.
 \end{align}
Since the direct 4D semi-Lagrangian method would be too costly, we compare here the Strang directional 
splitting and the new method ESR. From Proposition ~\ref{prop_rot_gen}, we define 
the ESR method by  
$$
e^{\Delta t M{\mathbf x}\cdot \nabla} = e^{\Delta t(y^{(\ell)}\cdot {\mathbf x})\partial_{x_2}}  \left(  e^{\Delta t(y^{(1)}\cdot {\mathbf x})\partial_{x_1}}  e^{\Delta t(y^{(3)}\cdot {\mathbf x})\partial_{x_3}} e^{\Delta t(y^{(4)}\cdot {\mathbf x})\partial_{x_4}}   \right)e^{\Delta t(y^{(r)}\cdot {\mathbf x})\partial_{x_2}},
$$  
whose coefficients are given by ($\Delta t=0.05$ here) 
$$
y^{(\ell)} \simeq \begin{pmatrix} 
 7.239003439520237 \\
 0 \\
 0.114915806141710  \\
 5.520828626111525  
 \end{pmatrix}, 
 \;\; 
 y^{(r)} \simeq \begin{pmatrix} 
 -7.124076298503538 \\
  0 \\
 -1.578152453772511  \\
  -5.445447353939971
 \end{pmatrix}, 
 $$
and 
$$
 y^{(1)} \!\simeq\! \begin{pmatrix} 0 \\
 -0.843365270467026 \\
 1.542035786578973  \\
  3.239936743553417
 \end{pmatrix}\!\!, \; 
 y^{(3)} \!\simeq\!\! \begin{pmatrix} 
 -2.171812638482090 \\
 1.937050589058292 \\
 0  \\
 -0.368205582274782
 \end{pmatrix}\!\!,
\;\; 
y^{(4)} \!\simeq\!\! \begin{pmatrix} 
 -3.333162655369549 \\
 0.915289658696578 \\
 0.087822505478295  \\
  0
 \end{pmatrix}\!\!. 
 $$

In the following numerical results, the space grid has $N=47$ points per direction and the final computation time is $t = 30$ 
for the two methods. In Figure \ref{fig:4Dcontour}, the time evolution (in semi-$\log$ scale) of the $L^2$ error 
defined in \eqref{error_L2} is plotted for the Strang method: the error grows up to $10^{-4}$ whereas the error 
for ESR is about $10^{-11}$. Moreover, some contour plots are also presented in Figure \ref{fig:4Dcontour}: the two-dimensional 
quantity $f(t, x_1, x_2, x_3=-0.9574, x_4=-0.9574)$ for $t=0$ and $t=30$ is displayed for ESR and Strang.  
One can observe that the Strang method has large errors which is partly due to the wrong angular velocity. 
Let us remark that even if pseudo-spectral method have been chosen here to make the error close to machine precision,  
alternative reconstruction methods can also be chosen such as high order interpolation methods (see \cite{highorder}). 
Regarding the complexity, only $n+1=5$ shears are required in the exact splitting for each time step whereas 
$2n-1=7$ shears are needed for the Strang splitting. 

\begin{figure}[htbp]
\center{
\subfigure[]{\includegraphics[scale=0.45]{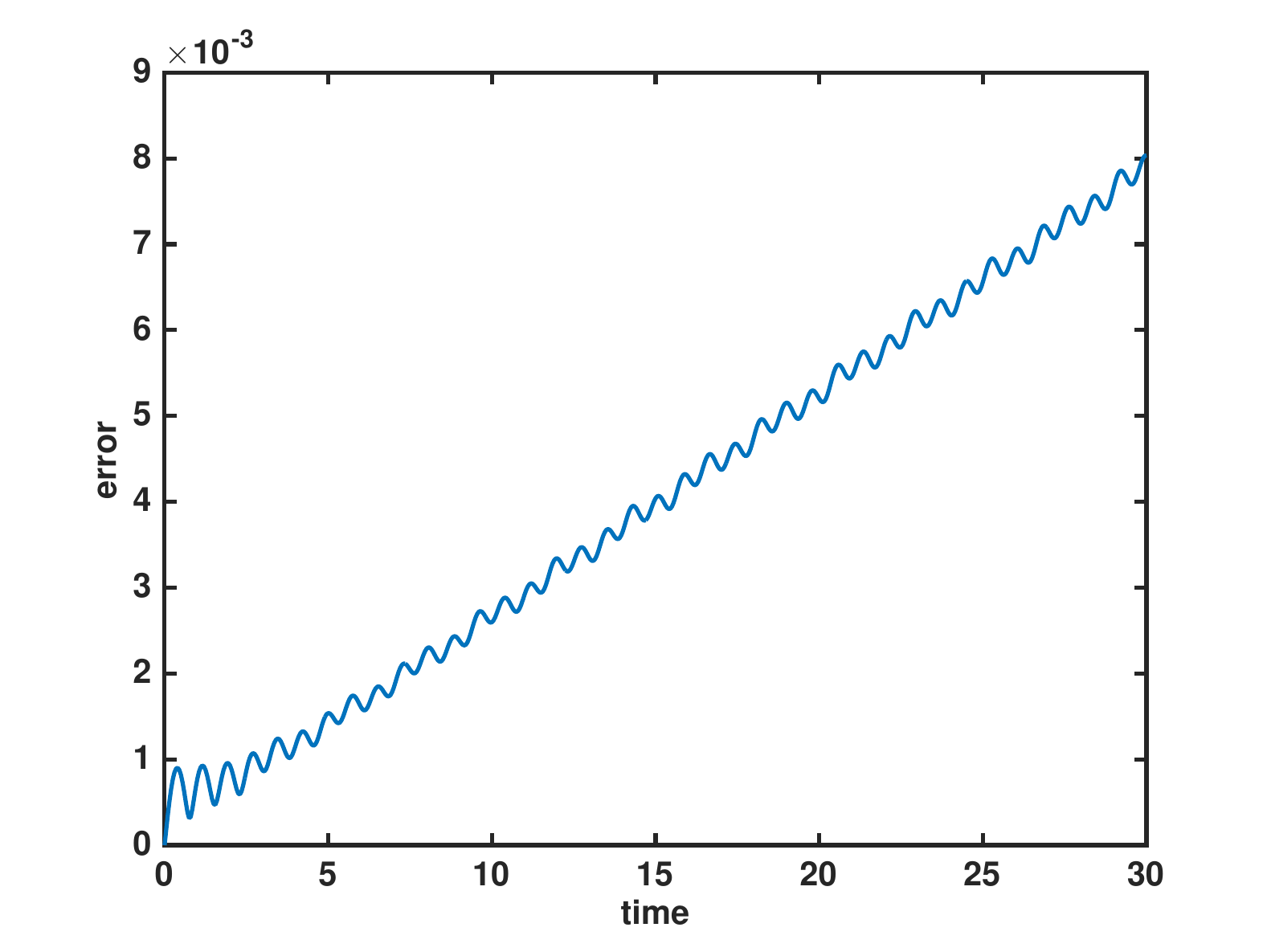}}
\subfigure[]{\includegraphics[scale=0.45]{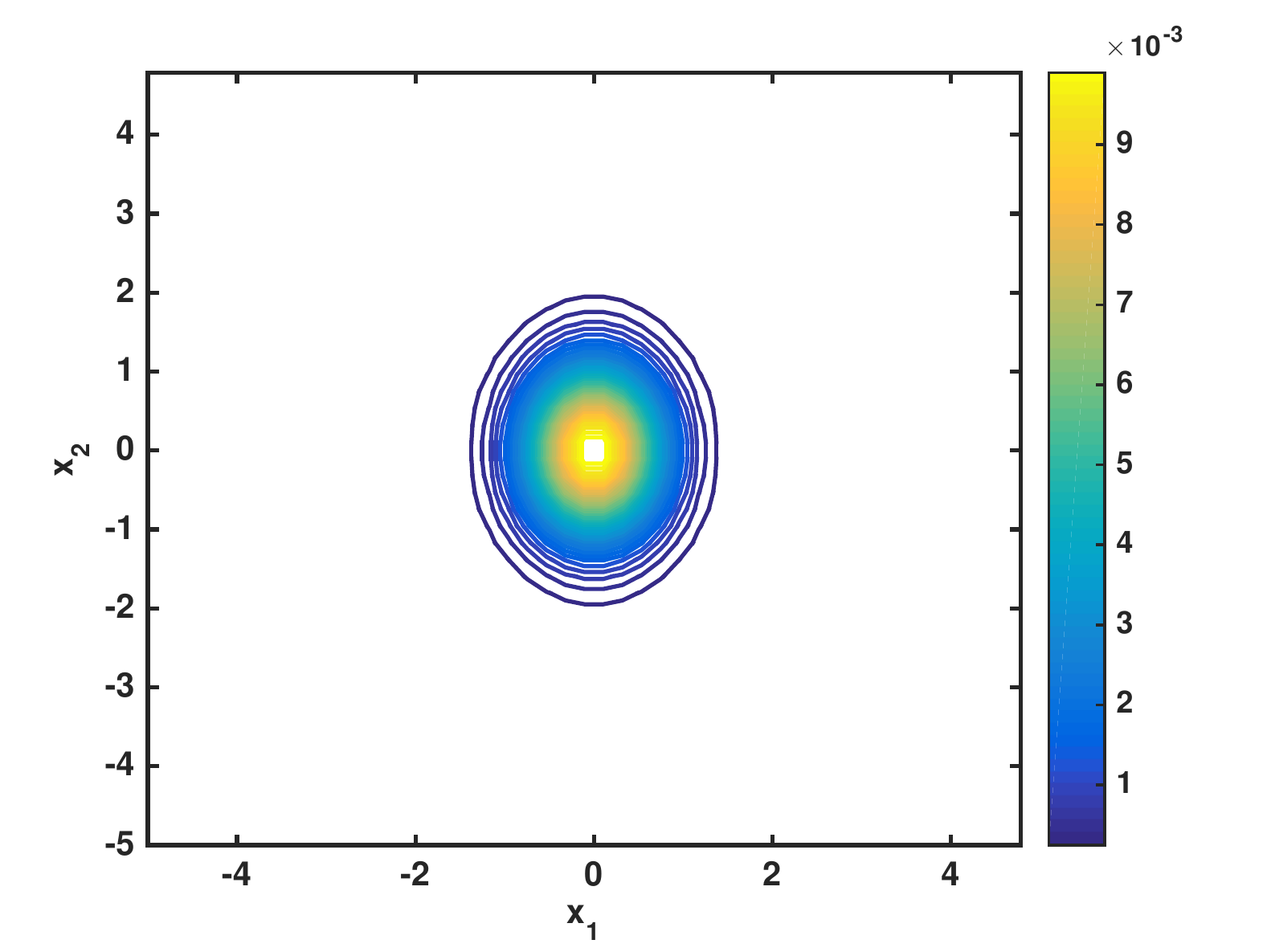}}\\
\subfigure[]{\includegraphics[scale=0.45]{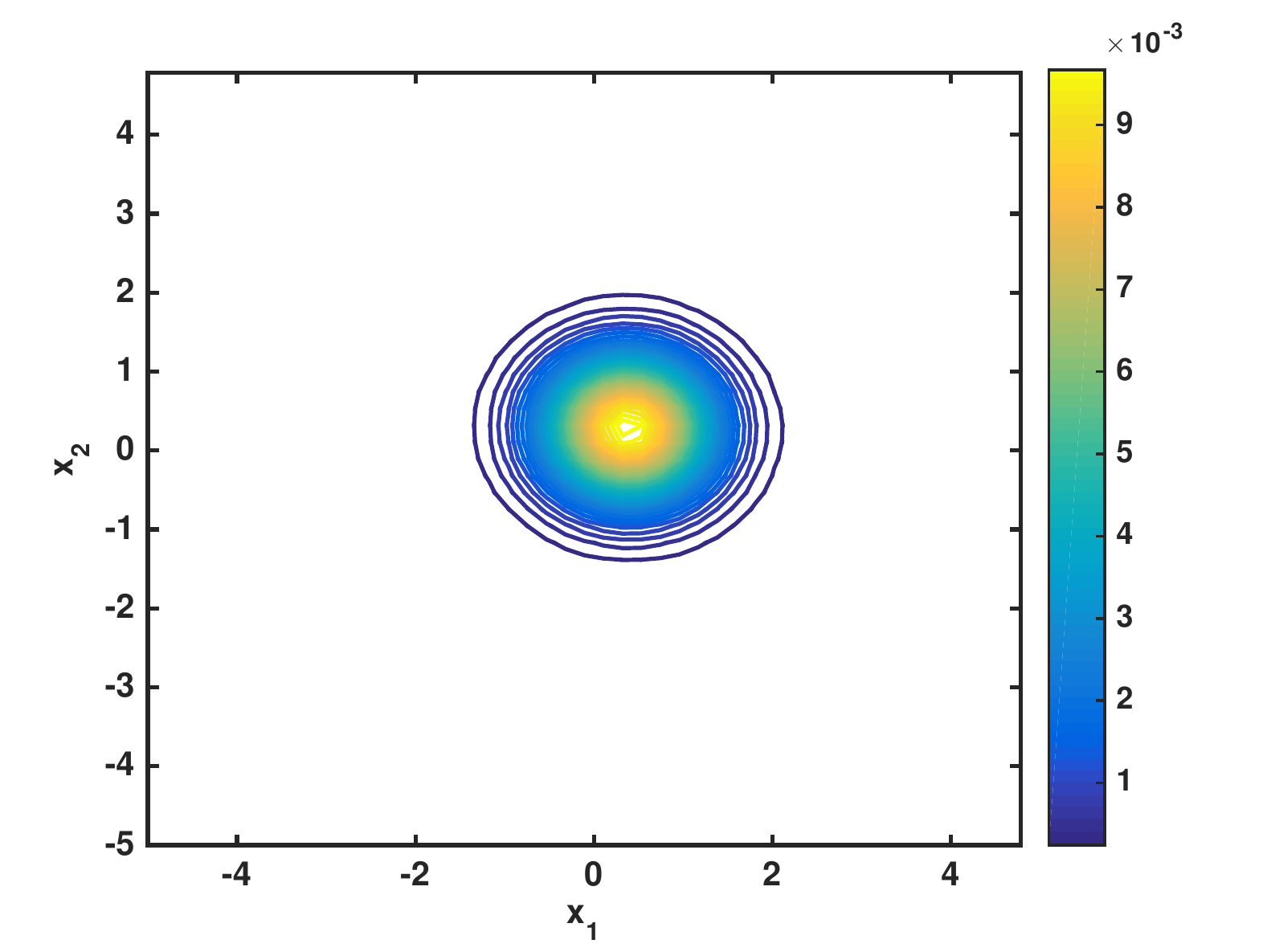}}
\subfigure[]{\includegraphics[scale=0.45]{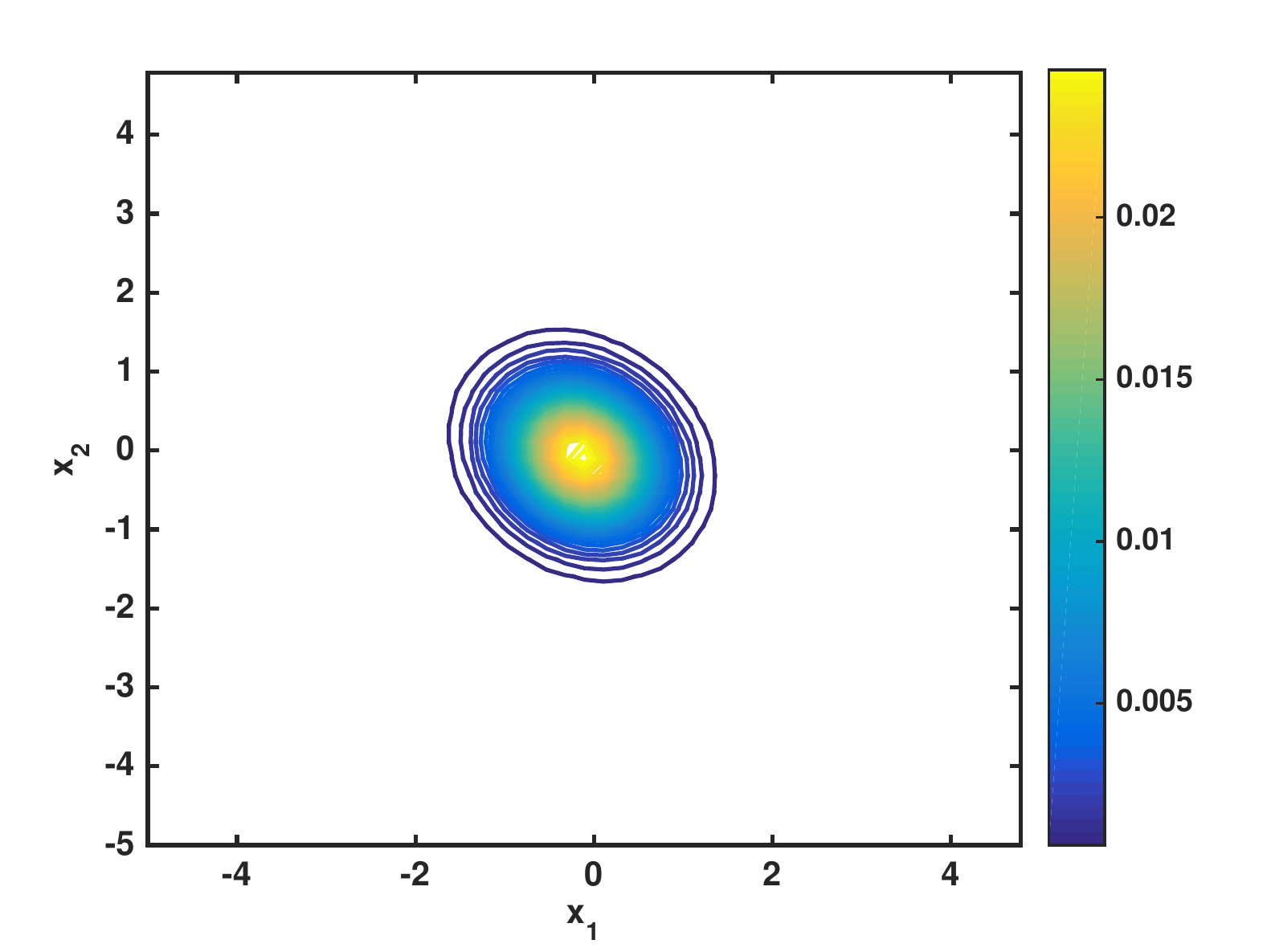}}
}
\caption{(a) Time evolution of $L^2$ error (semi-$\log_{10}$ scale of Strang; (b) Initial contour plot of $f(t=0, x_1, x_2,-0.9574,-0.9574 )$; (c) Contour plot of $f(t=30, x_1, x_2,-0.9574,-0.9574 )$ by ESR; (d) Contour plot of $f(t=30, x_1, x_2,-0.9574,-0.9574 )$  
by Strang. }\label{fig:4Dcontour}
\end{figure}

\section{Application to Fokker-Planck equations}
In this section, we are interested in Fokker-Planck type equations which can be used to describe particles system 
(in plasma physics or astrophysics). The unknown is a distribution function of particles 
$f(t, x, v)\in\mathbb{R}^+$ with the time $t\geq 0$ the space $x\in \mathbb{R}$ and velocity $v\in \mathbb{R}$.  
We will focus on two examples which contains a free transport part in $x$ and an operator (related to collisional terms) 
which only acts on the $v$ direction. The first example is the Kramer-Fokker-Planck equation (see \cite{herau2, herau, exponential} for some mathematical and numerical aspects) 
\begin{equation}
\label{kfp}
\tag{KFP}
\partial_t f + v^2 f-\partial_v^2f+v\partial_x f= 0,  \;\; f(t=0, x, v)=f_0(x, v).  
\end{equation}
The second example is the Fokker-Planck equation (see \cite{herau3, herau, exponential} for some mathematical and numerical aspects) 
\begin{equation}
\label{fp}
\tag{FP}
\partial_t f  + v\partial_x f - \partial_v^2 f - \partial_v (v f )= 0, \;\; f(t=0, x, v) = f_0(x,v). 
\end{equation}
For these two examples which enter in the class  
of inhomogeneous quadratic equations, 
exact splittings will be recalled from \cite{essiqo} and numerical results will be given.

\subsection{Presentation of the exact splittings}
For the Kramer-Fokker-Planck equation, the symbol is $p(x, v, \xi, \eta) = v^2 + \eta^2 + iv\xi$, 
and it writes $p(x, v, \xi, \eta) = iv \xi + \eta^2  - iv \eta - \frac12$ for the Fokker-Planck equation,  
where $\xi$ (resp. $\eta$) denotes the Fourier variable of $x$ (resp. $v$). We can see that for both cases, 
it is a polynomial function of degree $2$ and according to \cite{essiqo}, the solution can be split exactly 
into simple flows. More precisely, for KFP, we have the following exact splitting formula 
\begin{equation}
\label{eq:exKFP}
\forall t\geq 0, \ e^{-t(v^2-\partial_v^2+v\partial_x)} = e^{-\frac12 \tanh t \ v^2}  e^{ \nabla\cdot (A^{KFP}_t \nabla )} e^{- \tanh t \ v\partial_x} e^{-\frac12 \tanh t \ v^2},
\end{equation}
with $\nabla=(\partial_x, \partial_v)$ and where $A_t$ is the following nonnegative matrix defined by
\begin{equation}
\label{akfp}
A^{KFP}_t = \frac12 \begin{pmatrix} \frac{1}{2}\left(t  - \tanh t(1-\sinh(t)^2)\right)  & \sinh^2 t\\
\sinh^2 t & \sinh 2t
\end{pmatrix} .
\end{equation}
For FP, we have the exact splitting formula
\begin{equation}
\label{top}
e^{- t ( v\partial_x - \partial_v^2 - \partial_v v  )} = e^{t/2}  e^{- (e^t-1) v\partial_x } e^{ \nabla \cdot (A^{FP}_t \nabla)} e^{ i\alpha_t  \partial_v^2}  e^{-i \beta_t  v^2} e^{-i \beta_t \partial_v^2} e^{i \alpha_t  v^2}, 
\end{equation}
where  {$\alpha_{t} = \frac12 \sqrt{(1 - e^{-t})e^{-t}}$, $\beta_t = \frac12 \sqrt{e^t-1}$}, 
and $A^{FP}_t$ is the positive following matrix (see \cite{AB}) defined by
$$
A^{FP}_t =\frac12 \begin{pmatrix} e^{2t} + 2t + 3 -4 e^t  & -4 \sinh^2(t/2) \\
 -4 \sinh^2(t/2) & 1-e^{-2t}
\end{pmatrix}.
$$

Below, we detail a bit the link between splitting for PDE and finite dimensional Hamiltonian systems for the KFP case. 
Following \cite{essiqo}, the exact splitting \eqref{eq:exKFP} is equivalent to prove the following equality between matrices
\begin{equation}
\label{eq:exactode}
e^{-2itJQ} = e^{-2itJQ_1}e^{-2itJQ_2}e^{-2itJQ_3}e^{-2itJQ_1},
\end{equation}
where $J$ is the symplectic $4$x$4$ matrix, $Q, Q_i \in \mathrm{S}_4(\mathbb C)$ ($i=1, 2, 3$) are the matrices corresponding 
to the quadratic form $q, q_i$ ($i=1, 2, 3$)  defining the operators involving in the exact splitting \eqref{eq:exKFP}.  
Indeed, the quadratic form $q$ associated to the quadratic operator $q^w:=v^2-\partial_v^2 +v\partial_x$ is 
$q(X)=\transp{X} Q X$ 
with $X=(x, v, \xi, \eta)$ and where $Q$ is given by 
$$
Q=\begin{pmatrix} 
0 & 0 & 0 & 0 \\ 
0 & 1 & i/2 & 0 \\
0 & i/2 & 0 & 0 \\ 
0 & 0 & 0 & 1 
\end{pmatrix}. 
$$ 
Let us define the other quadratic form involved in \eqref{eq:exKFP}: 
$ q_1(X) = \frac{\tanh t}{t}v^2= \transp{X} Q_1 X, \ q_2(X) =  \transp{X} Q_2 X, \ q_3 = i\frac{\tanh t}{t}v \xi =  \transp{X} Q_3 X$ 
where 
$$
Q_1=\frac{\tanh t}{t} 
\begin{pmatrix} 
0 & 0 & 0 & 0 \\ 
0 & 1 & 0 & 0 \\
0 & 0 & 0 & 0 \\ 
0 & 0 & 0 & 0
\end{pmatrix},  
\; 
Q_2=\frac{1}{t}
\begin{pmatrix} 
{\bf 0}_2  & {\bf 0}_2  \\ 
{\bf 0}_2  & A_t^{KFP} 
\end{pmatrix}, 
\;
Q_3=\frac{\tanh t}{t} 
\begin{pmatrix} 
0 & 0 & 0 & 0 \\ 
0 & 0 & i/2 & 0 \\
0 & i/2 & 0 & 0 \\ 
0 & 0 & 0 & 0 
\end{pmatrix},  
$$ 
where ${\bf 0}_2$ is the zero $2$x$2$ matrix and $A_t^{KFP}$ is given by \eqref{akfp}.

Let us now focus on the FP case. From Theorem 2.1 in~\cite{essiqo}, we have  
$$
e^{- t ( v\partial_x - \partial_v^2 - \partial_v v  )} = e^{- t( iv \xi + \eta^2  - iv \eta - \frac12  )^w} = e^{t/2} e^{- t( \eta^2  + i(v\xi-v \eta))^w },  
$$
so that the $e^{- t( \eta^2  + i(v\xi-v \eta))^w }$ only involves homogeneous quadratic form and can be split as 
\eqref{top} and then can be checked as in the {\color{red}\ref{kfp}} case.


\subsection{Numerical results}
In this section, numerical simulations are performed using the above exact splittings to illustrate the 
exponential decay to equilibrium property and regularizing effects. Let us remark that since the  Fokker-Planck and Krammer-Fokker-Planck operators are homogeneous with respect to the space variable $x$, we do not 
have to consider localized functions in this direction and we can thus periodic functions in this direction. 
The domain is truncated to $(x, v)\in [-R_1, R_1]\times [-R_2, R_2]$ and the number of points is denoted 
by $N_1$ (resp. $N_2$) to sample the $x$-direction (resp. the $v$-direction).

\noindent{\bf Fokker-Planck equation}\\
For the FP equation, we aim at checking an important property that the solution converges 
to the equilibrium state exponentially with time (see \cite{exponential}).
The domain is taken as $R_1=\pi$ and $R_2=7$ (so that $(x, v)\in [-\pi, \pi] \times [-7,7]$) 
and the initial function is 
$$
f_0(x,v) = \frac{1}{\sqrt{2\pi}}\exp(-v^2/2) \left(1+0.5\sin(x)\cos(\frac{\pi}{7}v)\right).
$$
We will be interested in the time evolution of the entropy which is defined by 
\begin{equation}
\label{entropy}
\mathcal{F}(t) = \int_0^L \int_{-\infty}^{\infty} \frac{(f(t, x, v)-\mu(v))^2}{\mu(v)} \mathrm{d}v\mathrm{d}x,
\end{equation}
with $\mu(v) = \frac{1}{\sqrt{2 \pi}} e^{-\frac{v^2}{2}}$.

The numerical parameters are $N_1=27$, $N_2=181$ and the time step is 
$\Delta t = 0.1$ whereas the simulation is ended at $t =20$. In Figure \ref{fig:FP}, the distribution function 
is plotted at the initial and the ending time and we can observe the relaxation towards the Maxwellian profile. 
This is more quantitatively shown in Figure \ref{fig:FP}-(c) where the time history of entropy \eqref{entropy} is 
plotted (semi-$\log_{10}$ scale). Indeed, the exponential decay is clearly observed, the rate of which 
is equal to $-1.99$ (red straight line) which is good agreement with \cite{exponential}.
\begin{figure}[htbp]
\center{
\subfigure[]{\includegraphics[scale=0.45]{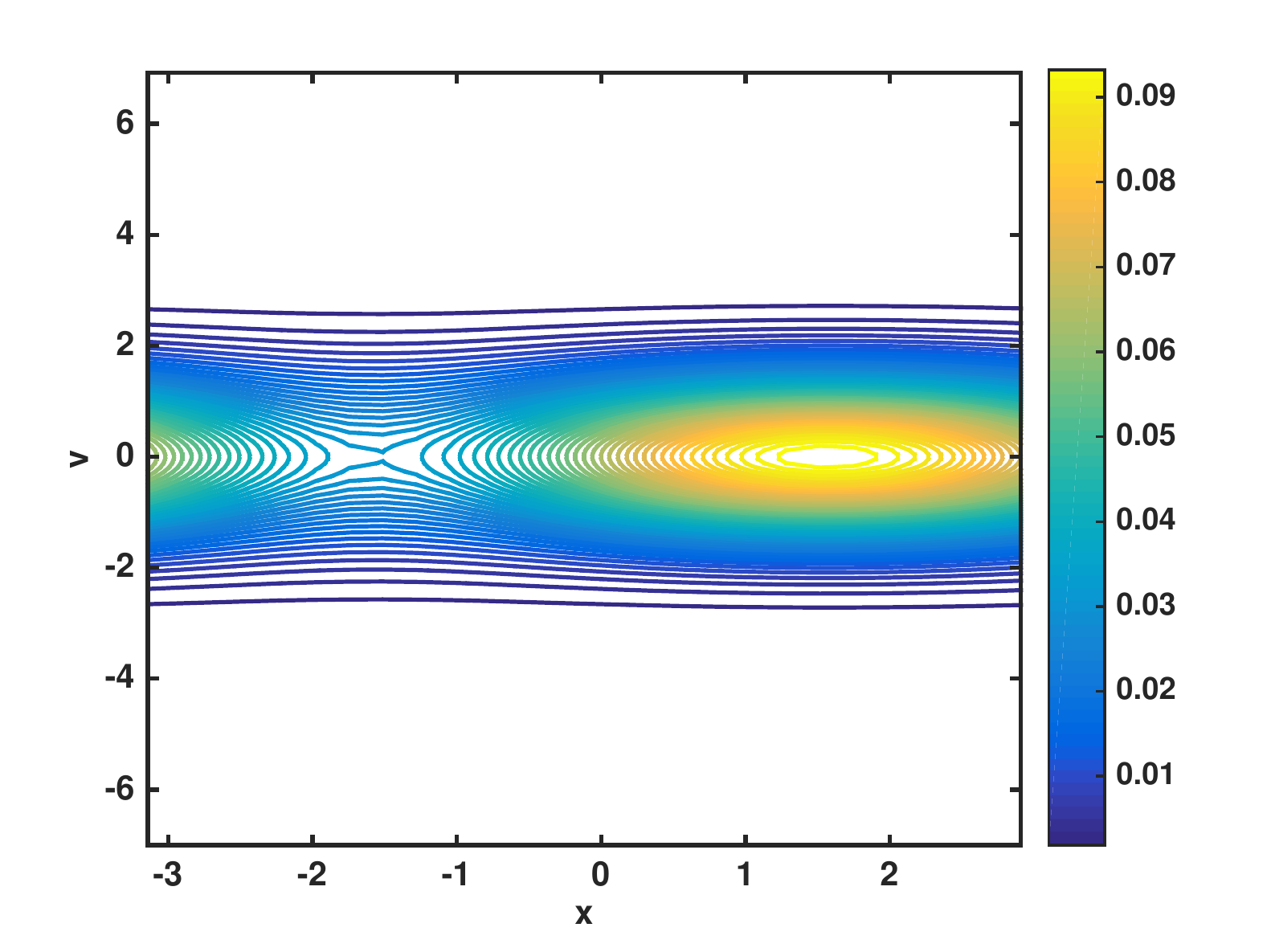}}
\subfigure[]{\includegraphics[scale=0.45]{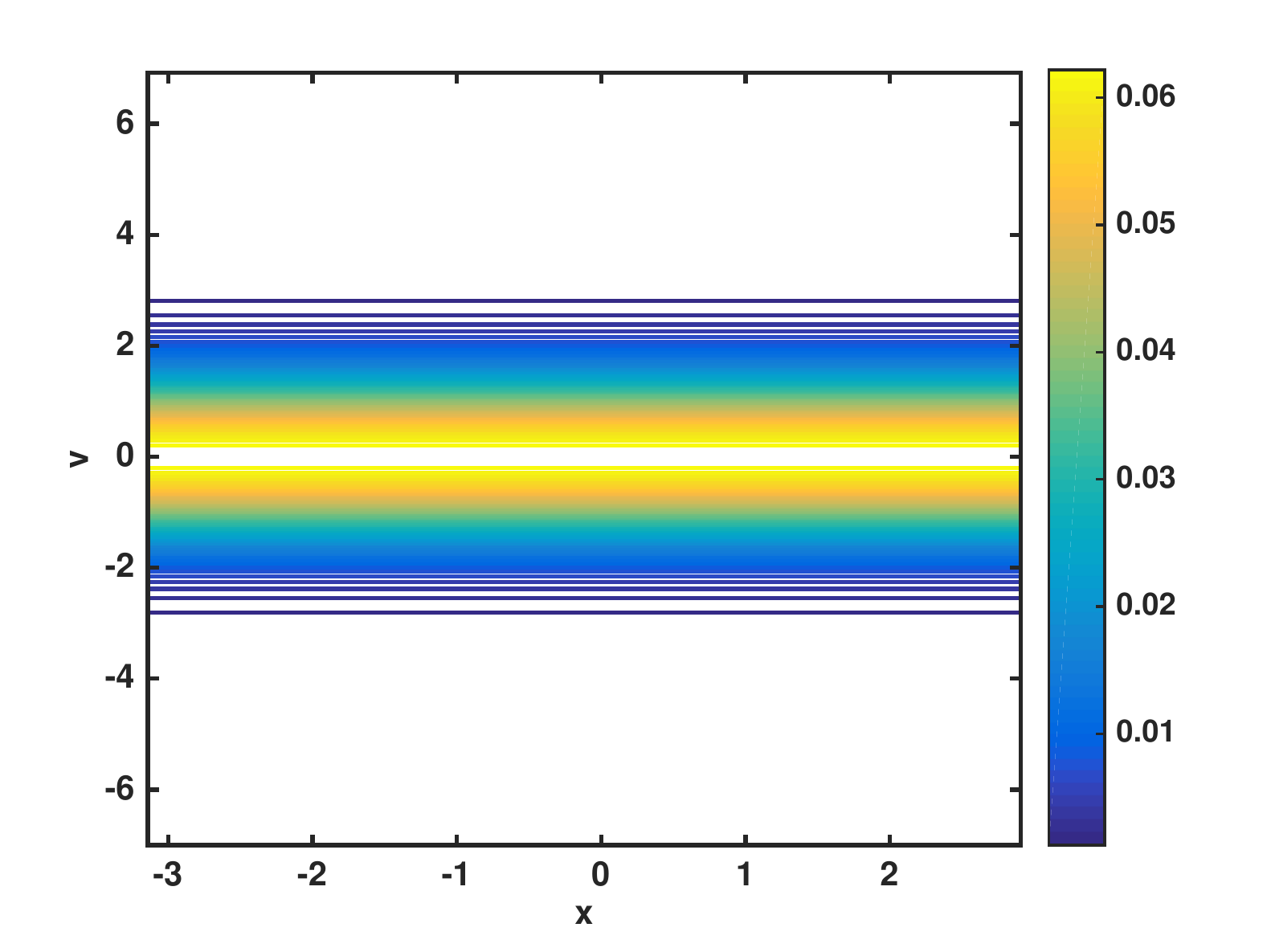}}\\
\subfigure[]{\includegraphics[scale=0.45]{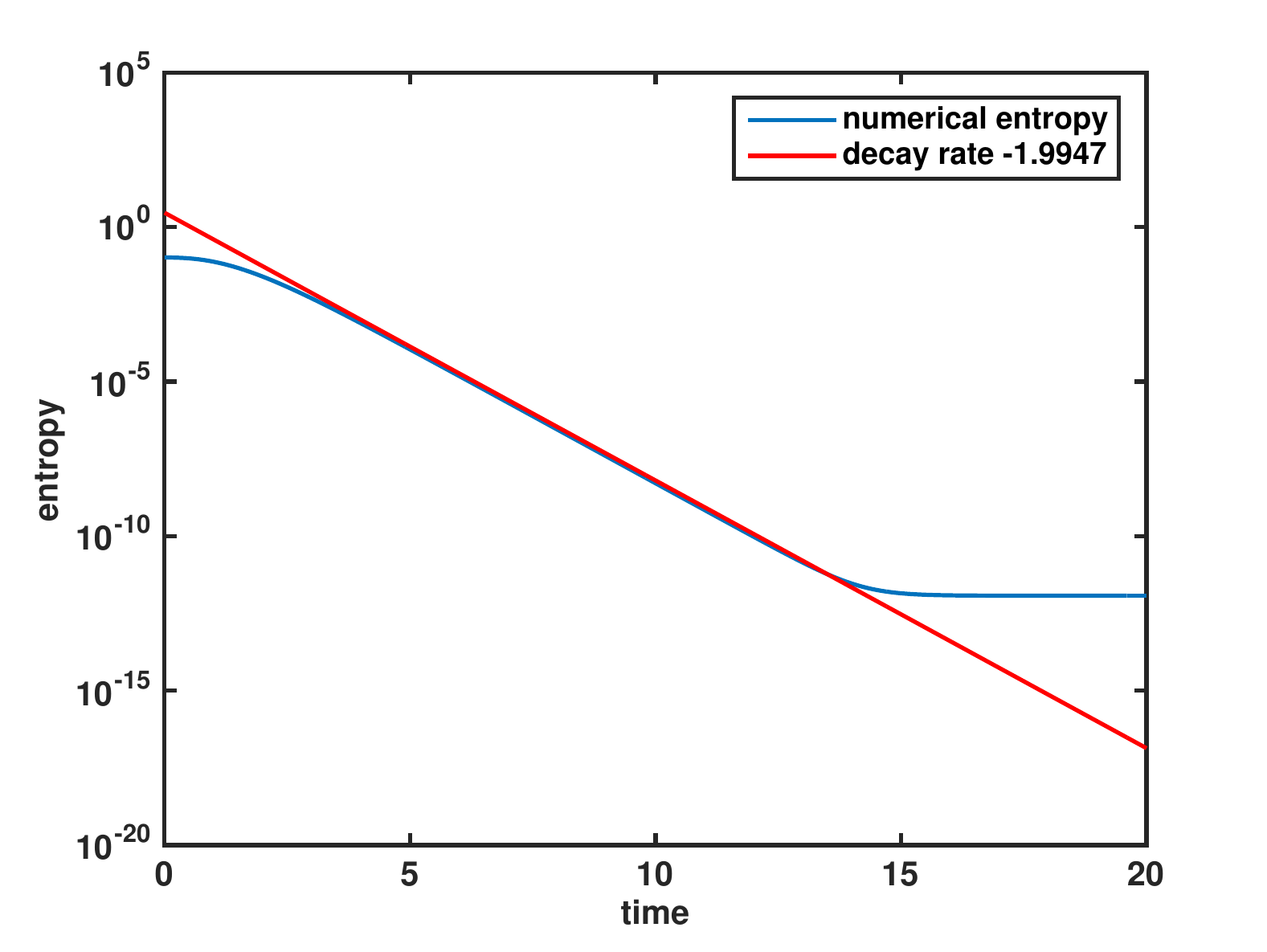}}
}
\caption{(a) Initial distribution function $f(t=0, x, v)$; (b) distribution function at $t=20$: $f(t=20, x, v)$; 
(c) time evolution of the entropy given by \eqref{entropy}.}
\label{fig:FP}
\end{figure}

\noindent{\bf Kramer-Fokker-Planck equation}\\
Now, we are interested in the numerical simulation of the {\color{red}\ref{kfp}} equation. 
The domain is chosen with $R_1=4$ and $R_2=15$ (so that $(x, v)\in [-4, 4] \times [-15, 15]$).  
In the following experiments, we have considered $N_1=N_2=199$. 

First, in Figure \ref{fig:Joackim_error}, we plot the $L^2$ error  in $x$ and $v$ (in $\log_{10}-\log_{10}$ scale) 
of the following formula  
\begin{equation}\label{eq:Joackim_error}
e^{- \Delta t(v^2-\partial_v^2+v\partial_x)}f_0(x, v)-e^{- \Delta t/2(v^2-\partial_v^2+v\partial_x) }e^{- \Delta t/2(v^2-\partial_v^2+v\partial_x) }f_0(x, v)
\end{equation}
using different time steps and where the initial data is smooth enough (a Maxwellian is considered here). 
Two different methods are used to compute the quantity $e^{- t(v^2-\partial_v^2+v\partial_x) }f_0(x, v)$: 
the exact splitting  \eqref{eq:exKFP} and a Strang operator splitting. 
First, we observe that the exact splitting gives an error at the machine precision level  
whereas we obtain an error of ${\cal O}(\Delta t^3)$ which corresponds the local error of the Strang method. 

Then, our goal is to illustrate a result from \cite{AB} in which the authors proved that the evolution operator of KFP 
has regularizing effects. To do so,  the initial condition is chosen as random values (discrete $L^1$ norm is 1) 
 and the step size is $\Delta t = 0.1$.  
In Figure \ref{fig:KFP}, the distribution function is plotted for different times: $t=0, 0.2, 1, 100$. 
We observe that starting from a random initial value, the numerical solution becomes smoother and smoother as 
time increases. Moreover, it can be proved that the solution is exponentially decreasing in time towards zero. 
This is illustrated in Figure \ref{fig:KFP_log} where we plot the time history $L^2$ norm (in $x$ and $v$) of $f$ in semi-$\log_{10}$ 
scale. 
\begin{figure}[htbp]
\center{
\subfigure[]{\includegraphics[scale=0.45]{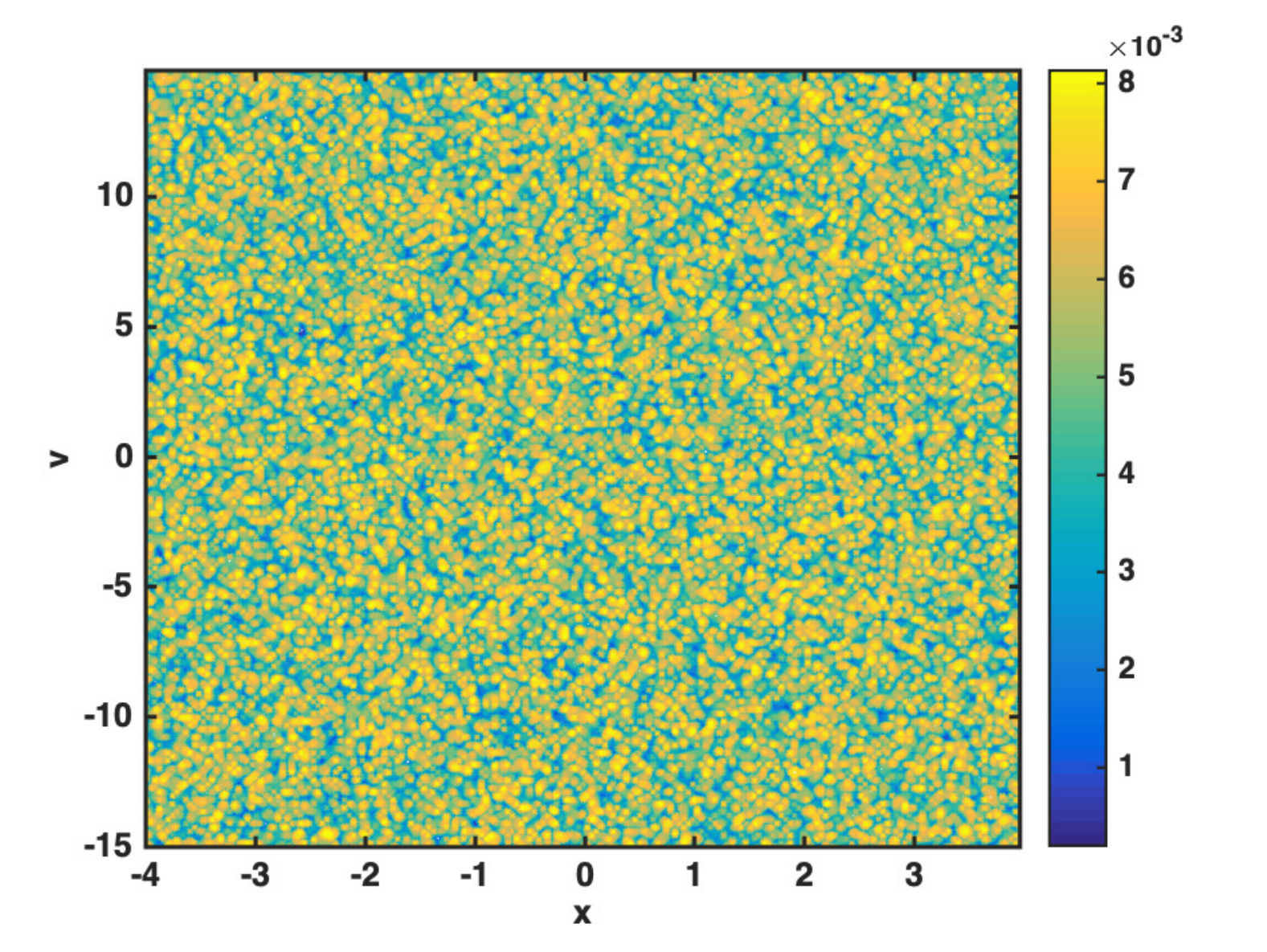}}
\subfigure[]{\includegraphics[scale=0.45]{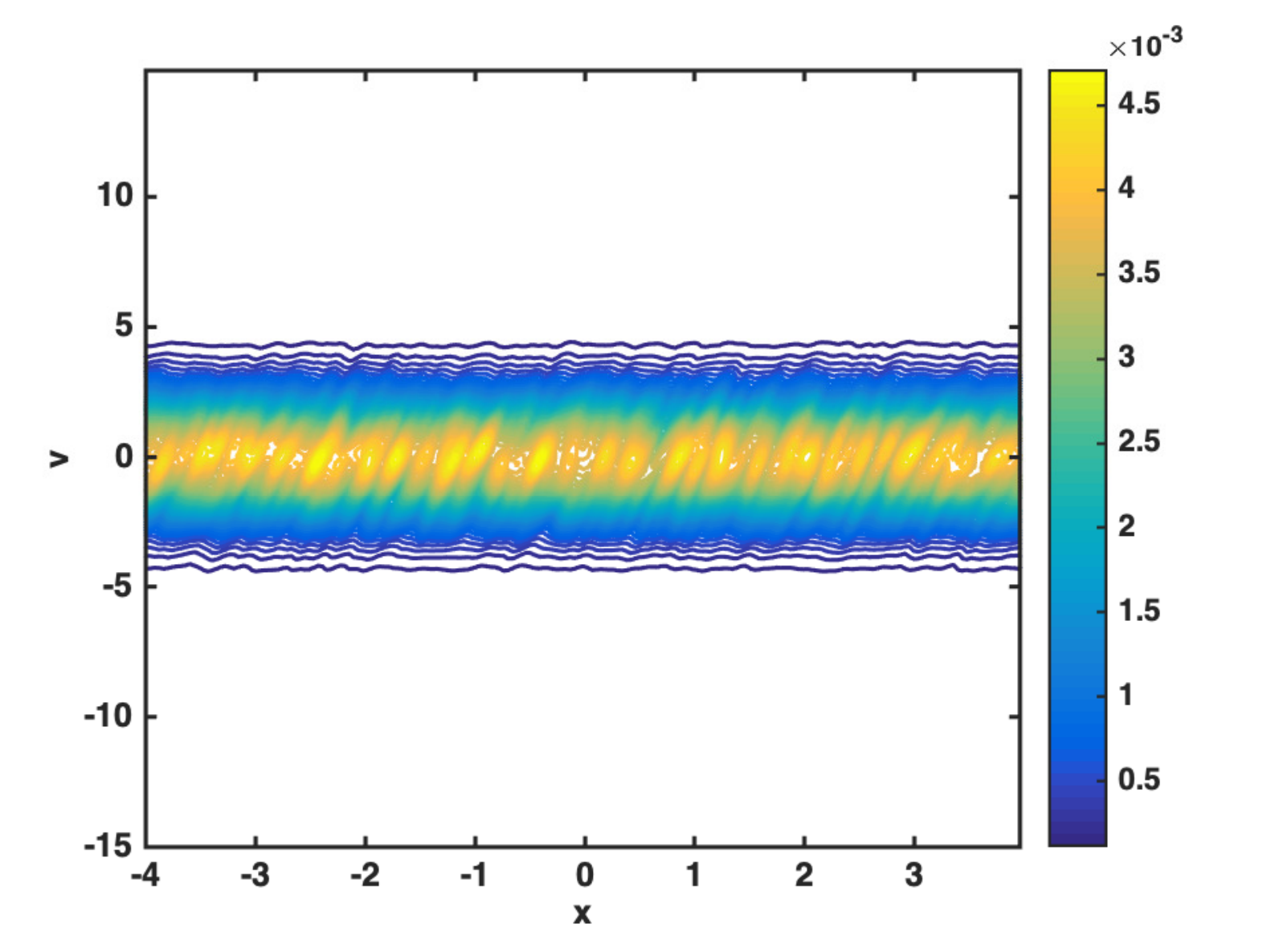}}\\
\subfigure[]{\includegraphics[scale=0.45]{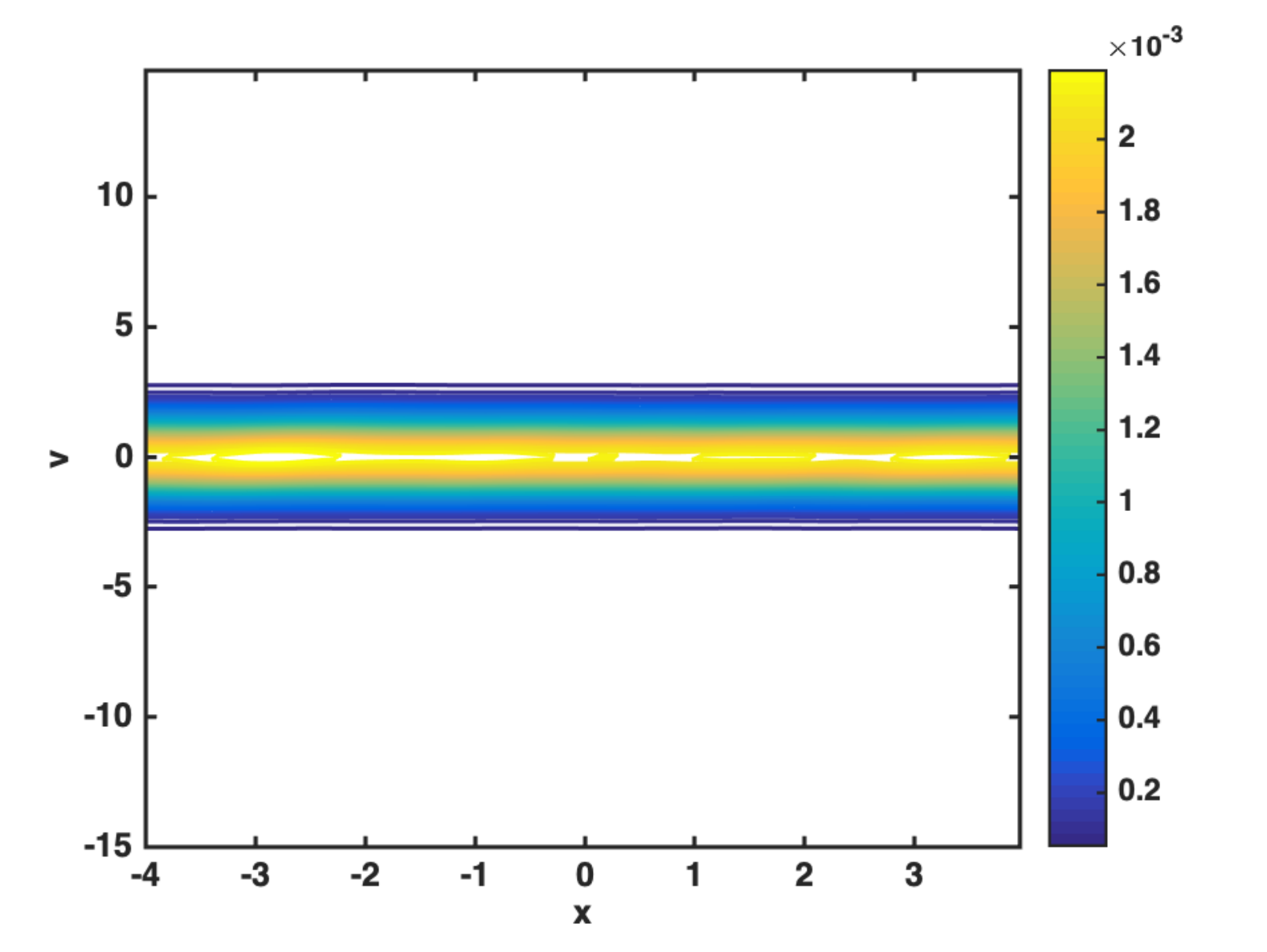}}
\subfigure[]{\includegraphics[scale=0.45]{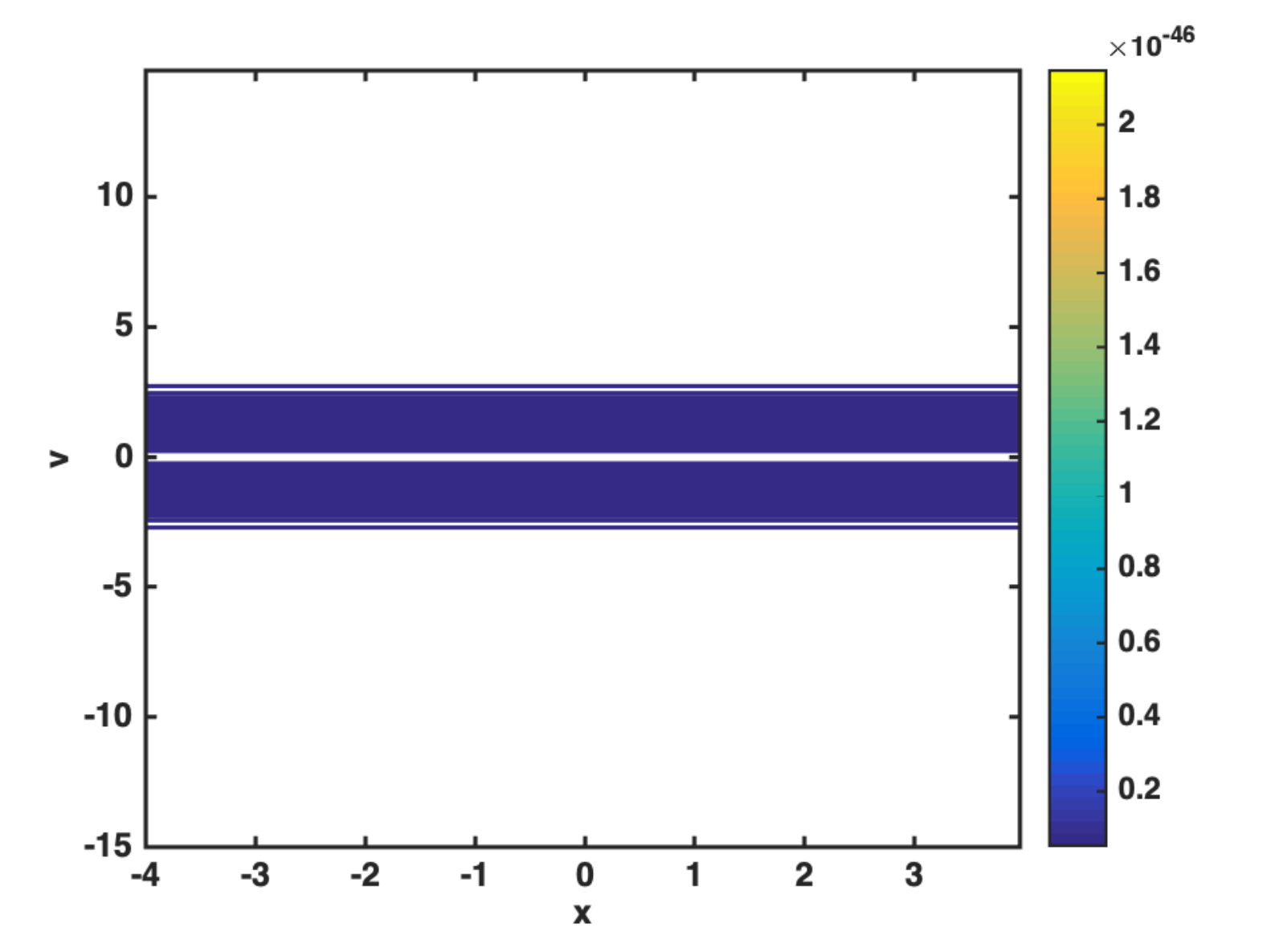}}
}
\caption{Time evolution of the distribution function $f$ (a) $t = 0$; (b) $t = 0.2$; (c) $t = 1$; (d) $t = 100$.}\label{fig:KFP}
\end{figure}
\begin{figure}[htbp]
\center{
{\includegraphics[scale=0.45]{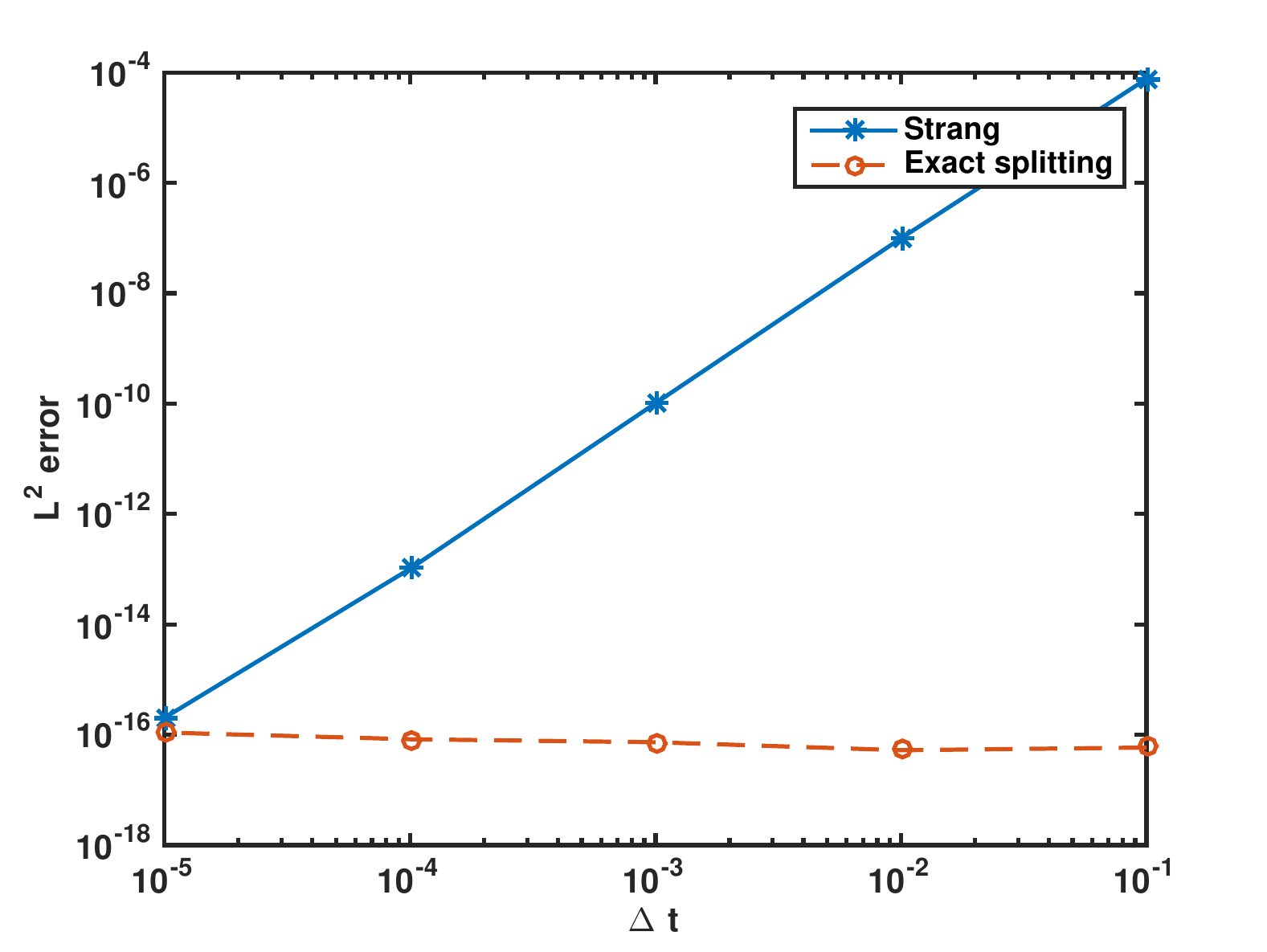}}
}
\caption{$L^2$ error ($\log_{10}-\log_{10}$ scale) of the formula \eqref{eq:Joackim_error} with different time step size $\Delta t$. } 
\label{fig:Joackim_error}
\end{figure}

\begin{figure}[htbp]
\center{
\includegraphics[scale=0.45]{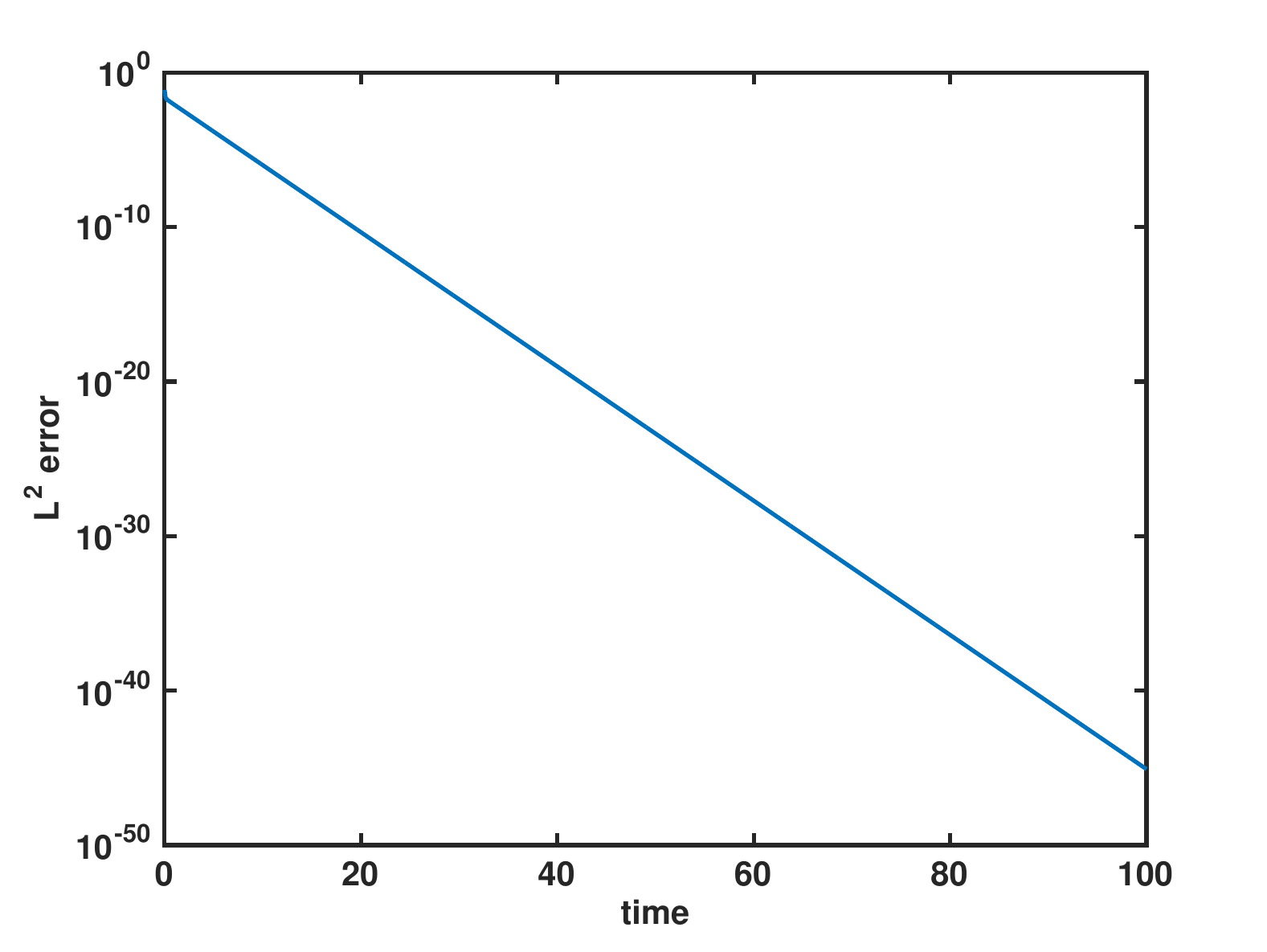}
}
\caption{Time evolution of the $L^2$ norm (in $x$ and $v$) of $f$ in semi-$\log_{10}$ scale.}
\label{fig:KFP_log}
\end{figure}

\section{Application to Schr\"odinger equations}
In this section, we consider Schr\"odinger type equations. First, the Schr\"odinger in the presence of 
an external electromagnetic field is an important model in computational quantum mechanics. 
The second model is the Gross-Pitaevskii  equation with angular momentum rotation which is widely used 
to describe Bose Einstein condensate at low temperature. We refer to  \cite{raymond, Ostermann, Xavier2, Bao, wang} 
for more details on these models. 

Then, we consider the following linear Schr\"odinger equation 
(with a rotation term and a quadratic external potential)  of unknown $\psi({\bf x}, t)\in \mathbb{C}$ 
with ${\mathbf x}\in \mathbb{R}^n, \ t \in \mathbb{R}_{+}$ 
\begin{equation}
\label{eq_QMS}\tag{QM}
\begin{aligned}
&i \frac{\partial \psi({\bf x}, t)}{\partial t} = -\frac{1}{2}\Delta \psi({\bf x},t) - i ({ B}  {\bf x})\cdot \nabla \psi({\bf x},t) + V({\bf x})\psi({\bf x},t),  \;\;  \psi({\mathbf x}, t=0) = \psi_0({\mathbf x}),
\end{aligned}
\end{equation}
where $n\in\mathbb{N}^*$, $B \in A_n(\mathbb{R})$ is a skew symmetric matrix of 
size $n$ and $V : \mathbb{R}^n \to \mathbb{R}$ is a quadratic potential. 
According to the previous framework, this model is an inhomogeneous quadratic PDEs 
since it can be represented by the following symbol 
\begin{equation}
\label{quad_pde}
p(X) = i\frac{|\boldsymbol{\xi}|^2}2   + iB {\mathbf{x}} \cdot {\boldsymbol \xi} + iV({\mathbf{x}}), \;\; X=({\mathbf{x}},\boldsymbol{\xi})\in \mathbb{R}^{2n}. 
\end{equation}
In the sequel, an exact splitting is presented for which the construction will be detailed. 
Then, an extension to nonlinear and non quadratic Schr\"odinger equations are discussed.  
This section will be ended by several numerical results that will be compared to different strategies from the literature 
to illustrate the efficiency of our approach. 

\subsection{Presentation of the exact splitting}
We present an exact splitting method for \eqref{eq_QMS} which has been introduced in~\cite{essiqo}. 

\begin{theorem} 
\label{thm_splt_MS2} 
There exists some quadratic forms $v_t^{(r)}, a_t$ on $\mathbb{R}^{n}$, a strictly upper triangular matrix $U_t \in M_n(\mathbb{R})$,  a strictly lower triangular matrix $L_t \in M_n(\mathbb{R})$ and a diagonal quadratic form $v_t^{(\ell)}$ on $\mathbb{R}^{n}$, all depending analytically on $t\in (-t_0,t_0)$ for some $t_0>0$, such that for all $t\in (-t_0,t_0)$ we have
\begin{equation}
\label{eq:MS2}
e^{ i t ( \Delta/2 - V({\mathbf x})) - t B {\mathbf x}\cdot \nabla } =e^{-i t v^{(\ell)}_t({\mathbf x})}  \left( \prod_{j=1}^{n-1} e^{- t (U_t {\mathbf x})_j \partial_{x_j}} \right) e^{i t a_t(\nabla)} \left( \prod_{j=2}^{n} e^{- t (L_t {\mathbf x})_j \partial_{x_j}} \right) e^{-i t v^{(r)}_t({\mathbf x})} 
\end{equation}
where $a_t(\nabla)$ denotes the Fourier multiplier of symbol $-a_t({\boldsymbol \xi})$ and $(U_t {\mathbf x})_j$ (resp. $(L_t {\mathbf x})_j$) the $j^{\mathrm{st}}$ coordinate of $U_t {\mathbf x}$ (resp. $L_t {\mathbf x}$).
\end{theorem}

Let us detail the steps of this splitting to emphasize the fact that, due the triangular structure of the matrices $L_t$ and $U_t$,    only $2n$  FFT calls are required. 
\begin{itemize}
\item From $e^{-i t v^{(r)}_t({\mathbf x})} $ to $e^{- t (L_t {\mathbf x})_n \partial_{x_n}}$, we need a FFT in $x_n$ direction;
\item From $e^{- t (L_t {\mathbf x})_j \partial_{x_j}}$ to $e^{- t (L_t {\mathbf x})_{j-1} \partial_{x_{j-1}}}, \ j \in \llbracket 3, n\rrbracket$, as $L_t$ is a strictly lower triangular matrix, $(L_t {\mathbf x})_j$ only depends on $x_i, \, i \in \llbracket 1, j-1 \rrbracket$, then we only need a FFT in $x_{j-1}$ direction. 
\item From $e^{- t (L_t {\mathbf x})_2 \partial_{x_2}}$ to $e^{i t a_t(\nabla)}$, we need a FFT in $x_{1}$ direction. 
\item From $e^{i t a_t(\nabla)}$ to $e^{- t (U_t {\mathbf x})_{n-1} \partial_{x_{n-1}}}$, we need an inverse FFT in $x_{n}$ direction;
\item  From $e^{- t (U_t {\mathbf x})_j \partial_{x_j}}$ to $e^{- t (U_t {\mathbf x})_{j-1} \partial_{x_{j-1}}}, \ j \in \llbracket 2, n-1\rrbracket$, because $U_t$ is a strictly upper triangular matrix, $(U_t {\mathbf x})_{j-1}$ only depends on $x_i, \, i \in \llbracket j, n \rrbracket$, we only need an inverse FFT in $x_{j}$ direction. 
\item From $e^{- t (U_t {\mathbf x})_{1} \partial_{x_{1}}}$ to $e^{-i t v^{(\ell)}_t({\mathbf x})} $, we need an inverse FFT in $x_{1}$ direction.
\end{itemize}
To sum up, this new method only needs $2n$ FFT (or  inverse FFT) calls, 

Below, we detail a bit the link between splitting for the Schr\"odinger equations 
and finite dimensional Hamiltonian systems.  
Following \cite{essiqo}, the exact splitting \eqref{eq:MS2} is equivalent to prove an equality at the level of matrices. 
Indeed, from H\"ormander \cite{Lars}, there exists a morphism between the Hamiltonian flow of the following linear 
ODE $\dot X = -iJ \nabla p(X)$ and $e^{-p^w}$ (up to one sign), where 
$J$  is the symplectic $2n$ matrix. 
So we can check the following exact splitting at the linear ODE level for (\ref{eq:MS2}): 
\begin{equation}\label{eq:MS2ODE}
e^{-2itJ Q} = e^{-2itJ V^{(\ell)}} \left( { \prod_{j=1}^{n-1}} e^{-2itJ U^{(j)}} \right) e^{-2itJ A}  \left( {\prod_{j=2}^n} e^{-2itJ L^{(j)}} \right) 
e^{-2itJ V^{(r)}}, 
\end{equation}
where  $Q, V^{(\ell)}, U^{(j)}, L^{(j)}, V^{(r)}$ are symmetric matrices of the quadratic forms (symbols) of the following 
operators $iv^{(\ell)}_t({\mathbf x}), i(U_t{\mathbf x})_j \partial_{x_j}, -iA_t(\nabla), i(L_t{\mathbf x})_j \partial_{x_j}, iv^{(r)}_t({\mathbf x})$ respectively.

\subsection{Practical construction of the splittings}
In this subsection, the iteration methods giving the coefficients of two above exact splittings are given. 
The proof for Theorem \ref{thm_splt_MS2} is based on the implicit function theorem 
which gives a practical way to construct the exact splitting. Indeed, it furnishes an iteration method which is 
presented below. We refer to \cite{essiqo} for more details.

In this context, the iterative method giving the coefficients of the exact splitting can be made much more explicit. Indeed, identifying a quadratic form with its symmetric matrix in order we define, if $t$ is small enough, we define, by induction, the following sequences
$$
\left\{\begin{array}{lll} A_{t,k+1} &=& A_{t,k} + I_n/2 - \widetilde{A}_{t,k} \\
L_{t,k+1} &=& L_{t,k} + L - \widetilde{L}_{t,k}\\
				U_{t,k+1} &=& U_{t,k} + U - \widetilde{U}_{t,k} \\
				V_{t,k+1}^{(m)} &=& V_{t,k}^{(m)} + V - \widetilde{V}_{t,k}^{(m)} + \frac{t}2 [D_{t,k},B ]  + \frac{t^2}2 D_{t,k}^2
\end{array}  \right.
$$
where $(A_{t,0},L_{t,0}+U_{t,0},V_{t,0}^{(m)}) = (I_n/2,B,V)$, $L+U=B$ and
$$
\begin{pmatrix} 2 \widetilde{V}_{t,k}^{(m)}  & \transp{\widetilde{L}_{t,k}}+\transp{\widetilde{U}_{t,k}} + t D_{t,k}\\
				\widetilde{L}_{t,k}+\widetilde{U}_{t,k} + t D_{t,k} & 2\widetilde{A}_{t,k}	 \\
\end{pmatrix} = -t^{-1} J \log(P_{t,k})
$$
and
\begin{multline*}
P_{t,k} =  \left[ \prod_{j=1}^{n-1} \begin{pmatrix} I_n + t U_{t,k}^{(j)}&   \\ & I_n - t \transp{ U_{t,k}^{(j)}}  \end{pmatrix}\right]  \begin{pmatrix} I_n &  2 t A_{t,k} \\  & I_n \end{pmatrix} 
\left[ \prod_{j=2}^n \begin{pmatrix} I_n + t L_{t,k}^{(j)}&   \\ & I_n - t \transp{ L_{t,k}^{(j)}}  \end{pmatrix}\right] \\ \times
  \begin{pmatrix} I_n &   \\  -2 t V_{t,k}^{(m)}  & I_n \end{pmatrix}
\end{multline*}
with $L_{t,k}^{(j)} = (e_j \otimes e_j)L_{t,k}$, $U_{t,k}^{(j)} = (e_j \otimes e_j)U_{t,k}$ and $(e_1,\dots,e_n)$ the canonical basis of $\mathbb{R}^n$.

As previously, one can prove that the sequence 
$(\widetilde{A}_{t,k}, \widetilde{L}_{t,k}, \widetilde{U}_{t,k}, -\frac{t}{2} D_{t, k}, V_{t,k}^{(m)} + \frac12D_{t,k})$ 
generated by this induction converges towards the elements which define the splitting in \ref{thm_splt_MS2}, i.e.  
$$
|A_t - \widetilde{A}_{t,k} | + |L_t - \widetilde{L}_{t,k} | + |U_t - \widetilde{U}_{t,k} | + |V_t^{(\ell)} + \frac12 D_{t,k} | + |V_t^{(r)} - V_{t,k}^{(m)} -  \frac12D_{t,k} |  \leq \left(\frac{t}{\tau_0} \right)^k.
$$
as soon as $t$ is small enough ($0<|t|<\tau_0$ for a given $\tau_0>0$). 

\subsection{Extension to more general Schr\"odinger equations} 
Before presenting some numerical results, some time discretizations based on the 
previous exact splitting are proposed here in order to tackle more general Schr\"odinger equations.  
Keeping the same notations $\psi({\bf x}, t)\in \mathbb{C}$ for the unknown (${\bf x} \in \mathbb{R}^n$ and $t\geq 0$), 
we then consider the following Schr\"odinger equation to illustrate our strategy 
\begin{equation}
\label{eq:sch}
\begin{aligned}
&i \frac{\partial \psi({\bf x}, t)}{\partial t} = -\frac{1}{2}\Delta \psi({\bf x},t) - i ({ B}  {\bf x})\cdot \nabla \psi({\bf x},t) + V({\bf x})\psi({\bf x},t) + f({\bf x}, |\psi|^2) \psi({\bf x},t),   
\end{aligned}
\end{equation}
where $f$ is a real valued function, ${B} \in A_n(\mathbb{R})$ is a skew symmetric matrix of size $n$, 
and $V({\bf x}) :\mathbb{R}^n\to \mathbb{R}$ is a real valued quadratic potential. 
Some well known examples can be given in the case $n=2, 3$ 
\begin{itemize}
\item $f({\bf x}, |\psi|^2) = \beta |\psi|^2$ ($\beta \in \mathbb{R}$) and $({ B{\mathbf x}}) \cdot \nabla \psi = \Omega(x_2 \partial_{x_1} - x_1\partial_{x_2})\psi (\Omega \in \mathbb{R})$. In this case, (\ref{eq:sch}) is the so-called Gross--Pitaevskii equation (GPE) with an angular momentum rotation term (see \cite{Bao, wang}).  
\item $f({\bf x}, |\psi|^2)=V_{nq}({\bf x})$ where $V_{nq}({\bf x})$ denotes a non-quadratic potential and $V({\bf x}) = \frac{1}{2}| { B{\mathbf x}}|^2$. In this case, (\ref{eq:sch}) is the so-called magnetic Schr\"odinger equation (see \cite{Jin, Ostermann}). 
\end{itemize}
On can show that (\ref{eq:sch}) has the following two conserved quantities,
\begin{equation}
\begin{aligned}\label{eq:invariant}
&\text{(mass)} \quad M(t) =  \int_{\mathbb{R}^n}|\psi({\bf x}, t)|^2 \mathrm{d}{\bf x}, \\ 
&\text{(energy)} \quad E(t) = \int_{\mathbb{R}^n} \Big[ \frac{1}{2}|\nabla \psi|^2 +  V|\psi|^2 + f({\bf x}, |\psi|^2)|\psi|^2 - \text{Re}(i({B} {\bf x})\cdot \nabla \psi \, \psi^{*}) \Big]\mathrm{d}{\bf x},
\end{aligned}
\end{equation}
where $f^{*}$ and $\text{Re}(f)$ denote the conjugate and real part of the function $f$ respectively.

From the exact splitting presented above, we deduce a new splitting for \eqref{eq:sch}. 
This splitting is based on Strang composition of the quadratic and the non-quadratic parts.  Indeed, 
we first rewrite \eqref{eq:sch} as 
$$
i\partial_t \psi = -p^w \psi + f({\bf x}, |\psi|^2) \psi, 
$$
where $-p^w\psi := -\frac{1}{2}\Delta \psi- i ({ B}  {\bf x})\cdot \nabla \psi + V({\bf x})\psi$ denotes the quadratic part 
(in the sense of Section \ref{section1}) and $f({\bf x}, |\psi|^2) \psi$ denotes the non quadratic part (which can be nonlinear).  
Based on this formulation and on the fact that exact splitting have been derived for the quadratic part, 
we then propose the following splitting (ESQM method) 
\begin{equation}
\label{esqm2}
 ( \psi({\bf x}, t_n)\approx) \psi^n({\mathbf{x}}) \ \ =  \left( e^{-i\frac{\Delta t}{2}f({\mathbf x}, |\psi|^2)} e^{-i\Delta t p^w} e^{-i\frac{\Delta t}{2}f({\mathbf x}, |\psi|^2)}  \right)^n\psi_0({\mathbf{x}})
\end{equation}
where the computation of $e^{-i\Delta t p^w}$ is done using (\ref{eq:MS2}). 
Let us remark that in this Strang based splitting, each part can be solved exactly in time and high order composition methods 
can be easily used to derive arbitrary high order time integrator (see \cite{HLW}). It can be shown that $\|\psi^n\|_{\ell^2}$ 
is preserved by the numerical schemes proposed here. In Algorithm \ref{algo1}, 
we detail the different steps of the exact splitting \eqref{esqm2} 
involving the pseudo-spectral discretization in dimension $3$.

\begin{algorithm}[t]
\caption{Pseudo-spectral method for ESQM \eqref{esqm2}}
\label{algo1}
    \begin{algorithmic}[1]
        \REQUIRE $\psi^{0} = \psi^0_{| \mathbb{G}_1 \times \mathbb{G}_2 \times \mathbb{G}_3 }$
        \FOR{$n=0$ to $n_{final}-1$}
            \STATE $\psi^{(1)} = e^{-i \Delta t/2 \; f(g, |\psi_{g}^{n}|^2)}\psi^{n}$ 
            \STATE  $\psi^{(2)} = e^{- i\Delta t V^{(r)}_{\Delta t}(g)} \psi^{(1)}$            
            \STATE ${\psi}^{(3)} =  e^{-i\Delta t \omega_3 (L_{\Delta t, 31}g_1 + L_{\Delta t, 32} g_2)}\mathcal{F}_3\psi^{(2)}$
                 \STATE ${\psi}^{(4)} =  e^{-i\Delta t \omega_2 L_{\Delta t, 21}g_1}\mathcal{F}_2{\psi}^{(3)}$
                 \STATE ${\psi}^{(5)} =  e^{-i\Delta t a(\omega)}\mathcal{F}_1 {\psi}^{(4)}$
                 \STATE ${\psi}^{(6)} =  e^{-i\Delta t \omega_2 U_{\Delta t, 23}g_3}\mathcal{F}^{-1}_3 {\psi}^{(5)}$
            \STATE ${\psi}^{(7)} =  e^{-i\Delta t \omega_1 (U_{\Delta t, 12}g_2 + U_{\Delta t, 13}g_3)}\mathcal{F}^{-1}_2 {\psi}^{(6)}$
            \STATE $\psi^{(8)} = e^{-i \Delta t V^{(r)}_{\Delta t}(g) }\mathcal{F}_1^{-1} {\psi}^{(7)}$
            \STATE $\psi^{n+1} = e^{-i{\Delta t/2} \; f(g, |\psi_{g}^{(8)}|^2)} {\psi}^{(8)}$
            \ENDFOR    
        \ENSURE  $\psi^{n_{final}}$    
\end{algorithmic}
\end{algorithm}

\begin{remark}
Note that some optimizations can be performed in Algorithm \ref{algo1} by noticing the following rearrangement 
\begin{equation} \label{eq:strang2}
\begin{split}
\psi^{n_{final}}&= \left( e^{-i\frac{\Delta t}{2}f({\mathbf x}, |\psi|^2)} e^{i\Delta t (-V({\mathbf{x}} )+ \frac{\Delta}{2} - B{\mathbf{x}}\cdot \nabla )}e^{-i\frac{\Delta t}{2}f({\mathbf x}, |\psi|^2)}  \right)^{n_{final}}\psi^0\\
&=e^{-i\frac{\Delta t}{2}f({\mathbf x}, |\psi|^2)-i \Delta t V^{(\ell)}_{\Delta t}({\mathbf x})}   e^{ -t (U_{\Delta t} {\mathbf x})_1 \partial_{x_1}  } e^{ -t (U_{\Delta t} {\mathbf x})_2 \partial_{x_2}  }  e^{i t A_{\Delta t}(\nabla)}     e^{ -\Delta t (L_{\Delta t} {\mathbf x})_2 \partial_{x_2}  }  e^{ -\Delta t (L_t {\mathbf x})_3 \partial_{x_3}  } \\
& \left( e^{-i{\Delta t}f({\mathbf x}, |\psi|^2)-i \Delta t V^{(\ell)}_{\Delta t}({\mathbf x})-i \Delta t V^{(r)}_{\Delta t}({\mathbf x})}   e^{ -\Delta t (U_{\Delta t} {\mathbf x})_1 \partial_{x_1}  }
e^{ -\Delta t (U_{\Delta t} {\mathbf x})_2 \partial_{x_2}  }  e^{i \Delta t A_{\Delta t}(\nabla)}  \right. \\ 
&\left.   e^{ -\Delta t (L_{\Delta t} {\mathbf x})_2 \partial_{x_2}  }  e^{ -\Delta t (L_{\Delta t} {\mathbf x})_3 \partial_{x_3}  }  \right)^{n_{final}-1} e^{-i \Delta t V^{(r)}_{\Delta t}(\mathbf{x})-i\frac{\Delta t}{2}f({\mathbf x}, |\psi|^2)} \psi^0.
\end{split}
\end{equation}
\end{remark}

\subsection{Numerical results}
This section is devoted to applications of exact splittings to the Schr\"odinger equations \eqref{eq:sch} both in the 
two and three dimensional cases. 
We show higher accuracy and efficiency of the exact splitting by comparing it to other usual numerical methods 
proposed in the literature~\cite{wang}. 

As previously, the space discretization requires a truncated domain denoted by $[-R_1, R_1] \times \dots \times [-R_n, R_n]$.
We will consider a uniform grid with $N_j$ points per direction so that the mesh size are $2R_j/N_j$. 

\subsubsection{2D Schr\"odinger equations}
Firstly, we consider the application of the exact splitting in the two-dimensional case on the magnetic Schr\"odinger equation 
and on the rotating Gross-Pitaevskii equation.

\noindent{\bf 2D magnetic Schr\"odinger equation}\\
In this numerical experiment, the 2D magnetic Schr\"odinger equation  is considered~\cite{Jin},
\begin{equation}\label{eq:2Dsch}
i\epsilon \partial_t \psi({\bf x},t) = -\frac{\epsilon^2}{2} \Delta  \psi({\bf x},t) + i\epsilon {\bf A} \cdot \nabla \psi ({\bf x},t) + \frac{1}{2}{|\bf A|}^2\psi({\bf x},t),
\end{equation}
with $\epsilon = 1/32$ and where ${\bf x} = (x_1,x_2) \in \mathbb{R}^2$, ${\bf A} = \frac{1}{2}\transp{(-x_2, x_1)}$.
The initial condition is given by 
$$
\psi_0({\bf x}) = e^{-20(x_1-0.05)^2-20(x_2-0.1)^2} e^{i \sin(x_1) \sin(x_2)/\epsilon}. 
$$
The numerical parameters are chosen as follows: $N_1=N_2=256$, $R_1=R_2=3\pi$ and $\Delta t=0.3$. 
We shall compare three different splittings: 
\begin{itemize}
\item ESQM (see \eqref{esqm2} with $f=0$ and \eqref{eq:MS2}); this method is exact in time.  
\item ESR (see Appendix \ref{magschro_split}); this method is second order accurate in time. 
\item Strang (see Appendix \ref{magschro_split}); this method is second order accurate in time. 
\end{itemize}
Let us remark that ESR and Strang are two operator 
splittings which differ from the treatment of the rotation part  $\partial_t \psi = {\bf A} \cdot \nabla \psi$. 
Indeed, this part is solved exactly in the ESR method (this is the reason why we used the same name 
as in Section \ref{transport}) whereas a second order directional splitting is used to approximate 
it in the Strang method. We refer to Appendix \ref{magschro_split} for the details. 
Let us now write the operators which define the exact splitting ESQM. From \eqref{eq:MS2}, we have  
$$
e^{i\epsilon \Delta t(\frac{\Delta }{2} - \frac{1}{2\epsilon^2} |{\mathbf A}|^2) + \epsilon  \Delta t{\frac{\mathbf A}{\epsilon}}\cdot \nabla} = e^{-\epsilon  \Delta tv^{(\ell)}({\mathbf x})} e^{-\epsilon  \Delta t(U_{\Delta t}{\mathbf x})_1\partial x_1} e^{i\epsilon  \Delta ta(\nabla)} e^{-\epsilon  \Delta t(L_{\Delta t}{\mathbf x})_2\partial_{x_2}} e^{-\epsilon  \Delta tv^{(r)}({\mathbf x})}, 
$$
where $v^{(\ell)}({\mathbf x}) = {\mathbf x}^{\mathrm{T}}V_{\Delta t}^{(\ell)} {\mathbf x}, \;\; v^{(r)}({\mathbf x}) = {\mathbf x}^{\mathrm{T}}V_{\Delta t}^{(r)} {\mathbf x}$, and $a(\nabla) = \nabla  \cdot (A_{\Delta t} \nabla )$, $V_{\Delta t}^{(\ell)}, L_{\Delta t},  U_{\Delta t}, V_{\Delta t}^{(r)}, A_{\Delta t}$ being $2$x$2$ matrices.  
The corresponding coefficients for ESQM are  
\begin{align*}
&A_{\Delta t} \simeq \begin{pmatrix} 0.503532819405421 & -0.074439184790650 \\
-0.074439184790650 & 0.503784060194312
 \end{pmatrix},\\
 &L_{\Delta t} \simeq  \begin{pmatrix} 0 & 0  \\
 -16.121089926218119 & 0 
 \end{pmatrix}, \quad 
 U_{\Delta t} \simeq \begin{pmatrix} 0 & 15.761077688604765  \\
0& 0 
 \end{pmatrix}, \\
 &V^{(\ell)}_{\Delta t} \simeq  \begin{pmatrix} 128.9687194097432&0
  \\
0 & -0.0000000000018
 \end{pmatrix},\\
 &V^{(r)}_{\Delta t} \simeq  \begin{pmatrix} 2.8800979009085&-19.0564313064080 \\
-19.0564313064080 & 126.0886215088400
 \end{pmatrix}.
 \end{align*}
First of all, we show the contour plots of $|\psi({\bf x}, t)|^2$ at the final time $t=300$ in  Figure \ref{fig:2D_plot} 
for the three methods. Second, the time history energy error (defined by \eqref{eq:invariant}) 
obtained by the three methods is presented on Figure \ref{fig:2Denergy}. 
 As \eqref{eq:2Dsch} is a quadratic equation, by Theorem \ref{thm_splt_MS2}, the ESQM method 
 gives the exact solution (if neglecting space error). From Figure \ref{fig:2Denergy}, we can see 
 that its energy error is at machine precision level, which is much smaller than the energy errors of Strang and ESR 
 (for which the energy errors oscillate around a constant).   
 Specifically, as the rotation velocity of Strang is not correct (see \cite{JC2}), 
 we can see in Figure \ref{fig:2D_plot} that the contour plot obtained by Strang is not good. 
 For ESR, even if the rotation velocity is right and as such the shape of the solution has the correct orientation, 
 some error are clearly observed. 

\begin{figure}[htbp]
\center{
\includegraphics[scale=0.7]{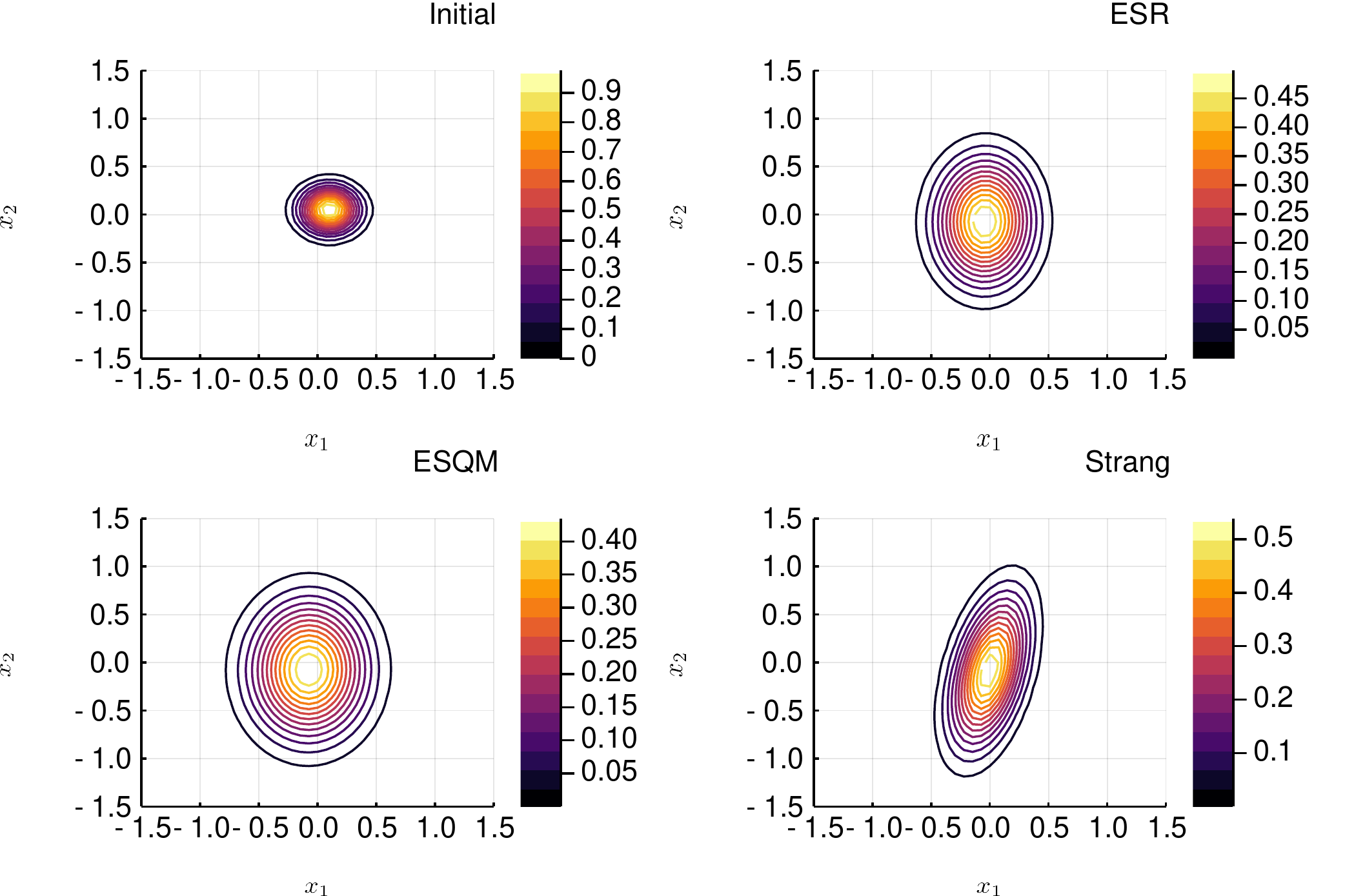}
}
\caption{Contour plot of initial density, and the contour plots of ESQM, ESR and Strang splitting method at $t = 300$ 
with  $\Delta t = 0.3$,  $N_1=N_2=256$ and $\epsilon = \frac{1}{32}$ for 2D magnetic Schr\"odinger problem.}\label{fig:2D_plot}
\end{figure}

\begin{figure}[htbp]
\center{
\includegraphics[scale=0.36]{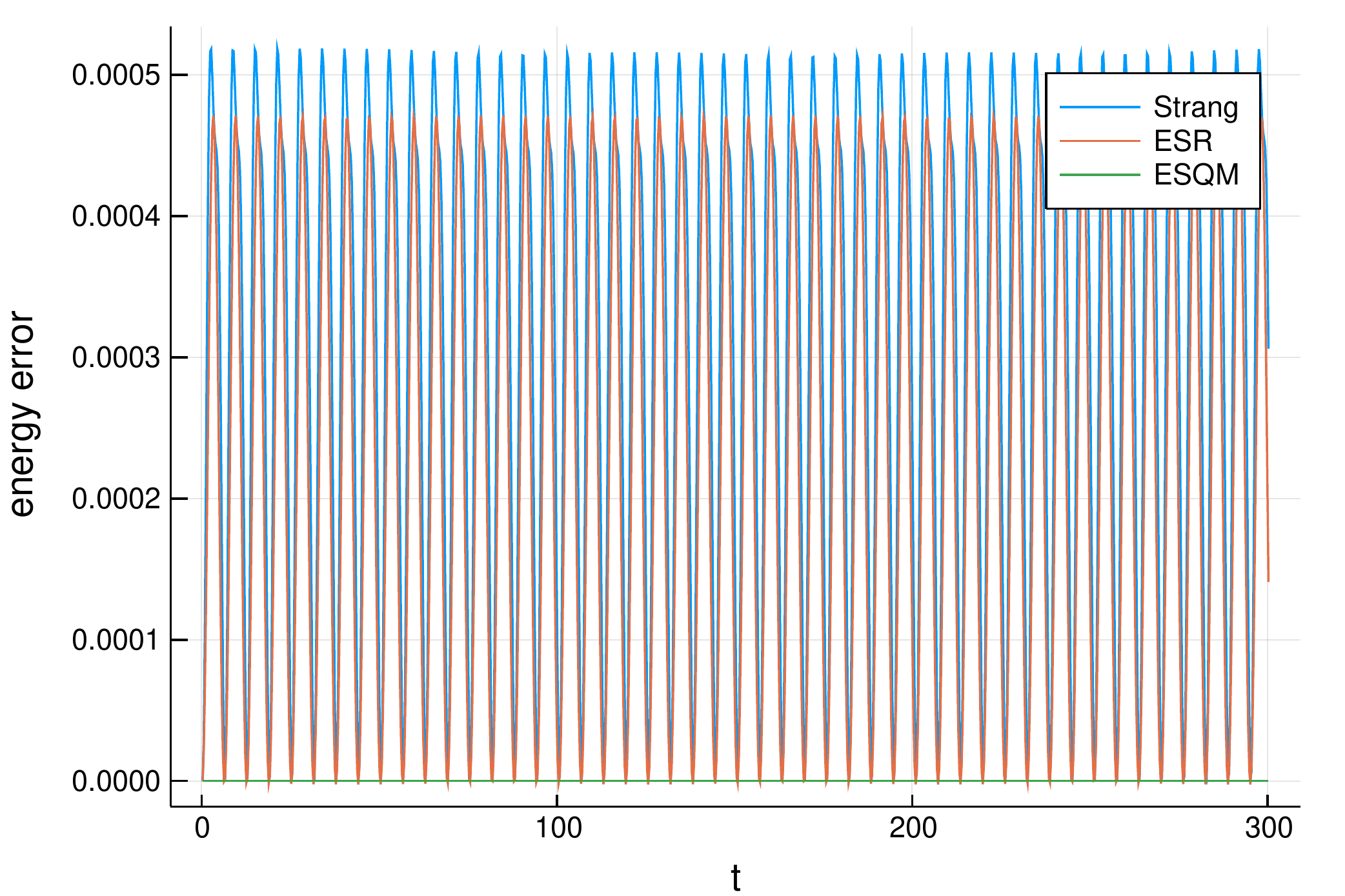}
}
\caption{Time evolution of energy error of ESR, ESQM and Strang splitting method with  $\Delta t = 0.3$,  $N_1=N_2=256$ and 
$\epsilon = \frac{1}{32}$ for 2D magnetic Schr\"odinger equation.}\label{fig:2Denergy}
\end{figure}

\noindent{\bf 2D rotating Gross-Pitaevskii equation}\\
We now consider the dynamics of rotating Bose-Einstein condensates, which is described by the 
macroscopic wave function $\psi({\bf x}, t)$ (${\bf x} = (x_1, x_2)\in \mathbb{R}^2, t\geq 0$) 
solution to the following rotating Gross-Pitaevskii equation (GPE) (see \cite{Bao, wang}) 
\begin{equation}
\begin{aligned}\label{eq:gpe}
&{i}\partial_t \psi({\bf x}, t) = -\frac{1}{2}\Delta \psi({\bf x},t) + V({\bf x}) \psi({\bf x},t) + \beta | \psi |^2 \psi({\bf x},t) - \Omega L_{x_3} \psi({\bf x},t), \;\;\; \psi({\bf x},0) = \psi_0({\bf x})\\
\end{aligned}
\end{equation}
where $L_{x_3} = - i (x_1\partial_{x_2} - x_2 \partial_{x_1})$ is the $x_3$-component of the angular momentum, $\Omega$ is the angular speed of the laser beam, $\beta$ is a constant characterizing the particle interactions and $V({\bf x})$ 
denotes the external harmonic oscillator potential 
\begin{equation}\label{eq:poten}
V({\bf x}) = \frac{1}{2}(\gamma_{x_1}^2 x_1^2 + \gamma_{x_2}^2 x_2^2), 
\end{equation} 
with constants $\gamma_{x_1} > 0$ and $\gamma_{x_2} > 0$.

In addition to the mass and energy preservations (\ref{eq:invariant}), 
 the expectation of angular momentum is also conserved when $\gamma_{x_1} = \gamma_{x_2}$ 
\begin{equation}
\label{angular}
Lz(t) := \int_{\mathbb{R}^2} \psi^*({\bf x}, t) L_{x_3}\psi({\bf x}, t) d{\bf x} = Lz(0).
\end{equation}
We are also interested in the time evolution of condensate widths and mass center defined as follows,
\begin{equation}
\begin{aligned}
&\text{condensate widths}:  \, S_{\alpha}(t) = \sqrt{\int_{\mathbb{R}^2} \alpha^2 |\psi({\bf x}, t)|^2 d{\bf x}}, \quad \alpha = x_1, x_2,\\
&\text{mass center}: {\bf x}_{c} (t)= \int_{\mathbb{R}^2} {\bf x} |\psi({\bf x}, t)|^2d{\bf x}.
\end{aligned}
\label{diag_gpe}
\end{equation}

For the two dimensional rotating GPE \eqref{eq:gpe}, our first numerical test is the so-called dynamics of a stationary state with a shifted center~\cite{Bao}. We take $ \gamma_{x_1} = \gamma_{x_2} = 1, \beta = 100$ in \eqref{eq:gpe} and 
the initial condition is taken as 
$$
\psi_0({\bf x}) = \phi_e({\bf x} - {\bf x}_0), 
$$
where $\phi_e({\bf x}) $ is a  ground state computed numerically from \cite{zhang} and ${\bf x}_0 = \transp{(1,1)}$.  The numerical parameters 
are chosen as follows: $\Delta t = 0.001$ and the spatial domain $[-8,8]^2$ is discretized using $N_1=N_2=256$ points.  

As in the magnetic Schr\"odinger case, we will consider the following three methods to 
approximate \eqref{eq:gpe}
\begin{itemize}
\item ESQM (see \eqref{esqm2} with $f({\bf x}, |\psi|^2) = \beta |\psi|^2$ and \eqref{eq:MS2}); this method is second order accurate in time.  
\item ESR (see Appendix \ref{esr_gpe}); this method is second order accurate in time. 
\item BW from \cite{Bao} (see Appendix \ref{bw}); this method is second order accurate in time. 
\end{itemize}
Concerning  ESQM, we then have to define from \eqref{eq:MS2} how the quadratic part $p^w:=(i/2) \Delta  - \Omega L_{x_3} - iV({\mathbf{x}})$ 
of the nonlinear equation \eqref{eq:gpe} is split. This is done as follows (the two cases $\Omega = -0.5$ and $\Omega = 0$ 
are considered) 
$$
e^{i\Delta t p^w} = e^{- \Delta t v^{(\ell)}({\mathbf x})} e^{-  \Delta t(U_{\Delta t}{\mathbf x})_1\partial_{x_1}} e^{i  \Delta ta(\nabla)} e^{-  \Delta t(L_{\Delta t}{\mathbf x})_2\partial_{x_2}} e^{- \Delta tv^{(r)}({\mathbf x})}, 
$$
where $v^{(\ell)}({\mathbf x}) = {\mathbf x}^{\mathrm{T}}V_{\Delta t}^{(\ell)} {\mathbf x}, \;\; v^{(r)}({\mathbf x}) = {\mathbf x}^{\mathrm{T}}V_{\Delta t}^{(r)} {\mathbf x}$, and $a(\nabla) = \nabla  \cdot (A_{\Delta t} \nabla )$, $V_{\Delta t}^{(\ell)}, L_{\Delta t}, U_{\Delta t},  V_{\Delta t}^{(r)}, A_{\Delta t}$ being $2$x$2$ matrices.  
In the case $\Omega = -0.5$, we have 
\begin{align*}
&A_{\Delta t} \simeq \begin{pmatrix} 0.499999979166481 & 0.000249999948070 \\
0.000249999948070 & 0.499999979166811
 \end{pmatrix},\\
 &L_{\Delta t} \simeq  \begin{pmatrix} 0 & 0  \\
 0.500000041666386 & 0 
 \end{pmatrix},\quad 
 U_{\Delta t} \simeq  \begin{pmatrix} 0 & -0.499999916666976  \\
 0 & 0 
 \end{pmatrix},\\
 &V^{(\ell)}_{\Delta t} \simeq  \begin{pmatrix} 0.312500037673140 &0 \\
0& 0.187500011883192
 \end{pmatrix},\\
 &V^{(r)}_{\Delta t} \simeq  \begin{pmatrix} 0.187500043056181&0.000062500002332 \\
0.000062500002332 & 0.312500006345821
 \end{pmatrix}.
 \end{align*}
and in the case $\Omega = 0$, we have 
\begin{align*}
&A_{\Delta t} \simeq \begin{pmatrix} 0.499999916666670 & 0 \\
0 & 0.499999916666676
 \end{pmatrix},\\
 & L_{\Delta t} \simeq  \begin{pmatrix} 0 & 0  \\
 0 & 0 
 \end{pmatrix},\quad  
  U_{\Delta t} \simeq  \begin{pmatrix} 0 & 0  \\
 0 & 0 
 \end{pmatrix},\\
 &V^{(\ell)}_{\Delta t} \simeq  \begin{pmatrix} 0.250000020830132 & 0 \\
0 & 0.250000020802363 
 \end{pmatrix},\\
 &V^{(r)}_{\Delta t} \simeq  \begin{pmatrix} 0.250000020836539&0 \\
0 & 0.250000020864292
 \end{pmatrix}.
 \end{align*} 
We first validate our ESQM approach by plotting the time history of mass center and condensate widths \eqref{diag_gpe}, 
and angular momentum expectation \eqref{angular} in Figure \ref{fig:evolution}. 
From \cite{Bao}, the mass center is known to be periodic, and the period is equal to   $2\pi$ (resp. $4\pi$) 
when $\Omega = 0$ (resp. $\Omega = -0.5$). As observed in the numerical results, 
the numerical method preserved accurately this property.

In the sequel, we compare  ESQM to ESR and BW. 
Let us remark that ESQM only needs $4$ FFT for each time step whereas BW needs $6$ and ESR needs $10$. 
As the FFT calls are the most consuming part of the three methods, ESQM is the most efficient and we then 
have to check its accuracy. 
The energy error \eqref{eq:invariant} and angular momentum expectation error \eqref{angular}  
of the three methods (ESR, ESQM and BW) are presented in Figure \ref{fig:energyerror1} for the case that $\Omega = -0.5$ 
and for different time steps.   
First, we notice that the three methods are second order accurate regarding the energy, as expected. However, the error constant 
is smaller for ESQM which is due to the fact that the linear part is solved exactly. In particular, 
the advantage of ESQM is more obvious when nonlinear parameter $\beta$ is smaller since the nonlinear part is less important 
and the exact treatment of the quadratic part in ESQM make it better. 
For the angular moment expectation conservation, we can see that BW is still second order in time, whereas 
ESR and ESQM are close to the machine precision independently of $\beta$. 
The reason is that angular moment expectation is conserved by the solution of each subsystem in ESQM and ESR 
(see \cite{Bao} for more details).  
Now in Figure \ref{fig:cost_magic_rotation}, we are interested in the computational costs of the three methods (ESR, ESQM, and BW) as a function of the number of grid points $N_1 \times N_2$ in space, by running $100$ iterations. 
In addition to its accuracy, one observes that ESQM is the most efficient. 
As mentioned before, the computational cost comes from the number of FFTs required in each method.  

To end this part, we focus on a second numerical experiment where the time evolution of a ground state 
is studied by changing the corresponding potential initially as \cite{Besse, review_3}. 
Now, the parameters are $\beta = 1000$, $\Omega = 0.9$, the potential is given by (\ref{eq:poten}) with 
$\gamma_{x_1} = 1.05, \gamma_{x_2} = 0.95$.  The initial condition is the ground state corresponding to the isotropic 
potential  $V({\mathbf x}) = |{\mathbf x}|^2/2$, $\beta = 1000$, and $\Omega = 0.9$,  generated  using the Matlab toolbox GPELab\footnote{http://gpelab.math.cnrs.fr}~\cite{Xavier, Xavier2}. In this numerical test, we only run ESQM 
and consider the numerical parameters as follows: the spatial grid 
is defined by $[-8,8]^2$ and $N_1=N_2=128$ whereas the time step size is $\Delta t=10^{-3}$. The coefficients 
for ESQM in \eqref{eq:MS2} are given by 
\begin{align*}
&A_{\Delta t} \simeq \begin{pmatrix}
0.500000110624718  &-0.000449999864087   \\
  -0.000449999864087 &  0.500000127291716   
 \end{pmatrix}, \\  
 & 
 L_{\Delta t}\simeq \begin{pmatrix}
                 0     &              0     \\              
  -0.900000273000100          &         0                   
 \end{pmatrix},         
U_{\Delta t}\simeq \begin{pmatrix}
       0 &  0.899999484000032 \\  
                   0            &       0  
\end{pmatrix}, \\  
 &
V^{(\ell)}_{\Delta t}\simeq \begin{pmatrix} 
   0.478125137872035        &           0             \\      
                   0  & 0.023124993175043                   
\end{pmatrix}, \\  
 &    
V^{(r)}_{\Delta t}\simeq \begin{pmatrix}
0.073125336336629 & -0.000364499913716\\   
  -0.000364499913716 &  0.428124781225271  
\end{pmatrix}. 
\end{align*}
The numerical results are displayed in Figure \ref{fig:evolution_groundstate} where the solution 
is plotted for different times ($t=0, 1.5, 3, 4$). These results are in very good agreement with those obtained in the literature \cite{Besse, review_3}. We also present in Figure \ref{fig:dynamic_error} the time evolution of the energy error, from which we can see that energy conservation is very good (about $10^{-7}$).


\begin{figure}[htbp]
\center{
\includegraphics[scale=0.7]{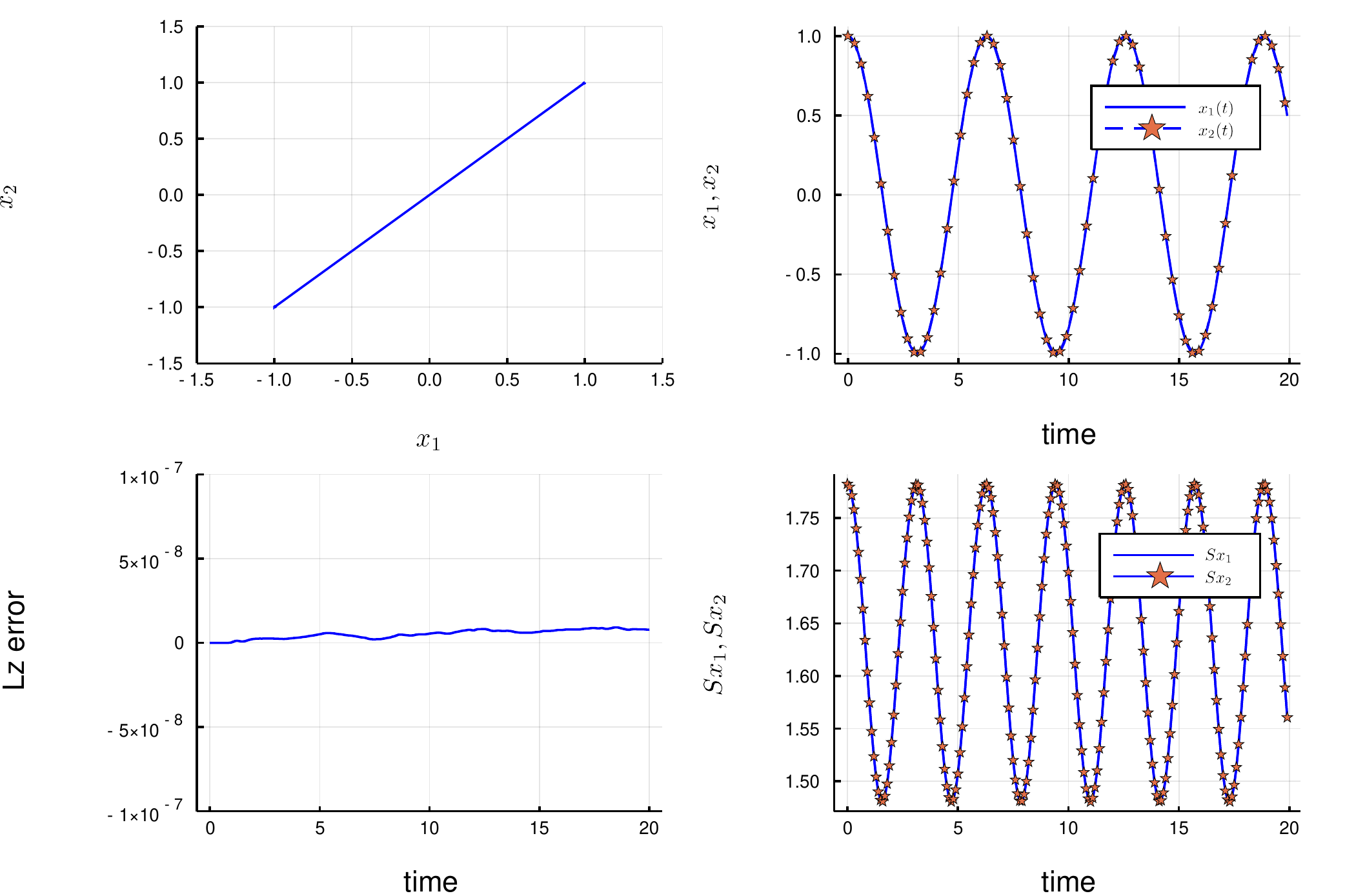}\\
\includegraphics[scale=0.7]{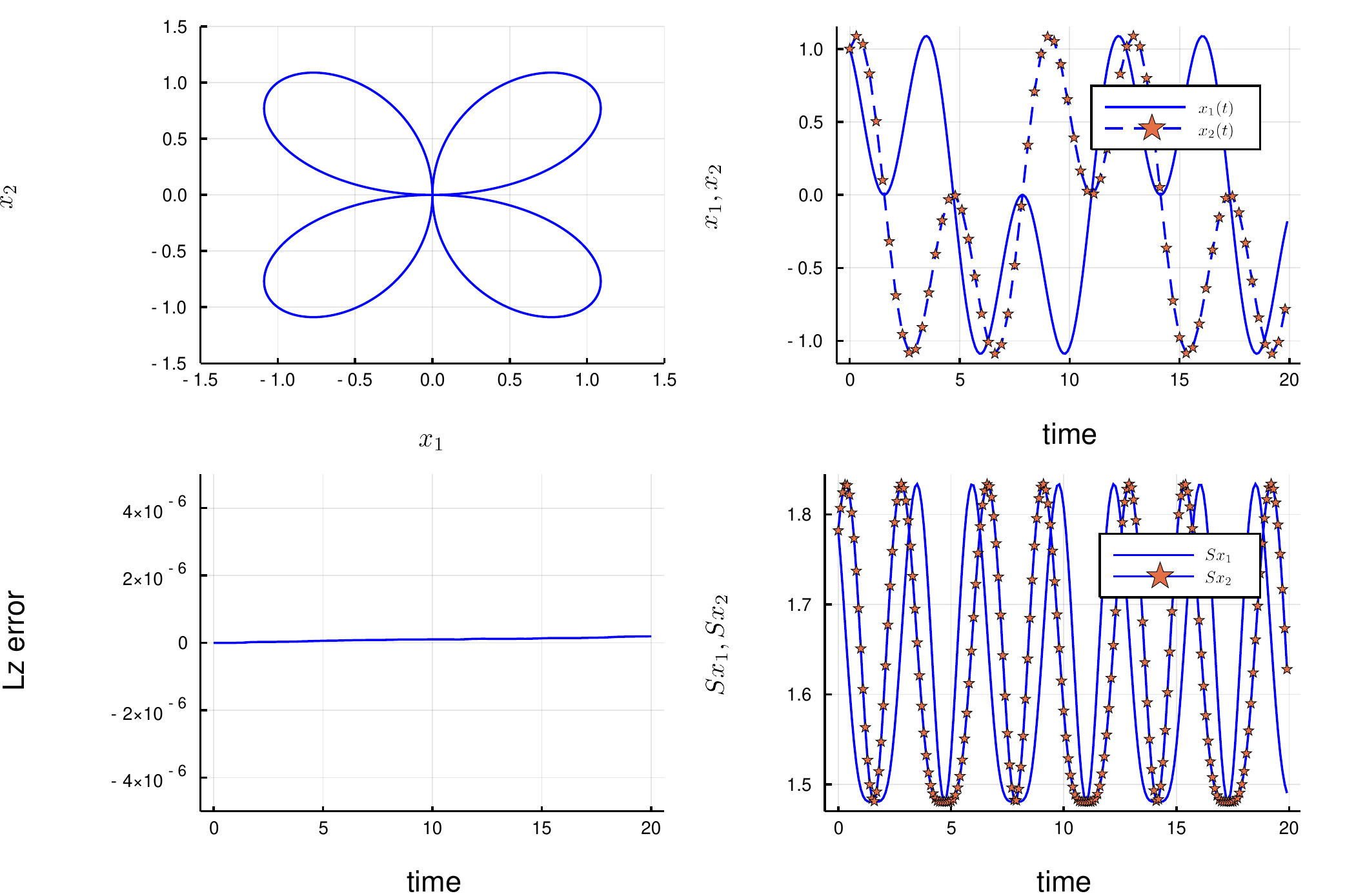}
}
\caption{Time evolution of mass center, coordinates of mass center, error on angular momentum expectation, and condensate widths by ESQM when $\Omega = 0$ (top four figures) and $\Omega = -0.5$ (bottom four figures).}
\label{fig:evolution}
\end{figure}


\begin{figure}[htbp]
\center{
\subfigure[]{\includegraphics[scale=0.35]{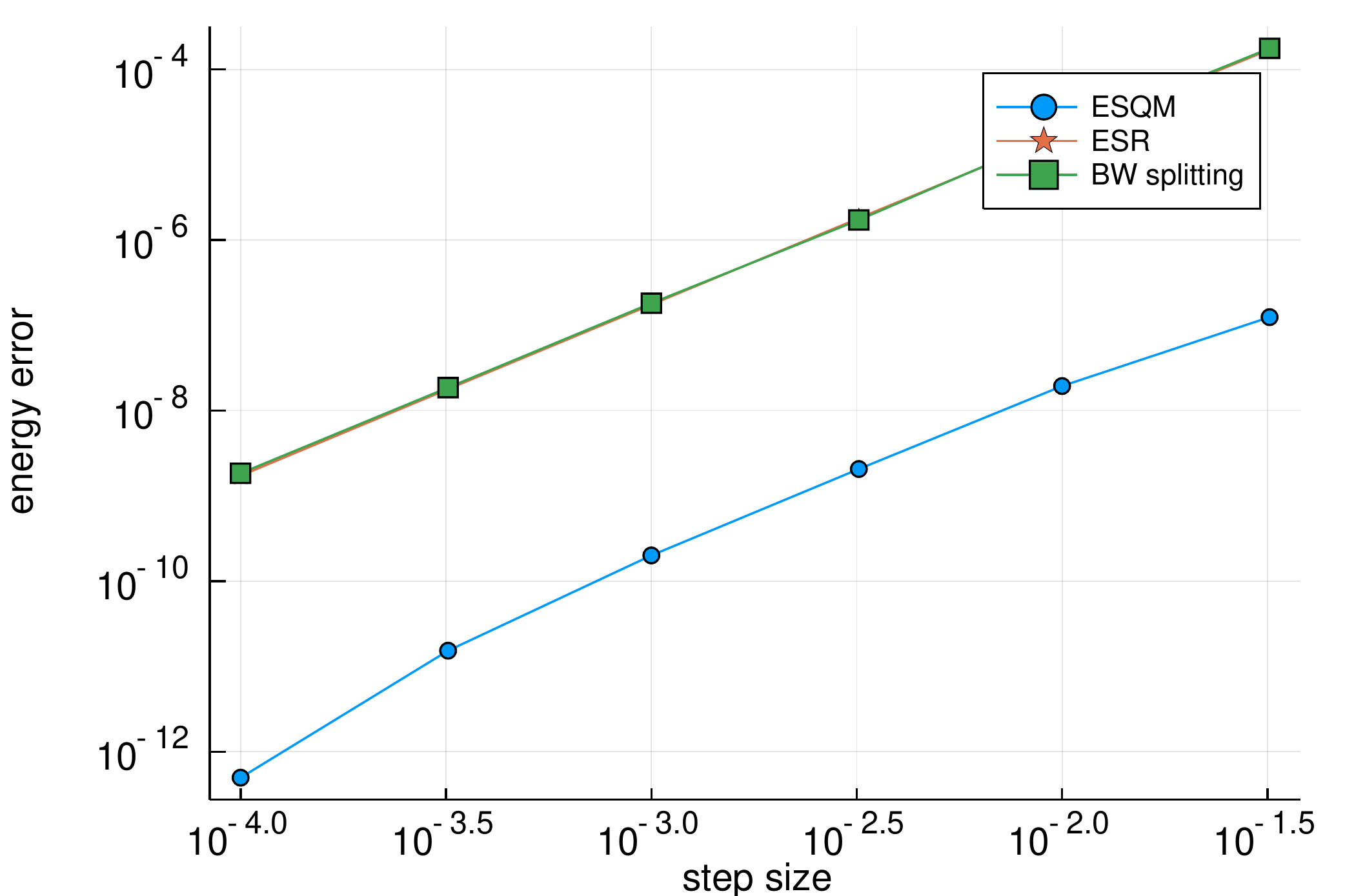}}
\subfigure[]{\includegraphics[scale=0.35]{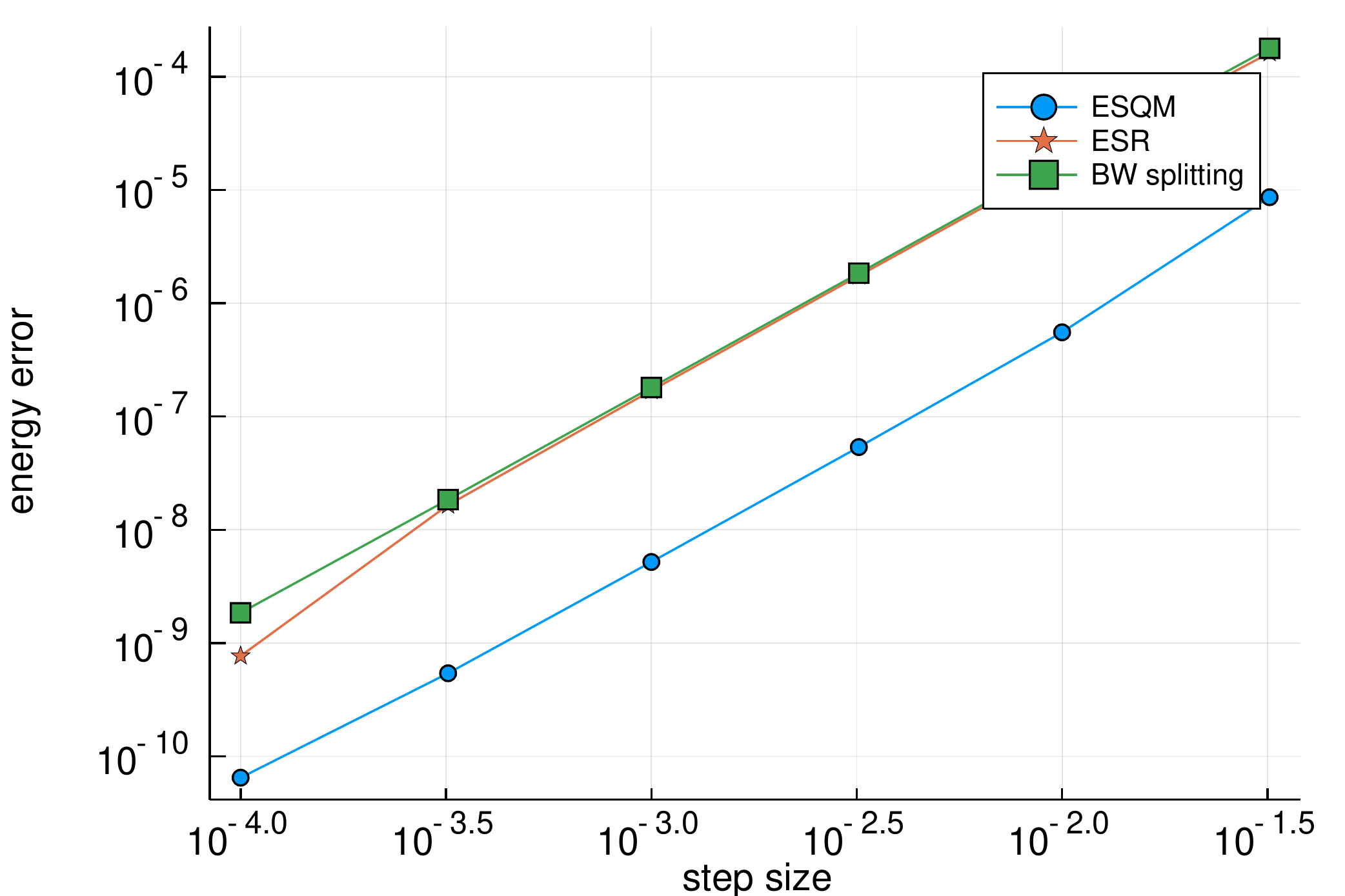}}\\
\subfigure[]{\includegraphics[scale=0.35]{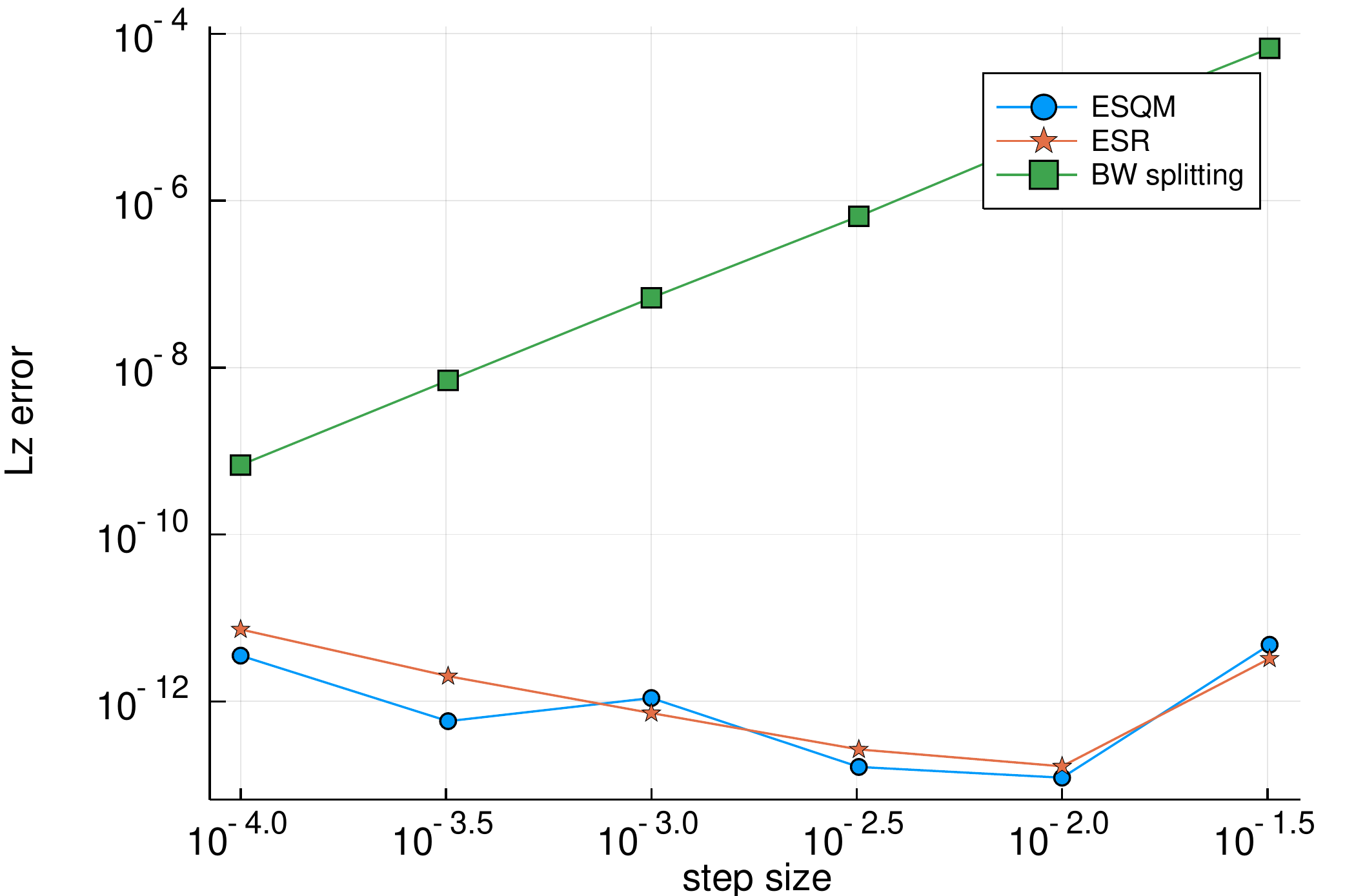}}
\subfigure[]{\includegraphics[scale=0.35]{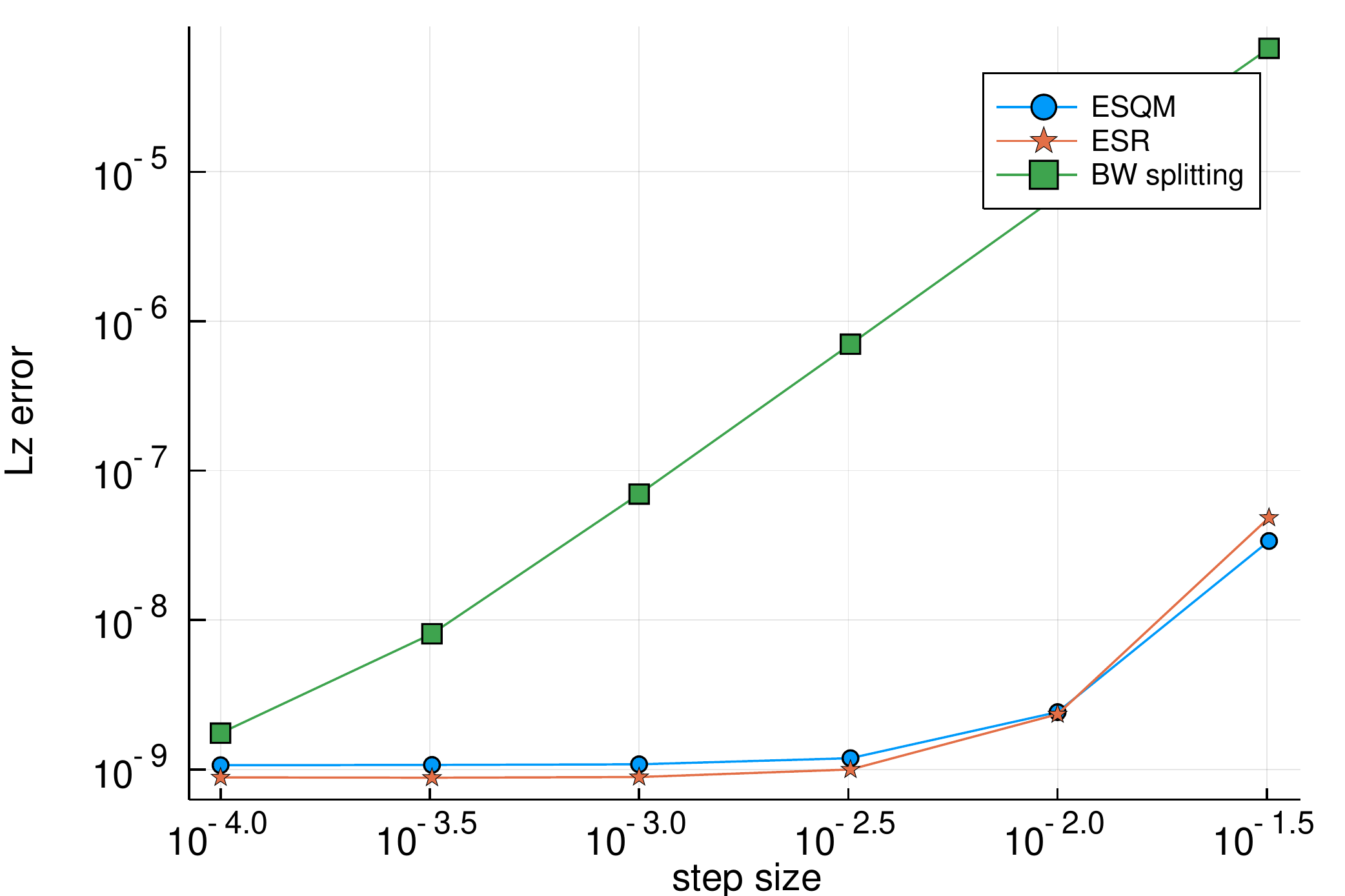}}
}
\caption{Energy error and angular momentum expectation error  (semi-$\log_{10}$ scale) 
as a function of the step size for the three methods ESQM, ESR and BW at $t=1$ for 
$\Omega = -0.5$, when $\beta = 5$,  (a,c), $\beta = 100$, (b,d).}
\label{fig:energyerror1}
\end{figure}

\begin{figure}[htbp]
\center{
\includegraphics[scale=0.4]{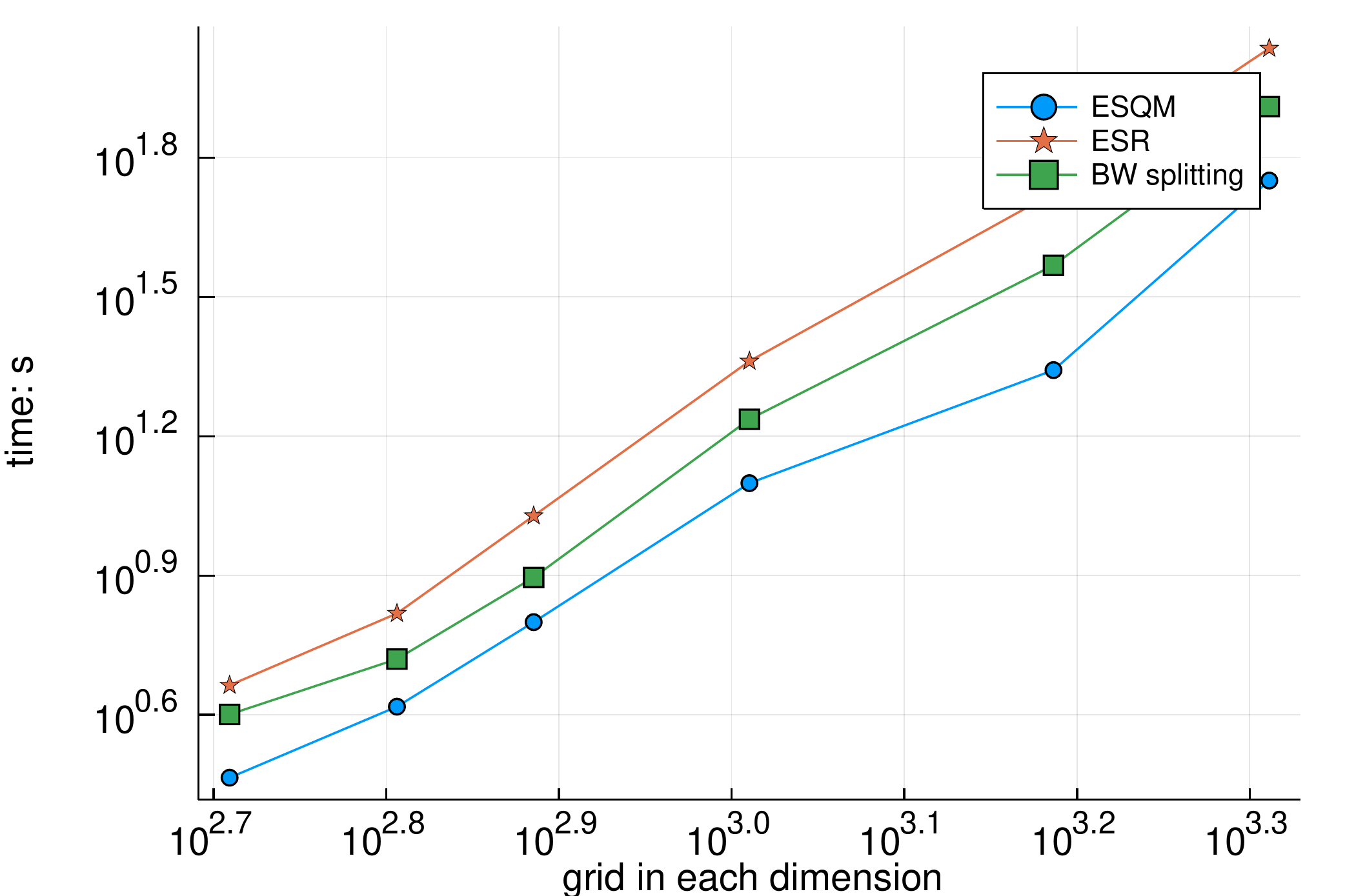}
}
\caption{Comparisons of computational costs between ESQM, ESR, and BW by running $100$ steps for rotating GPE \eqref{eq:gpe}.}
\label{fig:cost_magic_rotation}
\end{figure}


\begin{figure}[ht]
\center{
\includegraphics[scale=0.45]{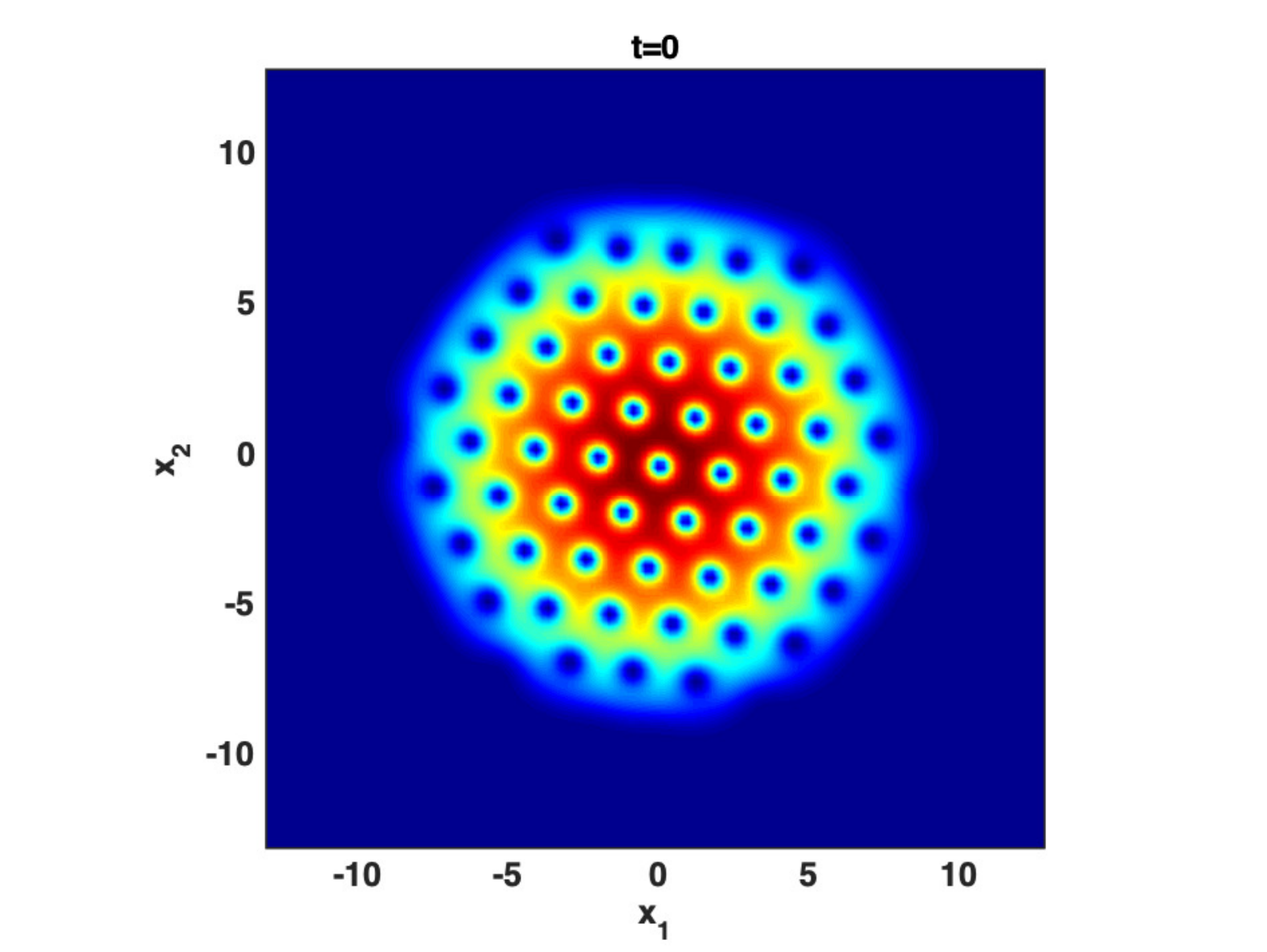}
\includegraphics[scale=0.45]{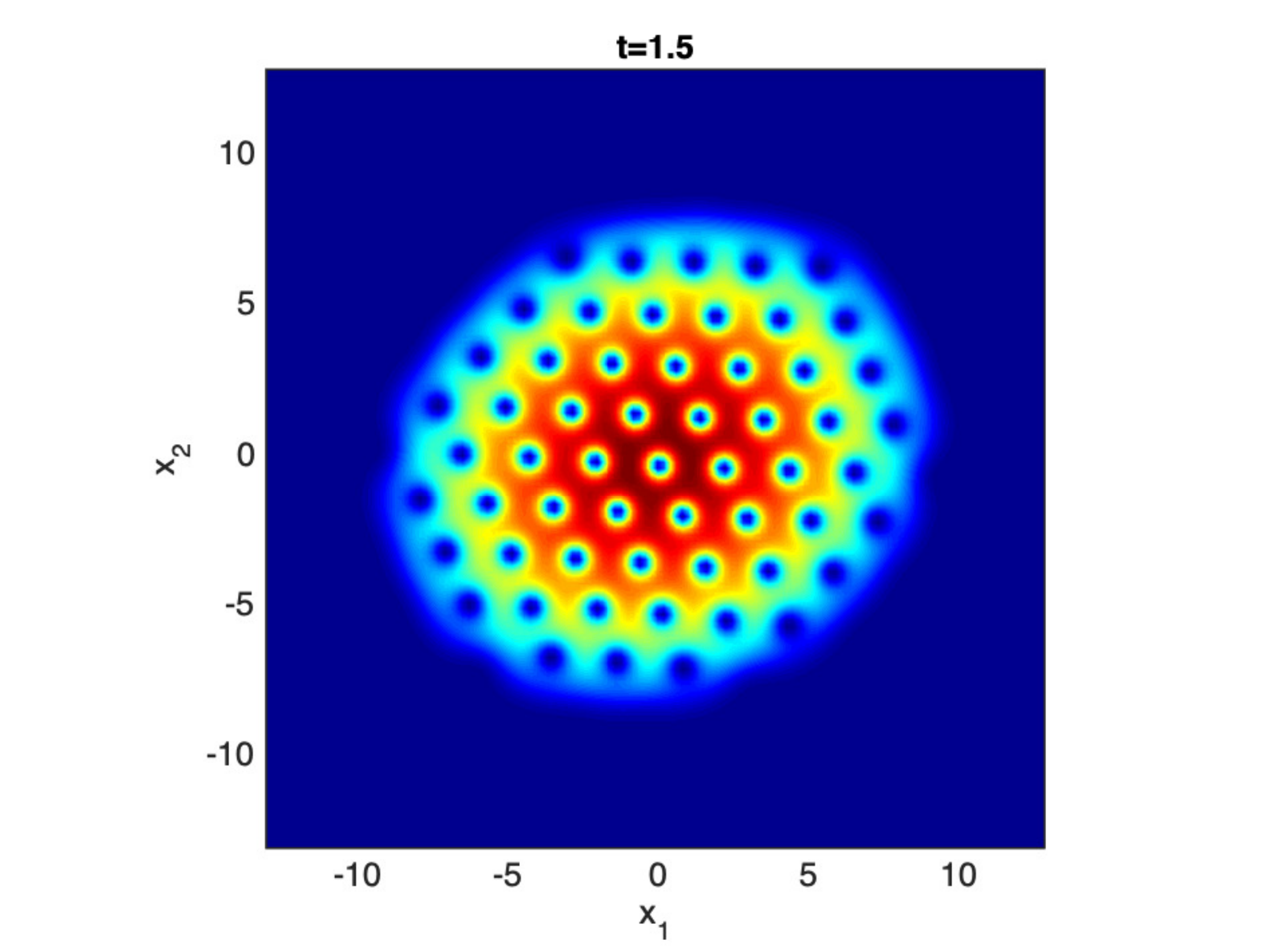}
\includegraphics[scale=0.45]{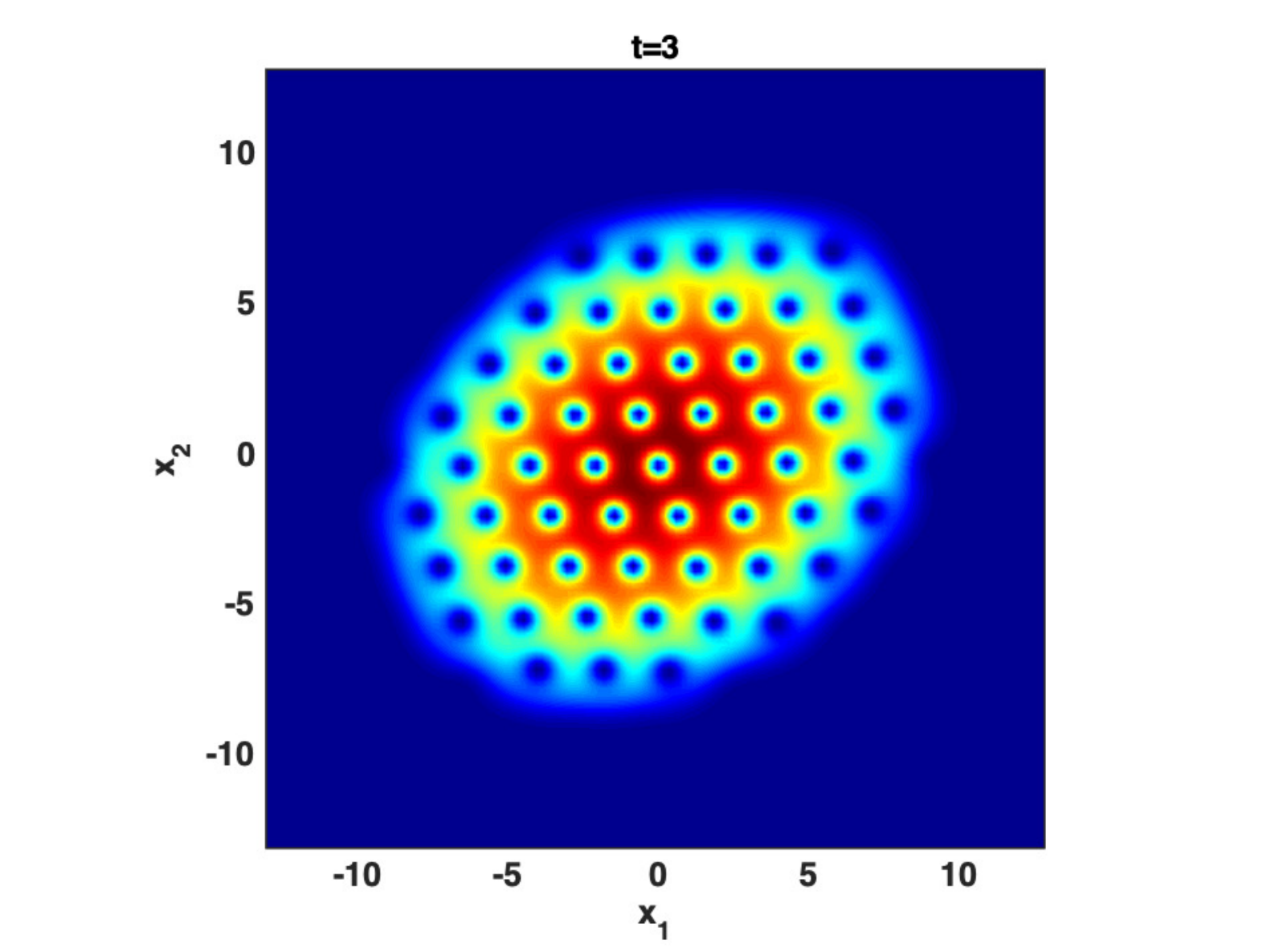}
\includegraphics[scale=0.45]{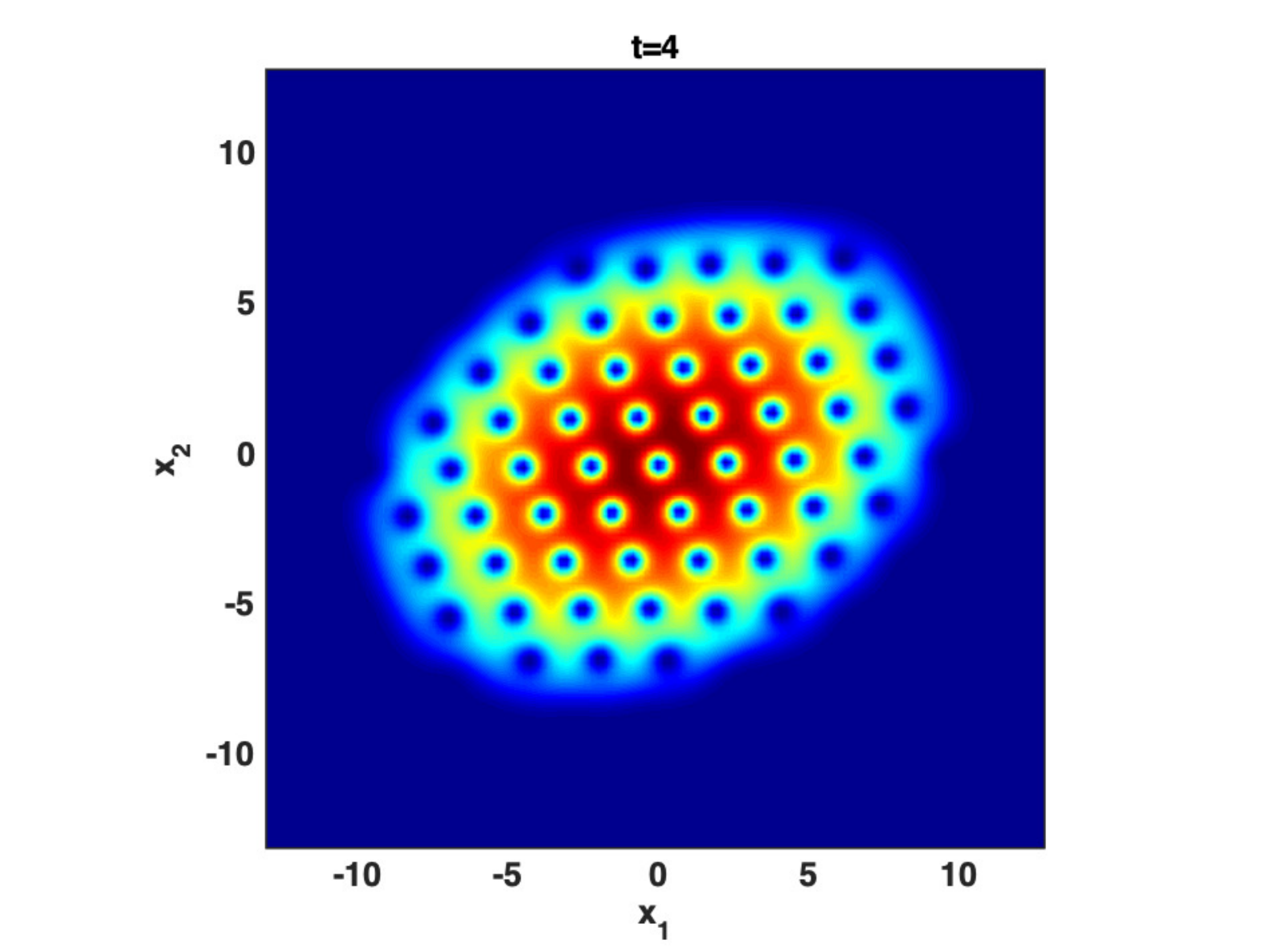}
}
\caption{$\beta = 1000, \Omega = 0.9$. Time evolution of the solution of \eqref{eq:gpe} by changing the potential initially.}\label{fig:evolution_groundstate}
\end{figure}

\begin{figure}[ht]
\center{
\includegraphics[scale=0.45]{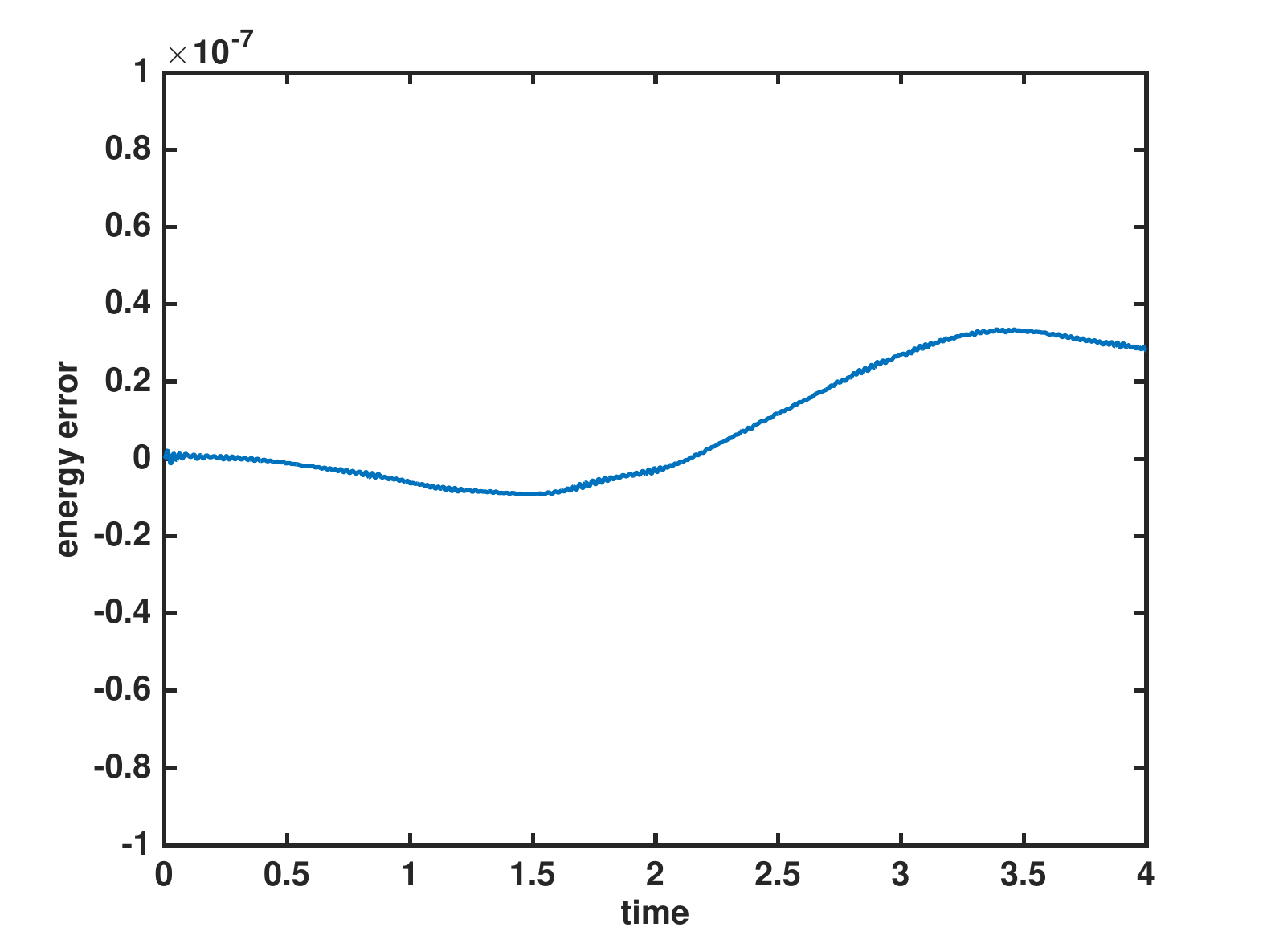}
}
\caption{$\beta = 1000, \Omega = 0.9$. Time evolution of the energy of \eqref{eq:gpe} by changing the potential initially.}\label{fig:dynamic_error}
\end{figure}
%

\subsubsection{3D Schr\"odinger equations}
In this section, 3D Schr\"odinger equations are considered through two cases: 
$(i)$ a quadratic Schr\"odinger equation is constructed specifically such that the solution is periodic in time; 
$(ii)$ a magnetic Schr\"odinger equation with a non-quadratic potential (see \cite{Ostermann}).

\noindent{\bf 3D time-periodic quadratic linear Schr\"odinger equation}\\
For (\ref{eq:sch}), we consider $f=0$ and 
\begin{equation}
\label{eq:3601}
B = \frac{\pi}3\begin{pmatrix} 0 & -1 & 1 \\ 1 & 0 & -1 \\ -1 & 1 & 0 \end{pmatrix} \quad \mbox{ and }\quad V({\bf x}) = \frac{\pi^2}9 \transp{{\bf x}}\begin{pmatrix} \lambda_1  \\& \lambda_2 \\ & & \lambda_3 \end{pmatrix} {\bf x}, 
\end{equation}
where ${\bf x}=(x_1, x_2, x_3)$ and $(\lambda_1, \lambda_2, \lambda_3)$ are the roots of the polynomial 
$Q(X)= 7200 X^3 - 72196 X^2 + 222088 X - 216341$, i.e.  
\begin{equation}\label{eq:3602}
\begin{pmatrix} \lambda_1  \\ \lambda_2 \\ \lambda_3 \end{pmatrix} \simeq \begin{pmatrix}  2.27017996551810 \\2.53418020791380 \\5.22286204879033\end{pmatrix}.
\end{equation}
In this case, the period of this system is $T=360$ (see in Appendix \ref{appendixc} for the proof) and the initial condition is  
\begin{align}
\psi_0(x_1,x_2,x_3) = \left( \frac{2}{\pi} \right)^3 e^{-x_1^2} e^{-x_2^2}e^{-(x_3-1)^2} + i \left( \frac{2}{\pi} \right)^3 e^{-x_1^2} e^{-(x_2+1)^2}e^{-(x_3-1)^2}. 
\end{align}
The numerical parameters are chosen as: the spatial domain $[-8,8]^3$ is discretized by $N_1=N_2=N_3=96$ points, 
the time step is $\Delta t = 0.2$, and the final time is $t = 720$ which corresponds to two periods. 
We will consider two different methods: 
\begin{itemize}
\item ESQM from (\ref{eq:MS2}) whose coefficients  are listed in Appendix \ref{sec:3D coefficients}; the method is exact in time. 
\item Strang (see in Appendix \ref{3dtime_split}); the method is second order accurate in time. 
\end{itemize}

In Figures \ref{fig:MS2} and \ref{fig:strang360}, the time evolution of 
$\psi(t, 0,0,0)$ (real and imaginary parts) are presented by using ESQM and Strang respectively. 
We also plot the difference $\psi(t\in [T, 2T], 0,0,0)-\psi(t\in [0, T], 0,0,0)$ 
(real and imaginary parts) which should be zero since the solution is time periodic of period $T=360$. 
We can see that with ESQM, the period is nicely preserved (up to $10^{-13}$) 
in spite of the fact that the time history of the solution is quite complicated. 
However, one can observe in Figure \ref{fig:strang360} that the conclusion is not the same for Strang: 
its error  is too large to identify the period. In Figure \ref{fig:error360}, 
the time history of energy error is plotted for both ESQM and Strang methods. Clearly, Strang produces large errors  
whereas the error from ESQM is very small (only due to the space approximation). 
Concerning the computational cost,  $6$ FFT (or inverse) are required for each time step for ESQM whereas  
Strang needs $15$ FFT (or inverse). 
Finally, some contour plots of the solution (at time $t=360$ and the third spatial direction $x_3$ is fixed to $0$) 
obtained by ESQM and Strang are presented at Figure \ref{fig:contour360}. We expect  $\psi(t=360, x_1, x_2, 0)$ 
to be very close to the initial condition since the solution is $360$ periodic in time. 
Even if ESQM gives very accurate results, one can see in Figure \ref{fig:contour360} that the result obtained 
by Strang is rather different. 
\begin{figure}[htbp]
\subfigure[]{\includegraphics[scale=0.45]{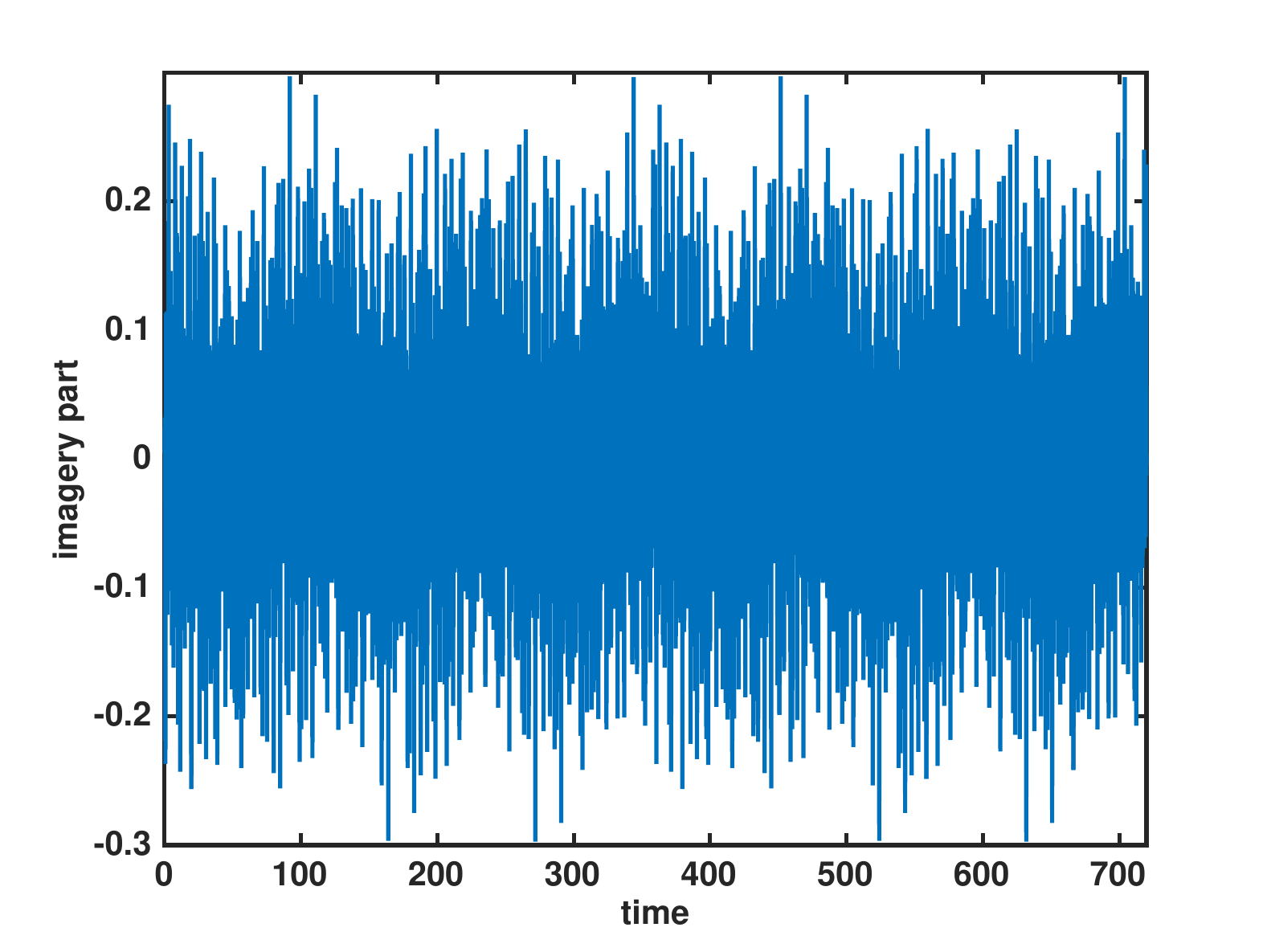}}
\subfigure[]{\includegraphics[scale=0.45]{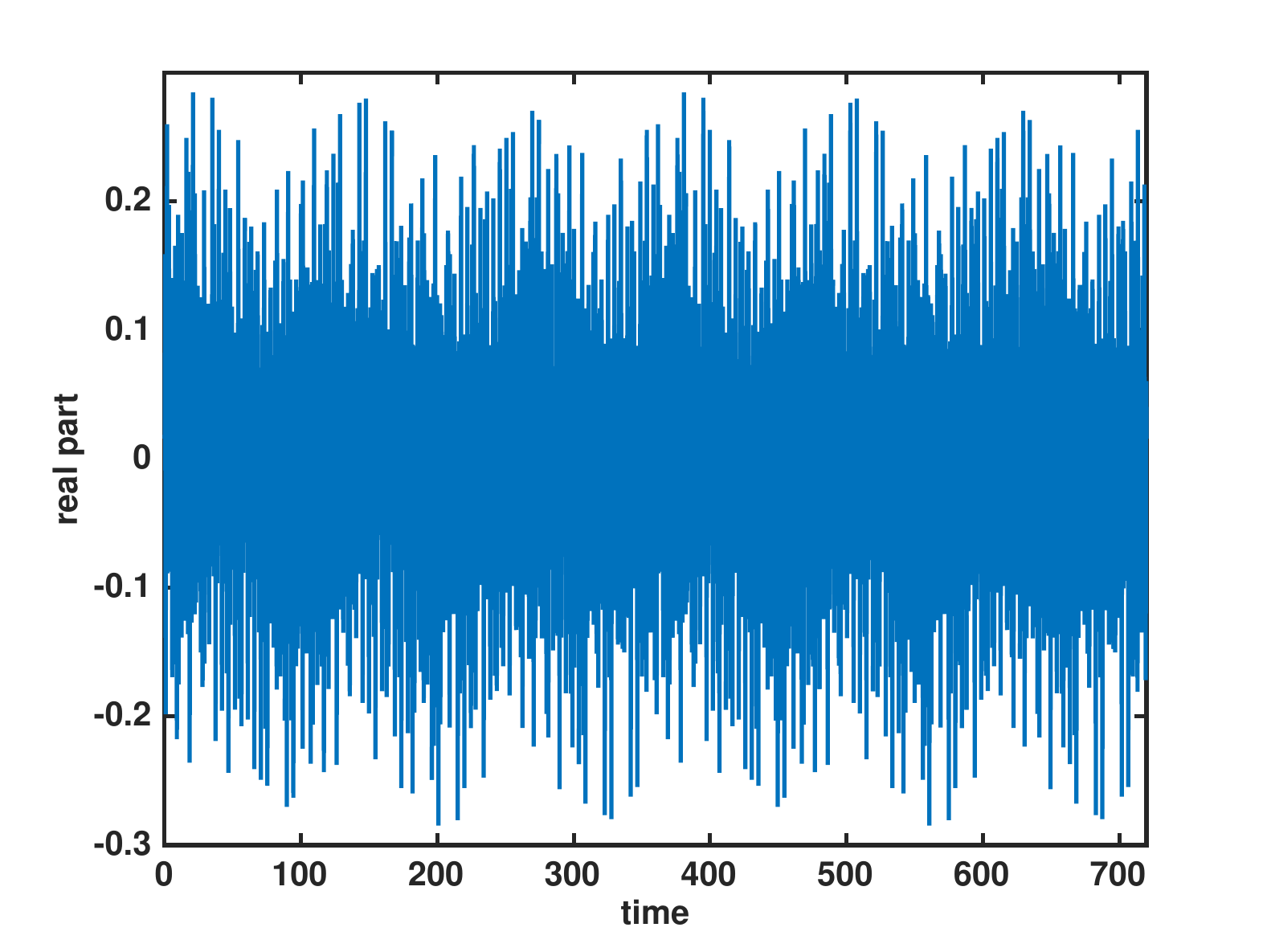}}\\
\subfigure[]{\includegraphics[scale=0.45]{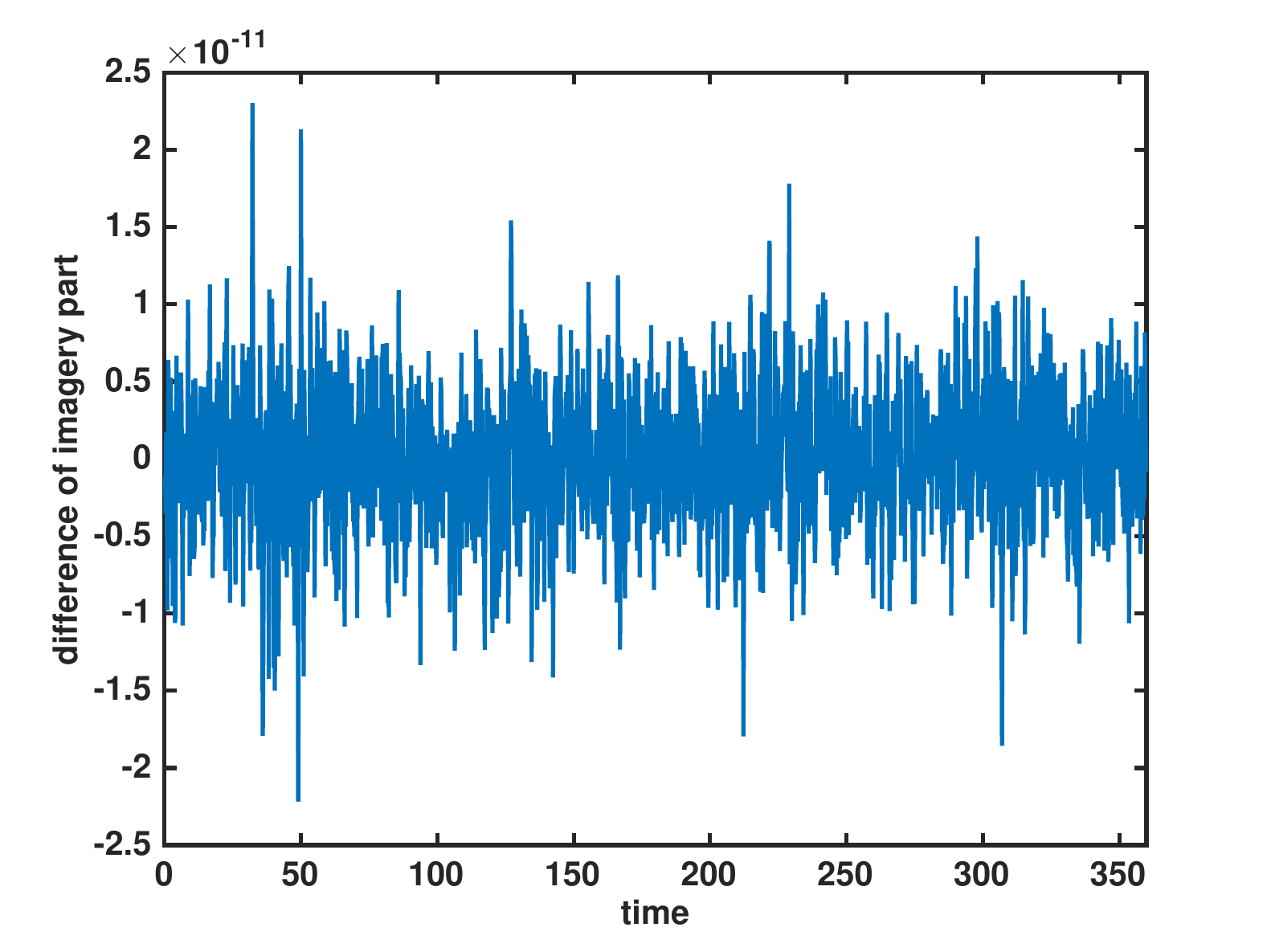}}
\subfigure[]{\includegraphics[scale=0.45]{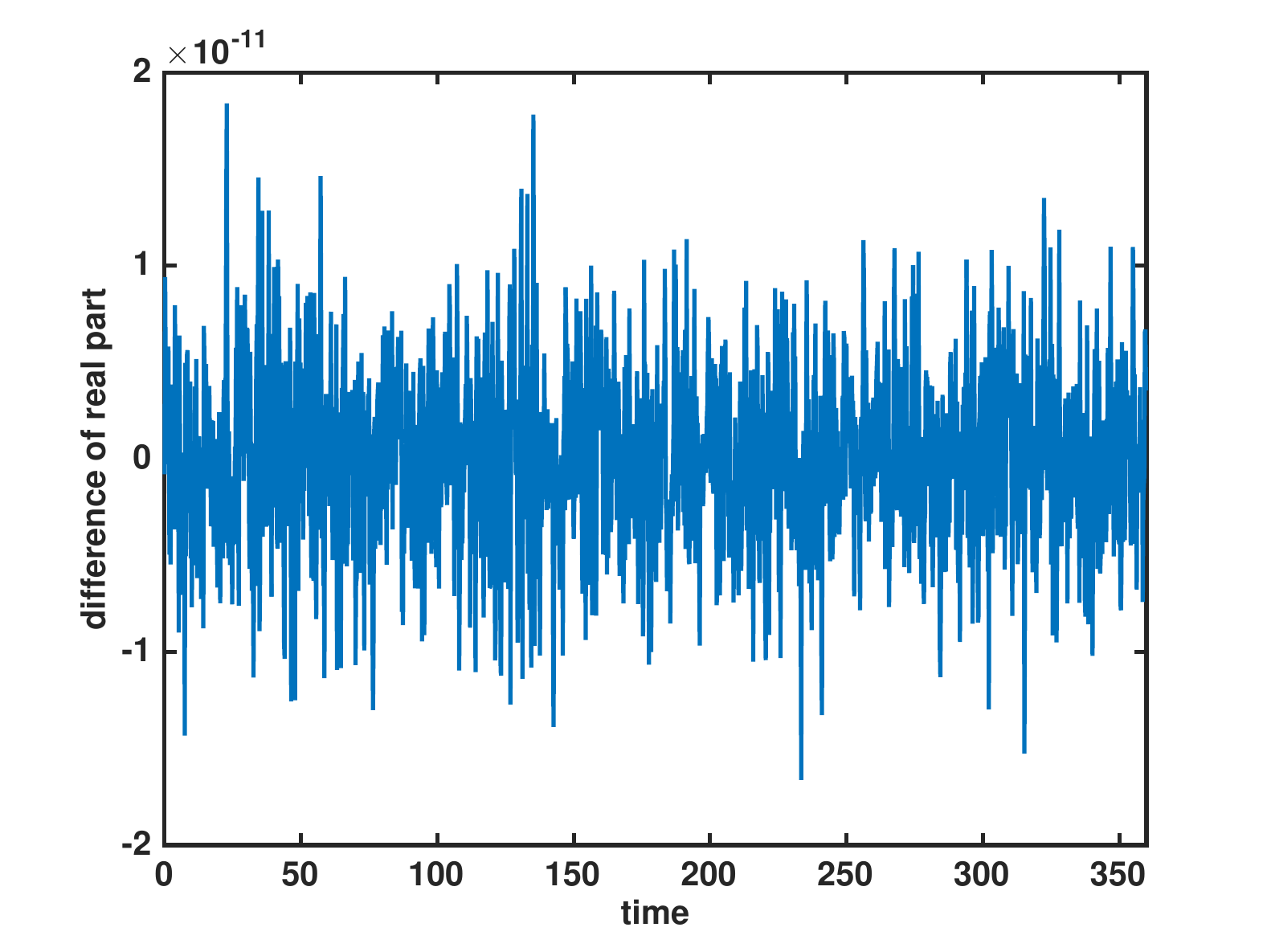}}
\caption{{ ESQM}: (a) Time evolution of imaginary part of $\psi(t, 0,0,0)$. (b) Time evolution of real part of $\psi(t, 0,0,0)$. (c) The difference $\psi([T, 2T], 0,0,0)-\psi([0, T], 0,0,0)$, $T=360$ (imaginary part). (d) The difference $\psi([T, 2T], 0,0,0)-\psi([0, T], 0,0,0)$, $T=360$ (real part).}\label{fig:MS2}
\end{figure}

\begin{figure}[htbp]
\subfigure[]{\includegraphics[scale=0.45]{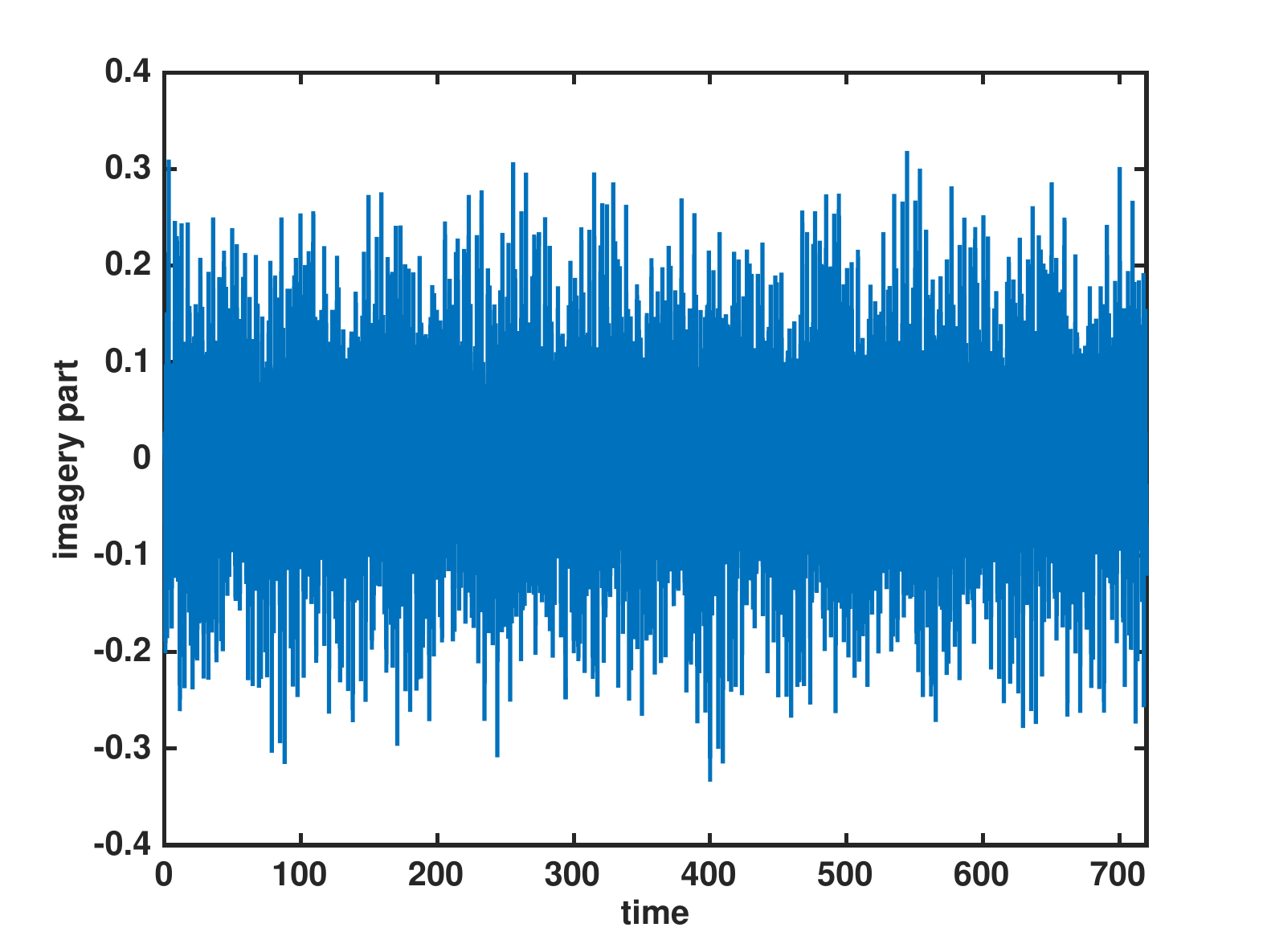}}
\subfigure[]{\includegraphics[scale=0.45]{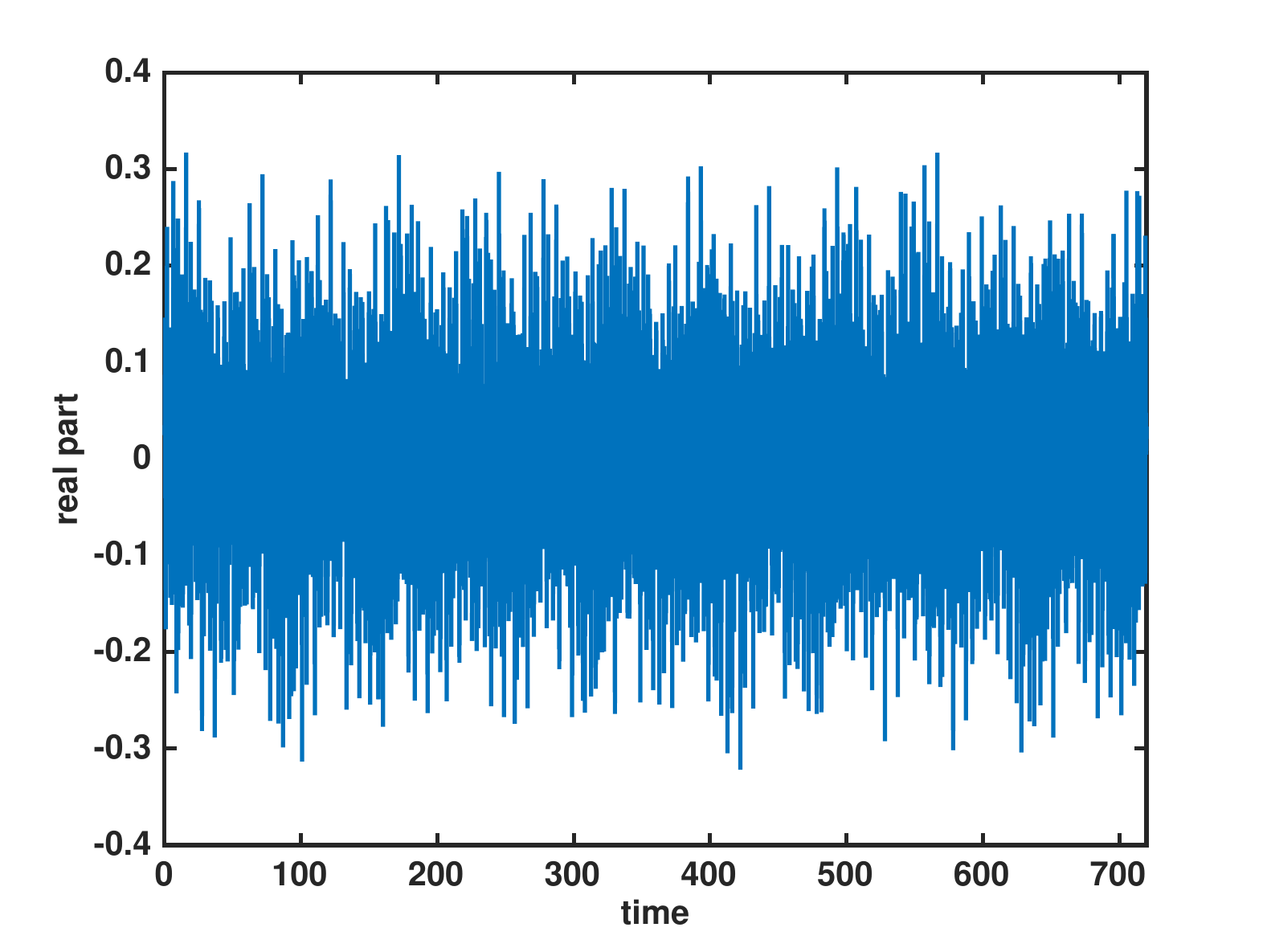}}\\
\subfigure[]{\includegraphics[scale=0.45]{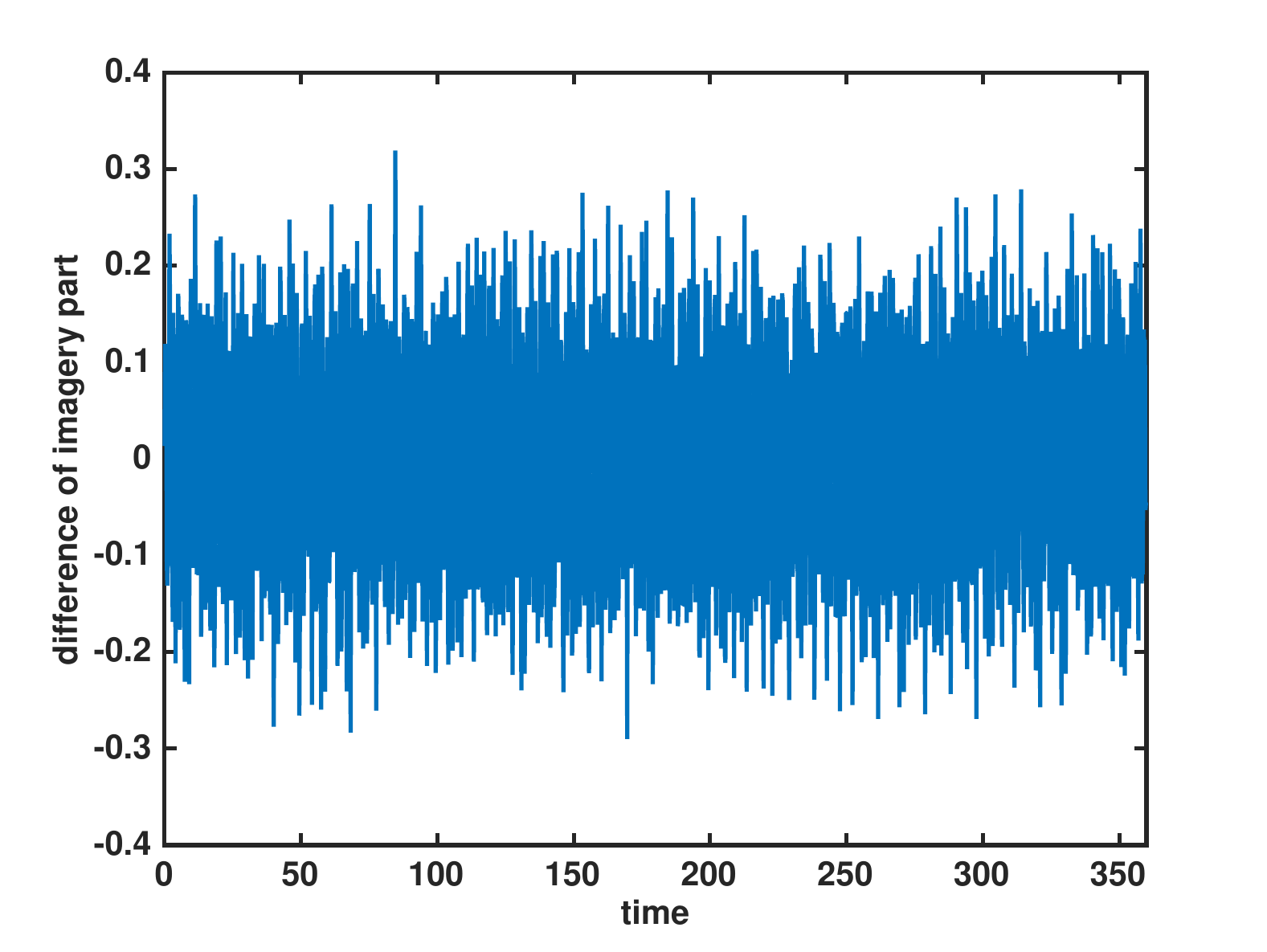}}
\subfigure[]{\includegraphics[scale=0.45]{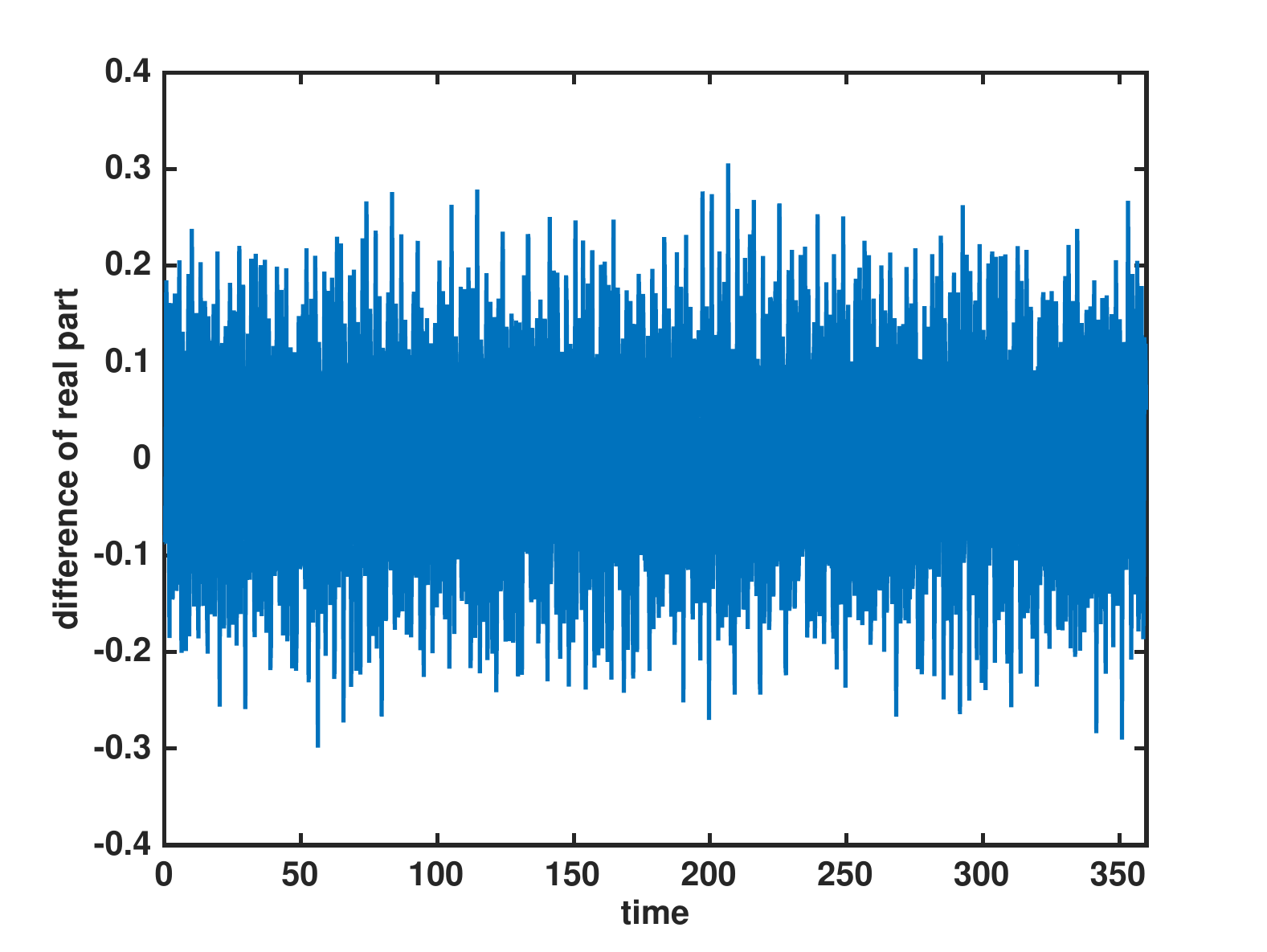}}
\caption{{ Strang}: (a) Time evolution of imaginary part of $\psi(t, 0,0,0)$. (b) Time evolution of real part of $\psi(t, 0,0,0)$. (c) The difference $\psi([T, 2T], 0,0,0)-\psi([0, T], 0,0,0)$, $T=360$ (imaginary part). (d) The difference $\psi([T, 2T], 0,0,0)-\psi([0, T], 0,0,0)$, $T=360$ (real part).}
\label{fig:strang360}
\end{figure}

\begin{figure}[htbp]
\subfigure[]{\includegraphics[scale=0.45]{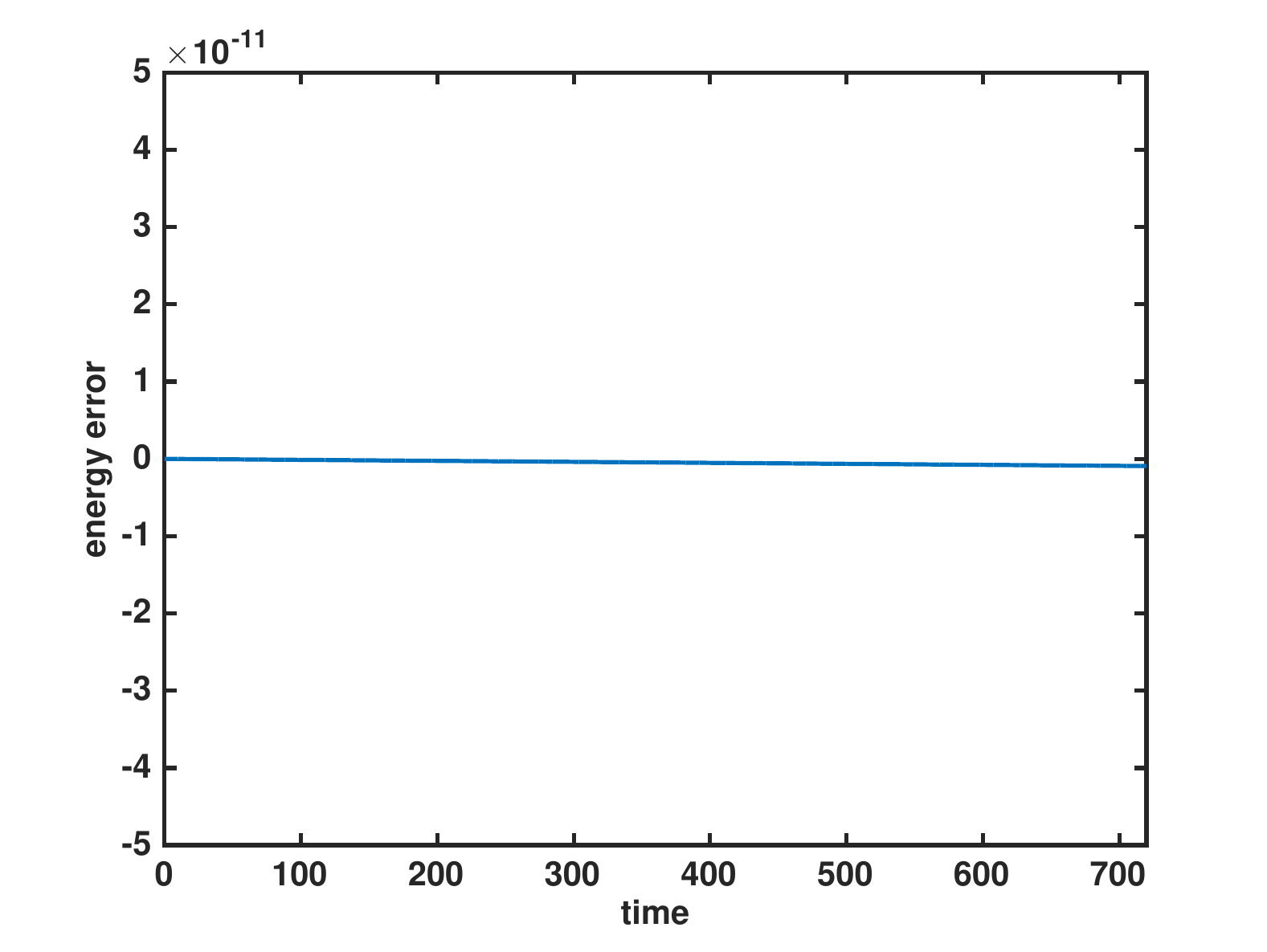}}
\subfigure[]{\includegraphics[scale=0.45]{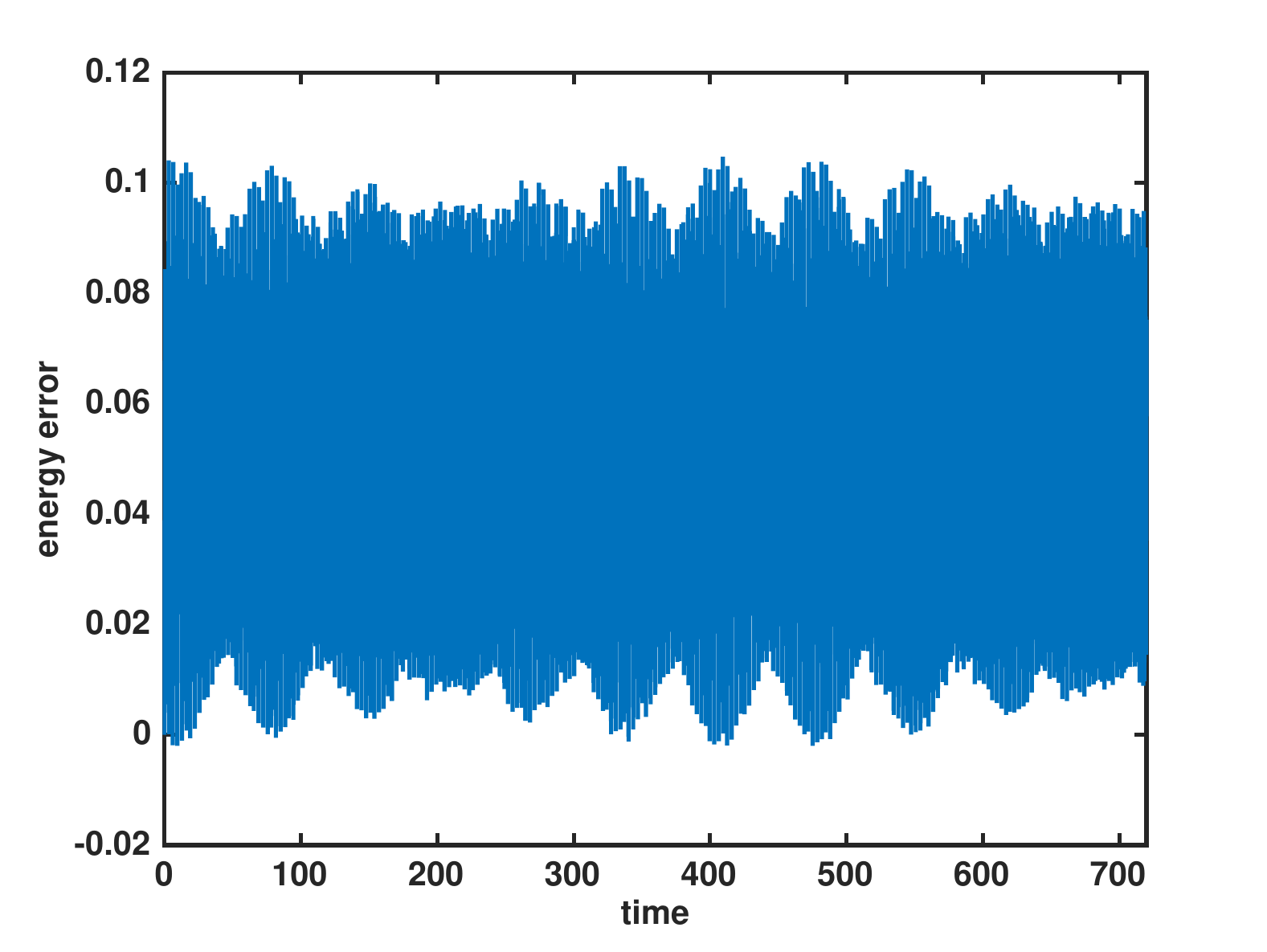}}
\caption{(a) Time evolution of the energy error by ESQM. (b) Time evolution of the energy error by Strang splitting.}\label{fig:error360}
\end{figure}

\begin{figure}[htbp]
\center{
\subfigure[]{\includegraphics[scale=0.45]{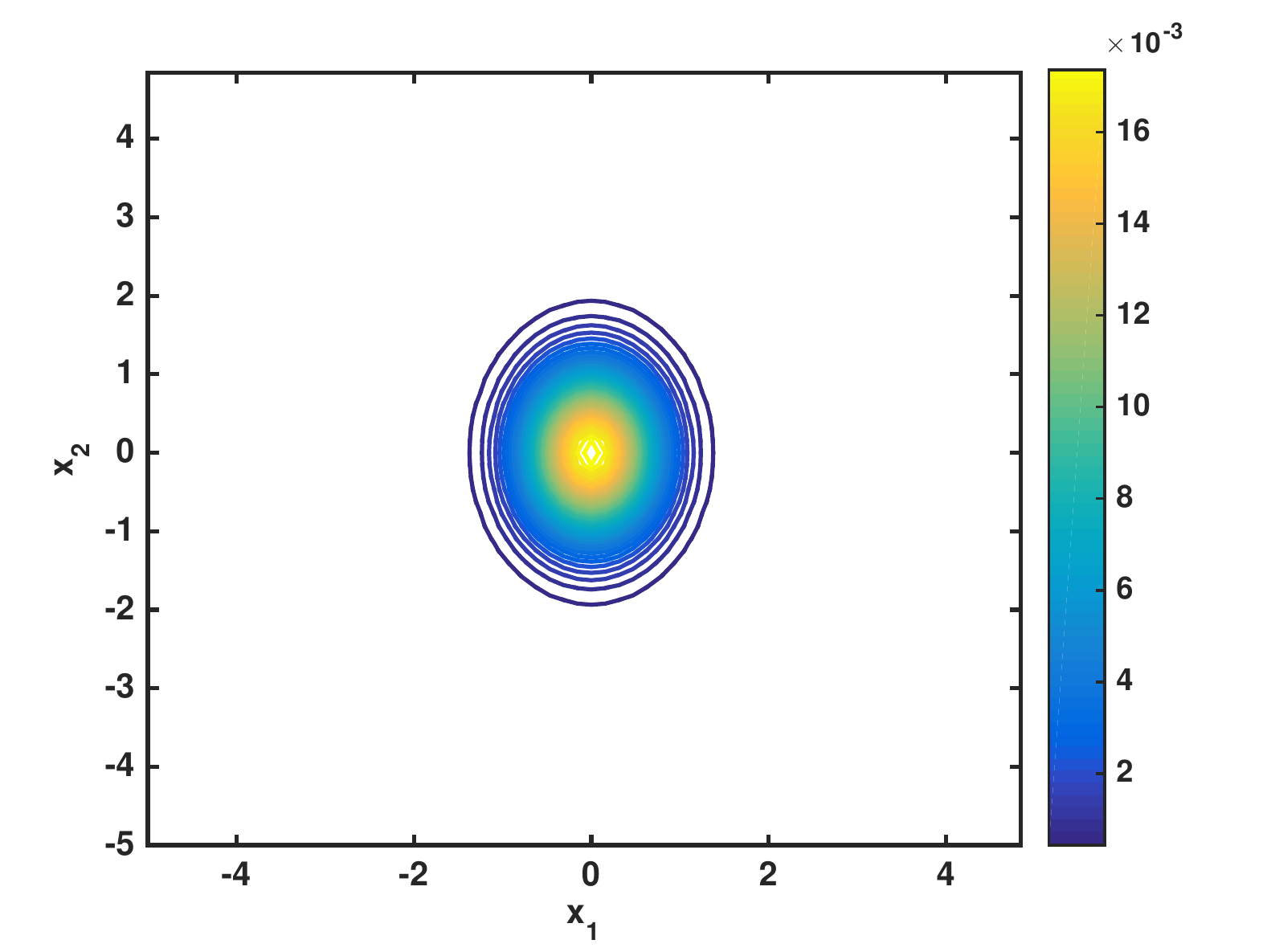}}\\
\subfigure[]{\includegraphics[scale=0.45]{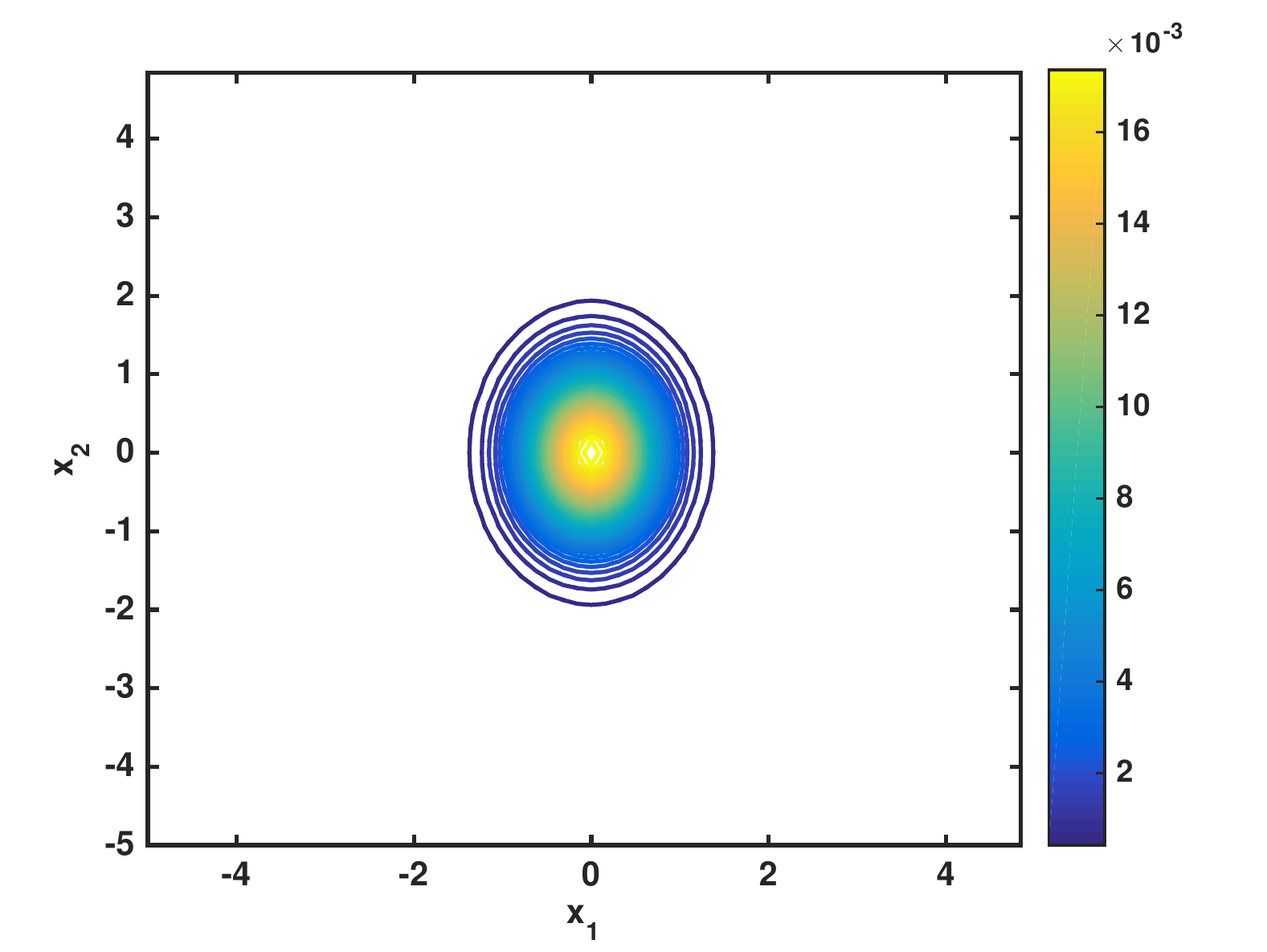}}
\subfigure[]{\includegraphics[scale=0.45]{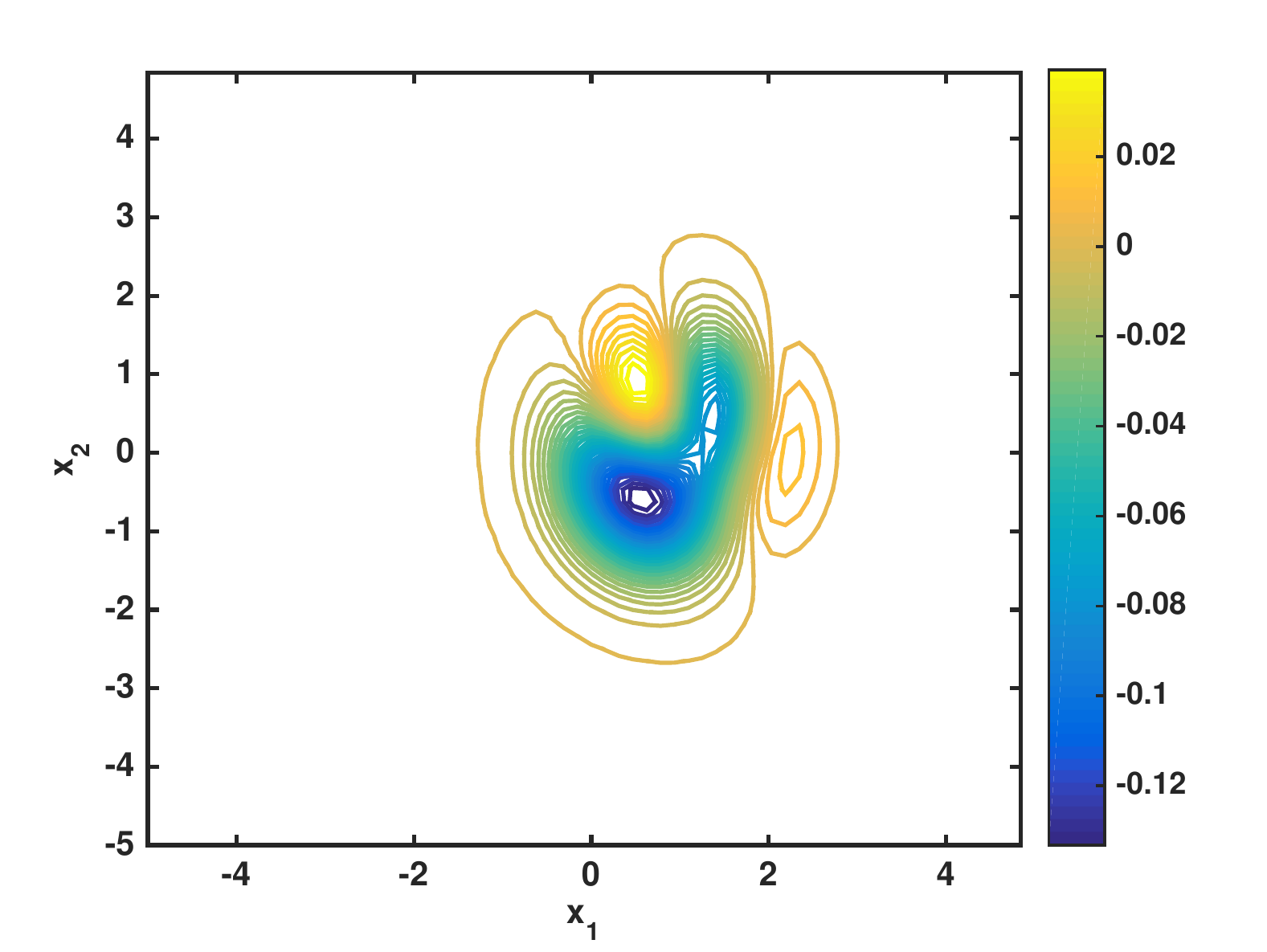}}}
\caption{(a) Initial contour plot of the real part of $\psi(t=0, x_1,x_2,0)$; (b) contour plot of real part of $\psi(t=360,x_1,x_2,0)$ by ESQM; (c) contour plot of real part of $\psi(t=360, x_1,x_2,0)$ by Strang. }\label{fig:contour360}
\end{figure}

%

\noindent{\bf 3D magnetic Schr\"odinger equation}\\
To end this part, the following 3D magnetic Schr\"odinger equation is considered (see \cite{Ostermann}),
\begin{equation}\label{eq:oster}
i \partial_t \psi({\bf x},t) = -\frac{1}{2}\Delta \psi({\bf x},t) + i {\bf A}({\bf x}) \cdot \nabla \psi({\bf x},t) + \frac{1}{2}|{\bf A}({\bf x})|^2\psi({\bf x},t) + V_{nq}({\bf x})\psi({\bf x},t),
\end{equation}
where ${\bf A}({\bf x}) = {\bf x} \times {\bf B}$, ${\bf B}=\transp{(1, 0.1, 2)}$, ${\bf x}=(x_1,x_2,x_3)$ and 
\begin{align} \label{eq:potential}
V_{nq}({\bf x}) =\alpha \left(20\cos(\frac{2\pi(x_1+5)}{10}) + 20\cos(\frac{2\pi(x_2+5)}{10}) + 20\cos(\frac{2\pi(x_3+5)}{10}) + 60\right), \alpha \in \mathbb{R}.
\end{align}
The initial condition is
$$
\psi_0({\bf x}) = \frac{2^{3/8}}{\pi^{3/2}} \exp\left(-\frac{\sqrt{2}}{2}((x_1-1)^2 + x_2^2 +x_3^2)\right),
$$
and the numerical parameters are: the spatial domain $[-5,5]^3$ is discretized by $N_1=N_2=N_3=64$ points  
and the final time is $t = 1$.  Here we consider three methods 
\begin{itemize}
\item ESQM (see \eqref{esqm2} with $f({\bf x}, |\psi|^2) = V_{nq}({\bf x})$ given by \eqref{eq:potential}  and with \eqref{eq:MS2}); this method is second order accurate in time. 
\item ESR (see Appendix \ref{3dmag}); this method is second order accurate in time. 
\item Strang (see  Appendix \ref{3dmag}); this method is second order accurate in time. 
\end{itemize}
The three methods are compared with different step sizes $\Delta t$ to solve the system \eqref{eq:oster}. 
The energy errors of these three methods are presented in Figure \ref{fig:3Derror}, 
by studying the influence of the parameter $\alpha$ which measures the amplitude of the non-quadratic part in \eqref{eq:oster}.  
By comparing the energy errors, we can see that the ESQM is the most accurate one, 
as it solves the linear quadratic part exactly.  Moreover, when $\alpha$ is smaller, i.e., the non-quadratic term in system~(\ref{eq:oster}) becomes smaller, we can see that the advantage of ESQM is more obvious. 


\begin{figure}[htbp]
\center{
\subfigure[]{\includegraphics[scale=0.35]{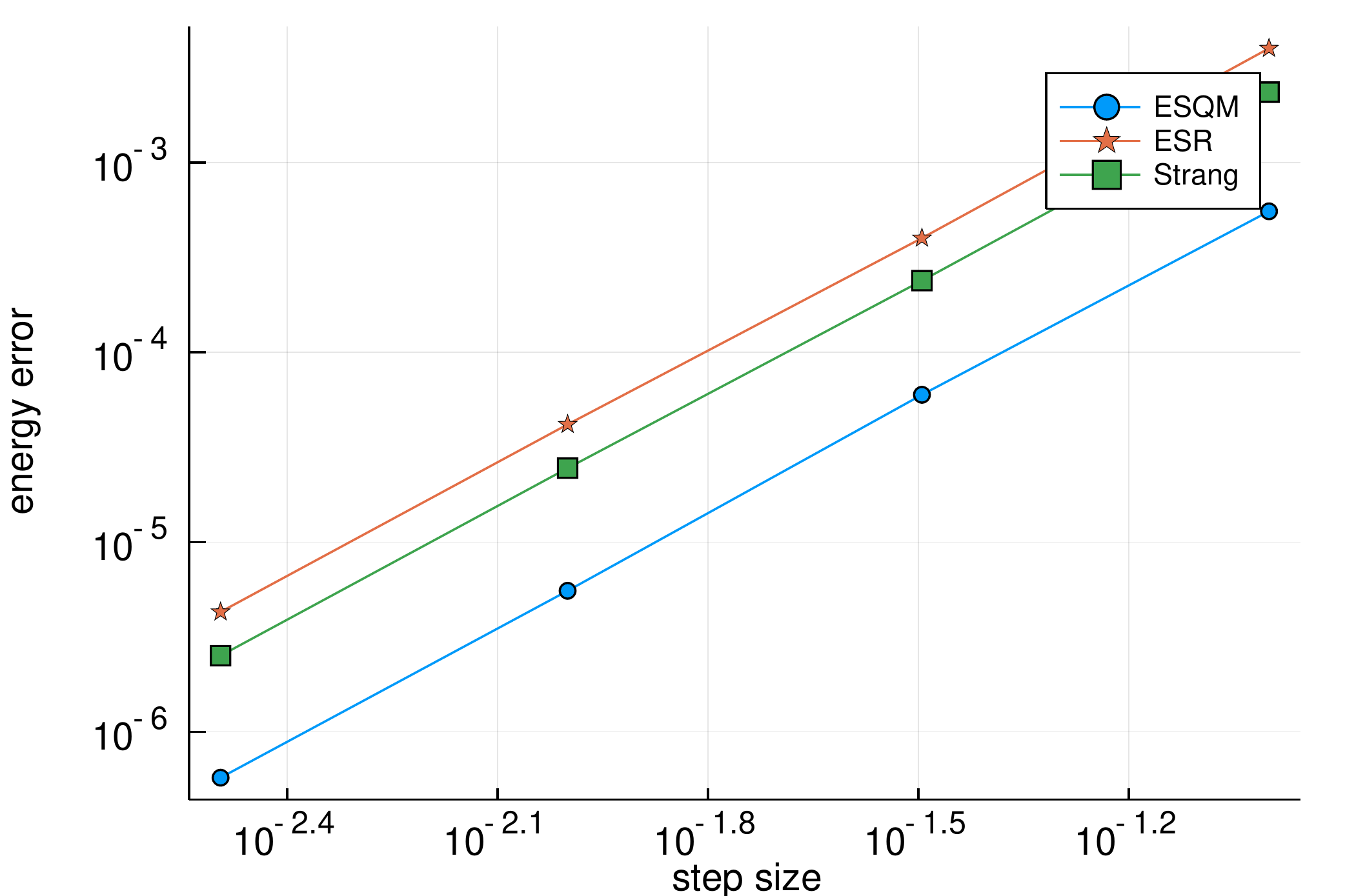}}
\subfigure[]{\includegraphics[scale=0.35]{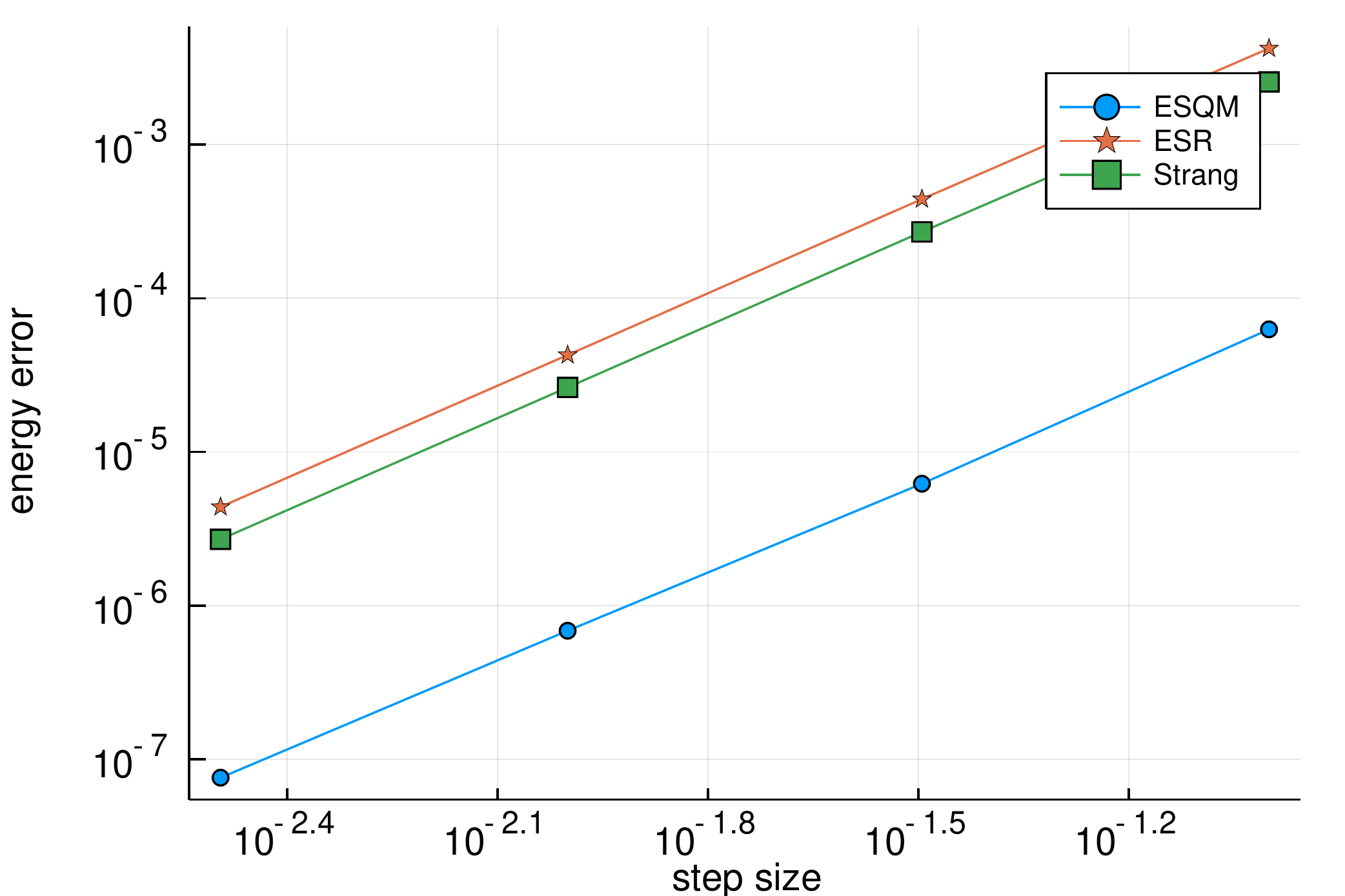}}
}
\caption{Plots of energy error with step size at $t=1$ with grids $N_1=N_2=N_3=64$. (a) $\alpha = 0.1$, (b) $\alpha = 0.01$.}\label{fig:3Derror}
\end{figure}



\newpage

\section{Appendix}
  

\subsection{2D magnetic Schr\"odinger equation}
\label{magschro_split}
\begin{equation}
i\epsilon \partial_t \psi({\bf x},t) = -\frac{\epsilon^2}{2} \Delta  \psi({\bf x},t) + i\epsilon {\bf A} \cdot \nabla \psi ({\bf x},t) + \frac{1}{2}{|\bf A|}^2\psi({\bf x},t),
\end{equation}
where ${\bf x} = (x_1,x_2) \in \mathbb{R}^2$, ${\bf A} = \frac{1}{2}(A_1, A_2)$, $A_1= -x_2$, $A_2 = x_1$. 
The above system can be split into three systems: 
\begin{align}
&i\epsilon \partial_t \psi({\bf x},t) = -\frac{\epsilon^2}{2} \Delta  \psi({\bf x},t),\\
& \partial_t \psi({\bf x},t) =  {\bf A} \cdot \nabla \psi ({\bf x},t), \\ 
&i\epsilon \partial_t \psi({\bf x},t) = \frac{1}{2}{|\bf A|}^2\psi({\bf x},t),
\end{align}
The solutions of the above three subsystems can be obtained by operators $e^{it\frac{\epsilon}{2} \Delta }$, $e^{t\text{Rot}}$, and $e^{tV}$ respectively.
Since the second is nothing but a 2D rotation, we call the associated solution $e^{t\text{Rot}}$. 
Then we have the following second order splitting method 
\begin{equation}\label{eq:ESRSTRANG}
\psi^{n+1} = e^{\frac{\Delta t}{2}V} e^{i \Delta t\frac{\epsilon}{4} \Delta } e^{\Delta t\text{Rot}} e^{i\Delta t\frac{\epsilon}{4} \Delta } e^{\frac{\Delta t}{2}V}, 
\end{equation}
from which we derive two variants according to the treatment of $e^{\Delta t\text{Rot}}$. Indeed,  
{\bf ESR} denotes the splitting method (\ref{eq:ESRSTRANG}) when $e^{\Delta t \text{Rot}}$ is solved by exact splittings for transport equation in Proposition \ref{prop_rot_gen}. {\bf Strang} denotes (\ref{eq:ESRSTRANG}) when $e^{\Delta t \text{Rot}}$ is approximated by Strang directional splitting.

\subsection{2D rotating Gross-Pitaevskii equation}
\label{GP_split}
The rotating Gross-Pitaevskii equation (GPE)~\cite{Bao, wang} is
\begin{equation}
\label{gpe_app}
\begin{aligned}
&{i}\partial_t \psi({\bf x}, t) = -\frac{1}{2}\Delta \psi({\bf x},t) + V({\bf x}) \psi({\bf x},t) + \beta | \psi |^2 \psi({\bf x},t) - \Omega L_{x_3} \psi({\bf x},t),\  {\mathbf x} \in {\mathbb R}^2,
\end{aligned}
\end{equation}
where $\psi({\bf x}, t)$ is the macroscopic wave function, ${\bf x} = (x_1, x_2)$, $L_{x_3} = - i (x_1\partial_{x_2} - x_2 \partial_{x_1})$. 
Two operator splittings are presented to approximate \eqref{gpe_app}.  

\subsubsection{ESR splitting} 
\label{esr_gpe}
The above equation \eqref{gpe_app} can be split into three parts which can be solved exactly in time 
\begin{align*}
&i \frac{\partial \psi({\bf x}, t)}{\partial t} = -\frac{1}{2}\Delta \psi({\bf x},t),\\
& \frac{\partial \psi({\bf x}, t)}{\partial t} =  \Omega L_{x_3} \psi({\bf x},t),\\
&i \frac{\partial \psi({\bf x}, t)}{\partial t} =  V({\bf x})\psi({\bf x},t) + \beta |\psi|^2 \psi({\bf x},t). 
\end{align*}
The solutions of the above three subsystems can be obtained by operators  $e^{-it\frac{1}{2}\Delta}$, $e^{t \text{Rot}}$, and $e^{t \text{VN}}$ respectively. 
Then we have the following second order method splitting method: 
\begin{equation}\label{eq:GPEES}
\psi^{n+1}({\mathbf x}) = e^{\frac{\Delta t}{2} \text{VN}} e^{-i\Delta t\frac{1}{4}\Delta} e^{\Delta t \text{Rot}} e^{-i\Delta t\frac{1}{4}\Delta}e^{\frac{\Delta t}{2} \text{VN}} \psi^n({\mathbf x}),
\end{equation} 
from which we derive two variants according to the treatment of $e^{\Delta t\text{Rot}}$ (the part $e^{\frac{\Delta t}{2} \text{VN}}$ 
can be solved exactly). As for magnetic Schr\"odinger case, 
{\bf ESR} denotes the splitting method (\ref{eq:GPEES}) when  $e^{\Delta t \text{Rot}}$ 
is solved by exact splittings for transport equation in Proposition \ref{prop_rot_gen}. 

\subsubsection{BW method}
\label{bw}
Here we recall the splitting method introduced in \cite{wang} to approximate \eqref{gpe_app}. 
We will call it BW in the sequel. 
BW splitting for rotating GPE (\ref{gpe_app}) is based on the following two-steps splitting 
\begin{align}
&i\partial_t \psi({\bf x}, t) = -\frac{1}{2}\Delta \psi({\bf x},  t)  - \Omega L_{x_3} \psi({\bf x}, t),\label{eq:first}\\
&\partial_t \psi({\bf x}, t) = V({\bf x})\psi({\bf x}, t) + \beta |\psi({\bf x}, t)|^2\psi({\bf x}, t). \label{eq:non}
\end{align}
Then, the authors in  \cite{wang} noticed  that (\ref{eq:first}) can be split further as 
\begin{align}
&i\partial_t \psi({\bf x}, t) = -\frac{1}{2}\partial^2_{x_1}\psi({\bf x}, t) -i\Omega x_2 \partial_{x_1} \psi({\bf x}, t), \label{eq:x}\\
&i\partial_t \psi({\bf x}, t) = -\frac{1}{2}\partial^2_{x_2}\psi({\bf x}, t) + i\Omega x_1 \partial_{x_2} \psi({\bf x}, t).\label{eq:y}
\end{align}
The solutions of subsystems (\ref{eq:non}), (\ref{eq:x}) and (\ref{eq:y}) can be obtained by operators $e^{tN}, e^{tX}$ and $e^{tY}$ respectively,  the second order BW method is then derived from the following composition 
\begin{equation}
\begin{aligned}\label{eq:strangBW}
\psi^n({\bf x}) &= \left( e^{\Delta t/2 \, Y}e^{\Delta t/2\,X}e^{\Delta t\, N}e^{\Delta t/2\,X}e^{\Delta t/2\, Y} \right)^n \psi_0({\bf x}),\\
& =  e^{\Delta t/2\, Y} (e^{\Delta t/2\, X}e^{\Delta t\, N}e^{\Delta t/2\, X}e^{\Delta t\, Y})^{n-1}e^{\Delta t/2\, X}e^{\Delta t\, N}e^{\Delta t/2\, X}e^{\Delta t/2\, Y} \psi_0({\bf x}).
\end{aligned}
\end{equation}
Combined with Fourier pseudo-spectral method in space, we can see that in each time step, we need six calls to FFT. 

\subsection{3D time-periodic quadratic linear  Schr\"odinger equation}
For (\ref{eq:sch}) with $f=0$ and $B$ and $V$ are specified in (\ref{eq:3601}) and (\ref{eq:3602}), 
we consider two numerical methods: ESQM and a standard Strang operator splitting. 

\subsubsection{Exact splitting}
\label{sec:3D coefficients}
The coefficients for ESQM \eqref{eq:MS2} are given by 
\begin{align*}
&A_{\Delta t} \simeq \begin{pmatrix} 0.503369336514750 &0.09260872887966 &-0.086577853155386   \\
0.092608728879667 & 0.499175997238123&0.090475411725230   \\
 -0.086577853155386 &0.090475411725230  &0.482430618251455 
 \end{pmatrix}, \\ 
 &V_{\Delta t}^{(\ell)} \simeq \begin{pmatrix} 1.838313777101704 & 0& 0  \\
0 &1.405233579215994  &0 \\
  0 &0  &2.416160688906186
 \end{pmatrix}, \\ 
 &V_{\Delta t}^{(r)} \simeq \begin{pmatrix} 0.765638127548775 &0.097739062052903 & -0.244124321719139  \\
0.097739062052903 &1.408683914880933 &  0.141925135897144 \\
-0.244124321719139 &0.14192513589714 &  3.535113753227984
 \end{pmatrix}, \\
 &L_{\Delta t} \simeq \begin{pmatrix} 0 & 0&0   \\
0.957867410476376 &0 & 0  \\
 -0.917880413070041  &1.133563918623215 & 0  
 \end{pmatrix}, \\
& U_{\Delta t} \simeq \begin{pmatrix} 0 & -1.132325985517193&  0.915677911046419 \\
0 & 0 &  -0.957661219232001 \\
0 & 0& 0
 \end{pmatrix}.
\end{align*}

\subsubsection{Strang method}
\label{3dtime_split}
Classically, we use the following operator splitting 
\begin{align*}
&i \frac{\partial \psi({\bf x}, t)}{\partial t} = -\frac{1}{2}\Delta \psi({\bf x},t),\\
& \frac{\partial \psi({\bf x}, t)}{\partial t} =-  ({ B}  {\bf x})\cdot \nabla \psi({\bf x},t),\\
&i \frac{\partial \psi({\bf x}, t)}{\partial t} =  V({\bf x})\psi({\bf x},t).
\end{align*}
The solutions of the above three subsystems can be obtained by operators $e^{-it\frac{1}{2}\Delta}$, $e^{t \text{Rot}}$, and $e^{-it \text{V}}$ respectively so that we have the following second order splitting method
\begin{equation}\label{eq:360ES}
\psi^{n+1}({\mathbf x}) = e^{-i\frac{\Delta t}{2} \text{V}} e^{-i\Delta t\frac{1}{4}\Delta} e^{\Delta t \text{Rot}} e^{-i\Delta t\frac{1}{4}\Delta}e^{-i\frac{\Delta t}{2} \text{V}} \psi^n({\mathbf x}).
\end{equation}
{\bf Strang} denotes (\ref{eq:360ES}) when $e^{\Delta t \text{Rot}}$ is also approximated by a Strang directional splitting.

\subsection{3D magnetic Schr\"odinger equation}
\label{3dmag}
From (\ref{eq:oster}),  
where ${\bf A}({\bf x}) = {\bf x} \times {\bf B}$, ${\bf B}=\transp{(1, 0.1, 2)}$ and $V$ given by \eqref{eq:potential},  
we can use  the following operator splitting 
\begin{align*}
&i \frac{\partial \psi({\bf x}, t)}{\partial t} = -\frac{1}{2}\Delta \psi({\bf x},t),\\
& \frac{\partial \psi({\bf x}, t)}{\partial t} = {\mathbf  A}  ({\bf x})\cdot \nabla \psi({\bf x},t), \\ 
&i \frac{\partial \psi({\bf x}, t)}{\partial t} =   \frac{1}{2}|{\bf A}({\bf x})|^2\psi({\bf x},t) + V({\bf x})\psi({\bf x},t),
\end{align*}
The solutions of the above three subsystems can be obtained by operators
$e^{-it\frac{1}{2}\Delta}$, $e^{t \text{Rot}}$, and $e^{t \text{VA}}$ respectively  
and we can derive a second order splitting method:
\begin{equation}\label{eq:3DmagneticES}
\psi^{n+1}({\mathbf x}) = e^{\frac{\Delta t}{2} \text{VA}} e^{-i\Delta t\frac{1}{4}\Delta} e^{\Delta t \text{Rot}} e^{-i\Delta t\frac{1}{4}\Delta}e^{\frac{\Delta t}{2} \text{VA}} \psi^n({\mathbf x}).
\end{equation}
{\bf ESR} denotes the splitting method (\ref{eq:3DmagneticES}) when  $e^{\Delta t \text{Rot}}$ is solved by exact splittings for transport equation in Proposition \ref{prop_rot_gen}. {\bf Strang} denotes (\ref{eq:3DmagneticES}) when $e^{\Delta t \text{Rot}}$ is approximated by Strang directional splitting. 

The coefficients when $\Delta t = 0.1$ for ESQM~\eqref{esqm2} are as follows 
\begin{align*}
&A_{\Delta t} \simeq \begin{pmatrix} 0.506160069704187 &0.098840554692409&0.001683128724191   \\
0.098840554692409 & 0.508317718832811&0.050167780151672   \\
 0.001683128724191 &0.050167780151672  &0.501715861437068 
 \end{pmatrix}, \\ 
 &V_{\Delta t}^{(\ell)} \simeq \begin{pmatrix} 2.025343613765655 & 0& 0  \\
0 &0.508168767491105  &0 \\
  0 &0  &0.000099459606977
 \end{pmatrix}, \\ 
 &V_{\Delta t}^{(r)} \simeq \begin{pmatrix} 0.072891278447532 &0.242556937819776 & -1.026420948565178  \\
0.242556937819776 &1.959142247295385 &  -0.046535665904951 \\
-1.026420948565178  &-0.046535665904951 &  0.508102737430800
 \end{pmatrix}, \\
 &L_{\Delta t} \simeq \begin{pmatrix} 0 & 0&0   \\
2.003434507092443 &0 & 0  \\
-0.099043028107977  &1.016569585390557& 0  
 \end{pmatrix}, \\
& U_{\Delta t} \simeq \begin{pmatrix} 0 & -1.963756896350695&  -0.099988990937417 \\
0 & 0 & -1.006635420674690 \\
0 & 0& 0
 \end{pmatrix}.
\end{align*}

\subsection{Proof of the period 360}
\label{appendixc}
\begin{lemma}
\label{lemma_per}
The function $t\mapsto U_t = e^{ i t ( \Delta/2 - V(x)) - t Bx \cdot \nabla }$, where $V$ and $B$ are given by \eqref{eq:3601} satisfies 
$$
\forall t\in \mathbb{R}, \ U_{t+180} = -U_t.
$$
\end{lemma}
\begin{proof} Since $t\mapsto U_t$ is a group, we just have to prove that 
$$U_{180}=-I_{L^2(\mathbb{R}^3)}.$$
We recall that by construction, we have $ U_t = e^{-t q_{\eqref{eq_QMS}}^w}$ where
$$
q_{\eqref{eq_QMS}}^w = i\frac{|\boldsymbol{\xi}|^2}2   + iB {\mathbf{x}} \cdot {\boldsymbol \xi} + iV({\mathbf{x}})
$$

\noindent \emph{Step 1: To conjugate $q_{\eqref{eq_QMS}}^w$ to a sum of harmonic oscillators.} We are going to prove that there exists $V\in \mathcal{U}(L^2(\mathbb{R}^n))$ such that
\begin{equation}
\label{its_the_harm_osci}
U_t = V \exp(-it \sum_{j=1}^3 \omega_j  (x_j^2 -\partial_{x_j}^2) ) V^*
\end{equation}
where $(\omega_1,\omega_2,\omega_3) = \frac{\pi}{180}(20,75,132)$. Assuming first this decomposition, we deduce that
$$
U_{180} = V \exp(-20 i \pi  (x_1^2 -\partial_{x_1}^2) ) \exp(-75 i \pi    (x_2^2 -\partial_{x_2}^2) ) \exp(-132 i \pi   (x_3^2 -\partial_{x_3}^2) )V^*.
$$
But, in dimension $1$, the eigenvalues of the harmonic oscillator $x^2-\partial_x^2$ being the odd positive integers, we know that $\exp(i\pi (x^2-\partial_x^2) ) = -I_{L^2(\mathbb{R}^3)}$. Thus, we deduce that 
$$
U_{180} = V I_{L^2(\mathbb{R}^3)} (-I_{L^2(\mathbb{R}^3)}) I_{L^2(\mathbb{R}^3)}  V^* = -I_{L^2(\mathbb{R}^3)}.
$$
 
 In order to prove \eqref{its_the_harm_osci} we are going to apply the following theorem due to H\"ormander.
 \begin{theorem}[H\"ormander, Theorem 21.5.3 in \cite{Lars_book}] 
\label{lars_you_are_the_best}
Let $Q\in S_{2n}^{++}(\mathbb{R})$ be a real symmetric positive matrix of size $2n$. There exists a real symplectic matrix $P\in \mathrm{Sp}_{2n}(\mathbb{R})$ of size $2n$ such that and some positive numbers $\omega_1,\dots,\omega_n$ such that
$$
\transp{P}QP = D(\omega)
$$
where 
$D(\omega) = \mathrm{diag}(\omega_1,\dots,\omega_n,\omega_1,\dots,\omega_n)$ is the diagonal matrix such that, for $j=1,\dots,n$, $D(\omega)_{j,j} = D(\omega)_{j+n,j+n} = \omega_j$ .
\end{theorem}
Indeed, here, it can be checked that $Q_{\eqref{eq_QMS}}$ (the matrix of $q_{\eqref{eq_QMS}}$) is a symmetric positive matrix (computing, for example, an approximation of its eigenvalues). Thus, applying Theorem \ref{lars_you_are_the_best}, we get a symplectic matrix $P$ and some positive 
numbers $\omega_1<\omega_2<\omega_3$ such that
\begin{equation}
\label{useful}
\transp{P}Q_{\eqref{eq_QMS}}P = D(\omega).
\end{equation}
Consequently, since $P$ is symplectic, we have
$$
\exp(2tJQ_{\eqref{eq_QMS}}) =  P \exp(2tJD(\omega)) P^{-1},
$$
where $J$ is the symplectic matrix of $\mathbb{R}^{2n}$.
Now, applying the monoid morphism (Theorem 3.1 in~\cite{essiqo}) introduced also by H\"ormander in \cite{Lars}, we get a function $t\mapsto \sigma_t\in \{\pm 1\}$ such that
$$
\forall t\in \mathbb{R}, \ U_t = e^{-itq_{\eqref{eq_QMS}}^w} = \sigma_t V \exp(-it \sum_{j=1}^3 \omega_j  (x_j^2 -\partial_{x_j}^2) ) V^*
$$
where $\pm V$ is the Fourier Integral Operator associated with $P$. Note that $V$ is unitary. Furthermore, by a straighforward argument of continuity we deduce that $\sigma_t=1$ for all $t\in \mathbb{R}$. Thus, to conclude, we just have to prove that $(\omega_1,\omega_2,\omega_3) = \frac{\pi}{180}(20,75,132)$.

\noindent \emph{Step 2: To determine $\omega$.} First, we observe that the matrices $JQ_{\eqref{eq_QMS}}$ and $JD(\omega)$ are similar. Indeed, since $P\in \mathrm{Sp}_6(\mathbb{R})$, we have $\transp{P}\in \mathrm{Sp}_6(\mathbb{R})$ and applying \eqref{useful} we deduce that
$$
JD(\omega) = J\transp{P}Q_{\eqref{eq_QMS}}P = (P^{-1}J\transp{P}^{-1}) \transp{P}Q_{\eqref{eq_QMS}}P = P^{-1}JQ_{\eqref{eq_QMS}}P.
$$
A fortiori, $JQ_{\eqref{eq_QMS}}$ and $JD(\omega)$ have the same eigenvalues. Thus, the eigenvalues of $JQ_{\eqref{eq_QMS}}$ are
\begin{equation}
\label{hehe}
\sigma(JQ_{\eqref{eq_QMS}}) = \sigma(JD(\omega)) = \{ i \omega_1,-i \omega_1,i\omega_2,-i\omega_2,i\omega_3,-i\omega_3\}.
\end{equation}
Consequently, to determine $\omega$ we just have to determine the roots of the characteristic polynomial of $JQ_{\eqref{eq_QMS}}$, denoted $\chi_{\eqref{eq_QMS}}$. By a straightforward calculation, we observe that
\begin{multline*}
\left(\frac{3}\pi\right)^{6} \chi_{\eqref{eq_QMS}}(\frac{\pi X}{3}) = X^6 + \frac{\lambda_1+\lambda_2+\lambda_3+3}2 X^4 + \frac{\lambda_1\lambda_2+ \lambda_1\lambda_3+\lambda_2\lambda_3+9/4}4 X^2 \\-  3\frac{\lambda_1+\lambda_2+\lambda_3}{32} + \frac{\lambda_1\lambda_2+ \lambda_1\lambda_3+\lambda_2\lambda_3}8 - \frac{\lambda_1\lambda_2\lambda_3}{8}.
\end{multline*}
But, by construction $\lambda_1  < \lambda_2 < \lambda_3$ are the roots of the polynomial
$$
7200 X^3 - 72196 X^2 + 222088 X - 216341.
$$ 
Thus, $\lambda_1+\lambda_2+\lambda_3$, $\lambda_1\lambda_2+ \lambda_1\lambda_3+\lambda_2\lambda_3$ and $\lambda_1\lambda_2\lambda_3$ are some explicit rational numbers and we deduce that
$$
\left(\frac{3}\pi\right)^{6} \chi_{\eqref{eq_QMS}}(\frac{\pi X}{3}) = X^6+\frac{407}{120}X^4+\frac{123}{80}X^2-\frac7{384}.
$$
Finally, we verify by an explicit computation that
$$
\chi_{\eqref{eq_QMS}}( i\frac{\pi}9 )= \chi_{\eqref{eq_QMS}}( i\frac{5\pi}{12} ) =\chi_{\eqref{eq_QMS}}(i\frac{11 \pi}{15})=0.
$$
So, we deduce of \eqref{hehe} that $(\omega_1,\omega_2,\omega_3) = \frac{\pi}{180}(20,75,132)$.
\end{proof}



\begin{thebibliography}{00}

\bibitem{AB}
{\sc P. Alphonse, J. Bernier}, 
\textit{Polar decomposition of semigroups generated by non-selfadjoint quadratic differential operators and regularizing effects}, 
hal-02280971, 2019.  

\bibitem{Ameres}
{\sc J. Ameres}, 
{\em Splitting methods for Fourier spectral discretizations of the strongly magnetized Vlasov-Poisson and the Vlasov-Maxwell system}, 
arXiv preprint arXiv:1907.05319, 2019.

\bibitem{review_3}
{\sc X. Antoine, W. Bao, C. Besse}, 
{\em Computational methods for the dynamics of the nonlinear Schr\"odinger/Gross-Pitaevskii equations}, 
Computer Physics Communications, 2013, {\bf 184} (12), pp.  2621-2633? 

\bibitem{Xavier}
{\sc X. Antoine, R. Duboscq}, 
{\em GPELab, a Matlab toolbox to solve Gross-Pitaevskii equations I: Computation of stationary solutions}, 
Computer Physics Communications, 2014, {\bf 185} (11), pp. 2969-2991.

\bibitem{Xavier2}
{\sc X. Antoine, R. Duboscq}, 
{\em Gpelab, a matlab toolbox to solve Gross-Pitaevskii equations II: Dynamics and stochastic simulations}, 
Computer Physics Communications, 2015, {\bf 193}, pp. 95-117.


\bibitem{Bao}
{\sc W. Bao, Q. Du, Y. Zhang},  
{\em Dynamics of rotating Bose-Einstein condensates and its efficient and accurate numerical computation},  
SIAM Journal on Applied Mathematics, 2006, {\bf 66} (3), pp. 758-786.

\bibitem{wang} 
{\sc W. Bao, H. Wang},  
{\em An efficient and spectrally accurate numerical method for computing dynamics of rotating Bose-Einstein condensates},  Journal of Computational Physics, 2006, {\bf 217} (2), pp. 612-626.

\bibitem{essiqo}
{\sc J. Bernier}, 
{\em Exact splitting methods for semigroups generated by inhomogeneous quadratic differential operators}, 
preprint. 

\bibitem{JC2}
{\sc J. Bernier, F. Casas, N. Crouseilles}, 
\textit{Splitting methods for rotations: application to Vlasov equations}, 
accepted in SIAM J. of Scientific Comput. 

\bibitem{Besse}
{\sc C. Besse, G. Dujardin, I. Lacroix-Violet}, 
{\em High order exponential integrators for nonlinear Schr\"odinger equations with application to rotating Bose-Einstein condensates},  
SIAM Journal on Numerical Analysis, 2017, {\bf 55} (3), pp. 1387-1411.

\bibitem{highorder}
{\sc N. Besse, M. Mehrenberger}, 
{\em Convergence of classes of high-order semi-Lagrangian schemes for the Vlasov-Poisson system}, 
Mathematics of computation, 2008, {\bf 77} (261), pp. 93-123.



\bibitem{Ostermann} 
{\sc M. Caliari, A. Ostermann, C. Piazzola}, 
{\em A splitting approach for the magnetic Schr\"odinger equation}, 
Journal of Computational and Applied Mathematics, 2017, {\bf 316}, pp. 74-85.



\bibitem{shear} 
{\sc B. Chen, A. Kaufman}, 
{\em 3D volume rotation using shear transformations}, 
Graphical Models, 2000, {\bf 62} (4), pp.  308-322.

\bibitem{coulaud}
{\sc O. Coulaud, E. Sonnendr\"ucker, E. Dillon, P. Bertrand}, 
J. Plasma Physics, 1999, {\bf 61}, pp. 435-448. 

\bibitem{exponential}
{\sc G. Dujardin, F. H\'erau, P. Lafitte}, 
{\em Coercivity, hypocoercivity, exponential time decay and simulations for discrete Fokker-Planck equations}, 
arXiv preprint arXiv:1802.02173, 2018.

\bibitem{HLW}
{\sc E. Hairer, C. Lubich, G. Wanner}, 
{Geometric numerical integration: Structure-Preserving Algorithms for Ordinary Differential Equations}, 
Springer Series in Computational Mathematics, 2006. 



\bibitem{herau}
{\sc F. H\'erau, J. Sj\"ostrand, M. Hitrik}, 
{\em Tunnel effect for the Kramers-Fokker-Planck type operators: return to equilibrium and applications}, 
Int. Math. Res. Not., 2008, {\bf 57}, p. 48. 

\bibitem{herau2}
{\sc F. H\'erau, J. Sj\"ostrand, M. Hitrik}, 
{\em Tunnel effect for the Kramers-Fokker-Planck type operators}, 
Ann. Henri Poincar\'e, 2008, {\bf 9}, pp. 209-274. 

\bibitem{herau3}
{\sc F. H\'erau, L. Thomann}, 
{\em On global existence and trend to the equilibrium for the Vlasov-Poisson-Fokker-Planck system with exterior confining potential},  
J. Funct. Anal, 2016, {\bf 271}, pp. 1301-1340. 

\bibitem{HO}
{\sc M. Hochbruck, A. Ostermann},  
{\em Exponential integrators}, 
Acta Numerica, 2010, {\bf19}, pp. 209-286.

\bibitem{Lars}
{\sc L. H\"ormander}, 
\textit{Symplectic classification of quadratic forms, and general Mehler formulas}, 
Math. Z.,  (1995), {\bf 219}, pp. 413-449. 

\bibitem{Lars_book}
{\sc L. H\"ormander}, 
{The analysis of linear partial differential operators. {III}}, Classics in Mathematics, Springer, Berlin, 2007.
Pseudo-differential operators, https://doi.org/10.1007/978-3-540-49938-1

\bibitem{Jin} 
{\sc S. Jin, Z. Zhou}, 
{\em A semi-Lagrangian time splitting method for the Schr\"odinger equation with vector potentials},  
Communications in Information and Systems, 2013, {\bf13} (3), pp. 247-289.

\bibitem{LH}
{\sc Y. Li, Y. He, Y. Sun, et al.},  
{\em Solving the Vlasov-Maxwell equations using Hamiltonian splitting}, 
Journal of Computational Physics, 2019, {\bf 396}, pp. 381-399

\bibitem{MQ}
{\sc R.I. McLachlan, G.R. Quispel}, 
{\em Splitting methods}, 
Acta Numerica, 2002, {\bf 11}, pp. 341-434.

\bibitem{symmetry}
{\sc J.E. Marsden, T.S.Ratiu}, 
{Introduction to mechanics and symmetry: a basic exposition of classical mechanical systems}, 
Springer Science \& Business Media, 2013.



\bibitem{raymond}
{\sc N. Raymond}, 
{Bound States of the Magnetic Schr\"odinger Operator}, 
EMS Tracts in Mathematics, 2017. 

\bibitem{constants}
{\sc J.S. Welling, W.F. Eddy, T.K. Young}, 
{\em Rotation of 3D volumes by Fourier-interpolated shears}, 
Graphical Models, 2006, {\bf 68} (4), pp. 356-370.

\bibitem{zhang}
{\sc R. Zeng, Y. Zhang}, 
{\em Efficiently computing vortex lattices in rapid rotating Bose-Einstein condensates}, 
Computer Physics Communications, 2009, {\bf 180} (6), pp.  854-860.
\end{thebibliography}
\end{document}